\documentclass{aart}

\usepackage{titlesec}
\titleformat{\subsection}[runin]{\normalfont\bfseries}{\thesubsection.}{.5em}{}[.]\titlespacing{\subsection}{0pt}{2ex plus .1ex minus .2ex}{.8em}
\titleformat{\subsubsection}[runin]{\normalfont\itshape}{\thesubsubsection.}{.3em}{}[.]\titlespacing{\subsubsection}{0pt}{1ex plus .1ex minus .2ex}{.5em}

\usepackage[labelfont=sc,font=small,labelsep=period]{caption}
\usepackage[letterpaper, hmargin=1in, top=1.5in, bottom=1.9in, footskip=0.7in]{geometry}

\usepackage{amsmath} 
\usepackage{amssymb}
\usepackage{amsfonts}
\usepackage{latexsym}
\usepackage{amsthm}
\usepackage{amsxtra}
\usepackage{amscd}
\usepackage{bbm}
\usepackage{mathrsfs}
\usepackage{bm}

\usepackage{cite}

\usepackage{graphicx, color}

\newcommand{\f}[1]{\bm{\mathrm{#1}}} 
\newcommand{\bb}{\mathbb} 
\newcommand{\rr}{\mathrm} 
\renewcommand{\cal}{\mathcal} 
 
\newcommand{\fra}{\mathfrak} 
\newcommand{\ol}[1]{\overline{#1} \!\,} 
\newcommand{\wh}{\widehat}
\newcommand{\wt}{\widetilde}

\newcommand{\ii}{\mathrm{i}}
\newcommand{\dd}{\mathrm{d}}
\newcommand{\col}{\mathrel{\vcenter{\baselineskip0.75ex \lineskiplimit0pt \hbox{.}\hbox{.}}}}
\newcommand*{\deq}{\mathrel{\vcenter{\baselineskip0.65ex \lineskiplimit0pt \hbox{.}\hbox{.}}}=}

\renewcommand{\le}{\leqslant}
\renewcommand{\leq}{\leqslant}
\renewcommand{\ge}{\geqslant}
\renewcommand{\geq}{\geqslant}

\newcommand{\ad}[1]{#1^*}

\newcommand{\ind}[1]{\f 1 (#1)}
\newcommand{\indb}[1]{\f 1 \pb{#1}}

\renewcommand{\epsilon}{\varepsilon}

\renewcommand{\P}{\mathbb{P}}
\newcommand{\E}{\mathbb{E}}
\newcommand{\R}{\mathbb{R}}
\newcommand{\C}{\mathbb{C}}
\newcommand{\N}{\mathbb{N}}
\newcommand{\Z}{\mathbb{Z}}

\newcommand{\p}[1]{({#1})}
\newcommand{\pb}[1]{\bigl({#1}\bigr)}
\newcommand{\pB}[1]{\Bigl({#1}\Bigr)}
\newcommand{\pbb}[1]{\biggl({#1}\biggr)}
\newcommand{\pBB}[1]{\Biggl({#1}\Biggr)}

\newcommand{\qB}[1]{\Bigl[{#1}\Bigr]}
\newcommand{\qbb}[1]{\biggl[{#1}\biggr]}
\newcommand{\qBB}[1]{\Biggl[{#1}\Biggr]}

\newcommand{\h}[1]{\{{#1}\}}
\newcommand{\hb}[1]{\bigl\{{#1}\bigr\}}

\newcommand{\hbb}[1]{\biggl\{{#1}\biggr\}}

\newcommand{\abs}[1]{\lvert #1 \rvert}
\newcommand{\absb}[1]{\bigl\lvert #1 \bigr\rvert}
\newcommand{\absB}[1]{\Bigl\lvert #1 \Bigr\rvert}
\newcommand{\absbb}[1]{\biggl\lvert #1 \biggr\rvert}

\newcommand{\norm}[1]{\lVert #1 \rVert}
\newcommand{\normb}[1]{\bigl\lVert #1 \bigr\rVert}

\newcommand{\normbb}[1]{\biggl\lVert #1 \biggr\rVert}

\DeclareMathOperator{\tr}{Tr}

\DeclareMathOperator{\re}{Re}
\DeclareMathOperator{\im}{Im}

\theoremstyle{plain} 
\newtheorem{theorem}{Theorem}[section]
\newtheorem*{theorem*}{Theorem}
\newtheorem{lemma}[theorem]{Lemma}
\newtheorem*{lemma*}{Lemma}

\newtheorem*{corollary*}{Corollary}
\newtheorem{proposition}[theorem]{Proposition}
\newtheorem*{proposition*}{Proposition}

\theoremstyle{definition} 
\newtheorem{definition}[theorem]{Definition}
\newtheorem*{definition*}{Definition}
\newtheorem{example}[theorem]{Example}
\newtheorem*{example*}{Example}
\newtheorem{remark}[theorem]{Remark}

\newtheorem*{remark*}{Remark}
\newtheorem*{remarks*}{Remarks}

\newtheorem*{convention*}{Convention}

\newcommand{\beqa}{\begin{eqnarray}}
\newcommand{\eeqa}{\end{eqnarray}}
\newcommand{\e}{\varepsilon}

\newcommand{\ba}{{\bf{a}}}

\newcommand{\al}{\alpha}

\newcommand{\be}{\begin{equation}}
\newcommand{\ee}{\end{equation}}

\newcommand{\cG}{{\mathcal G}}

\newcommand{\ov}{\overline}

\numberwithin{equation}{section}
\numberwithin{theorem}{section}
\numberwithin{figure}{section}

\title{Averaging Fluctuations in Resolvents of Random Band Matrices}
\author{
L\'aszl\'o Erd\H os${}^1$\thanks{Partially supported
by SFB-TR 12 Grant of the German Research Council} \quad
Antti Knowles${}^2$\thanks{Partially supported by NSF grant DMS-0757425} \quad
Horng-Tzer Yau${}^2$\thanks{Partially supported
by NSF grants DMS-0757425, 0804279}  
 \\\\\\
Institute of Mathematics, University of Munich, \\
Theresienstrasse 39, D-80333 Munich, Germany \\ lerdos@math.lmu.de ${}^1$ \\ \\
Department of Mathematics, Harvard University\\
Cambridge MA 02138, USA \\ knowles@math.harvard.edu, \quad htyau@math.harvard.edu ${}^2$  \\  \\
}

\begin{document}

\maketitle

\begin{abstract} 
We consider a general class of 
 random matrices whose entries are centred random variables, independent up to a symmetry constraint.
We establish precise high-probability bounds on the averages of arbitrary monomials in the resolvent matrix entries. Our results generalize the previous results of \cite{EYY2, EYY3, EKYY1} which constituted a key step in the proof of the local semicircle law with optimal error bound in mean-field random matrix models.
Our bounds apply to random band matrices, and improve previous estimates from order 2 to order 4 in the cases relevant for applications. In particular, they lead to a proof of the diffusion approximation for the magnitude of the resolvent of random band matrices. This, in turn, implies new delocalization bounds on the eigenvectors.
The applications are presented in a separate paper \cite{EKYY3}.
\end{abstract}

\vspace{1cm}

\medskip


\newpage

\section{Introduction}

Let $H= (h_{ij})$ be a complex Hermitian or real symmetric $N\times N$ random matrix with
centred matrix entries that are independent
up to the symmetry constraint. We assume that the variances
$s_{ij}\deq \E |h_{ij}|^2$ are normalized so that $\sum_j s_{ij}=1$ for each $i$,
and let $\norm{s}_\infty \deq \max_{ij} s_{ij}$ denote the maximal variance.
Let $G_{ab}(z) \deq (H-z)^{-1}_{ab}$ denote the resolvent matrix entries evaluated at a spectral parameter $z=E+i\eta$ whose imaginary part $\eta$ is positive and small.
It was established in \cite{EYY1, EKYY4} that 
\be
\Lambda \;\deq\; \max_{a \neq b}|G_{ab}| \;\lesssim\; \sqrt{\frac{\norm{s}_{\infty}}{\eta}}
\label{Gmax}
\ee
with high probability for large $N$, up to factors of $N^\epsilon$.

The matrix entries  $G_{ab} \equiv G_{ab}(z)$ depend strongly on the entries of the $a$-th and $b$-th columns of $H$, but weakly on the other columns. Focusing on the dependence on $a$ only, this can be seen from the simple expansion formula 
\be
   G_{ab} \;=\; - G_{aa}\sum_{i\ne a} h_{ai} G_{ib}^{(a)}\,,
\label{GAB}
\ee
where $G^{(a)}$ denotes the resolvent of the $(N-1)\times (N-1)$ minor
of $H$ obtained by removing the $a$-th row and column
(see Lemma~\ref{lemma: res id} below for the general statement).
 Since $G^{(a)}$ is independent of the family $(h_{ai})_{i = 1}^N$,
the formula \eqref{GAB} expresses $G_{ab}$ as a sum of independent centred
random variables (neglecting the prefactor $G_{aa}$ which still depends on $(h_{ai})_{i = 1}^N$).
Therefore the size of $G_{ab}$ is governed by a fluctuation averaging mechanism,
similar to the central limit theorem. This is the main reason why the bound \eqref{Gmax} is substantially better than the naive estimate $|G_{ab}|\le \eta^{-1}$.

In this paper we investigate a more subtle phenomenon. To take a simple example, we are interested in
averages of resolvent matrix entries of the form
\be\label{ave1}
   \frac{1}{N}\sum_{a}G_{ab}
\ee
or, more generally, its weighted version
\be\label{ave2}
   \sum_{a} s_{\mu a} G_{ab}\,,
\ee
where $\mu$ and $b$ are fixed. We aim to show that, with high probability,
 these averages are of order $\Lambda^2$ -- much smaller  than the naive bound $\Lambda$ which results from an application of \eqref{Gmax} to each summand (we shall always work in the regime where $\Lambda\ll1$). The mechanism behind this improved bound is that for $a\ne a'$ the
matrix entries $G_{ab}$ and $G_{a'b}$ are only weakly correlated. To see this, note that,
since $h_{ai}$ in \eqref{GAB} and $h_{a'i}$ 
in the analogous formula
$$
   G_{a'b} \;=\; - G_{a'a'}\sum_{i'\ne a'} h_{a'i'} G_{i'b}^{(a')}\,,
$$
are independent, the correlation between $G_{ab}$ and $G_{a'b}$
primarily comes from  correlations between $h_{ai}$ and $ G_{i'b}^{(a')}$
and between $h_{a'i'}$ and $ G_{ib}^{(a)}$. (As above, here we neglect the less important prefactors $G_{aa}$ and $G_{a'a'}$.) Now $ G_{i'b}^{(a')}$
depends only weakly on $h_{ai}$ unless some lower indices coincide: $i=i'$ or $i=b$ or $i'=a$. Such coincidences are atypical, however, and consequently give rise to lower-order terms.
Once the smallness of the correlation between  $G_{ab}$ and $G_{a'b}$ is established,
the variance of the averages \eqref{ave1} or \eqref{ave2} can be estimated.
The smallness of the higher-order correlations between different resolvent matrix entries 
allows one to compute high moments and turn the variance bound into a high-probability bound. 
However, keeping track of all weak correlations among a large product of expressions of
 the form \eqref{ave1} with different $a$'s
is  rather involved, and we shall need to develop a graphical representation to do this effectively.

This idea of exploiting the weak dependence among different resolvent entries
  of random matrices first appeared in 
\cite{EYY2} and was subsequently used in \cite{EYY3, EKYY1, PY}. Such estimates
provide optimal error bounds in 
 the {\it local semicircle law} --
 a basic ingredient in establishing
 the universality of local statistics of for Wigner matrices.

Our main result in this paper 
estimates with high probability (weighted) averages of general monomials in the resolvent matrix
entries and their complex conjugates, where the averaging is performed on a subset of the indices. 
A more complicated example is
\be\label{aver3}
   \sum_{a,b}s_{\mu a}s_{\nu b} \pB{ |G_{ab}|^2 |G_{a \rho}|^2 - \E_{ab}
   |G_{ab}|^2 |G_{a \rho}|^2}\,,
\ee
where $\mu$, $\nu$, and $\rho$ are fixed. Here we subtract from each summand its partial expectation $\E_{ab}$ with respect to
the random variables in the $a$-th and $b$-th columns of $H$.
(Note that we could also have subtracted $\E_a G_{ab}$ in \eqref{ave1} and \eqref{ave2}
as well, but this expectation turns out to be negligible, unlike the expectations
of the manifestly positive quantity $|G_{ab}|^2 |G_{a \rho}|^2$ in \eqref{aver3}). 

The expression \eqref{aver3} can trivially be
estimated by $\Lambda^4$ with high probability using the estimate \eqref{Gmax}
on each summand (neglecting that diagonal resolvent matrix entries $G_{aa}$
require a different estimate). However, we can in fact do better: the averaging over two indices gives rise to a cancellation of fluctuations, due to the weak correlations among the summands.
Since each averaging independently yields an extra factor $\Lambda$ as in \eqref{ave1} and \eqref{ave2}, it seems plausible that the naive estimate of order $\Lambda^4$ on \eqref{aver3} can be improved to $\Lambda^6$. 
This in fact turns out to be correct in the example \eqref{aver3}, but in general the principle that each averaging yields one extra $\Lambda$ factor is not optimal.
Depending on the structure of the monomial, the gain may be more than a single factor $\Lambda$ per averaged index.
For example, averaging in the index $a$ in the quantities
\be
  \text{(I)} \;\deq\; \sum_a  s_{\mu a} \pb{ G_{ba}G_{ab}^* - \E_a G_{ba}G_{ab}^*}
  \qquad \text{and} \qquad  \text{(II)} \;\deq\; \sum_a  s_{\mu a}\pb{G_{ba}G_{ab} - \E_a G_{ba}G_{ab}}
\label{34}
\ee
has different effects. The naive estimate using \eqref{Gmax} yields $\Lambda^2$
for both quantities, but (I) is in fact  of order $\Lambda^4$ while
the (II) is only of order $\Lambda^3$ (all estimates are understood
with high probability).

The reason behind the gain of a factor $\Lambda^2$ over the naive size in case of (I) is quite subtle.
  We already mentioned that the dependence
of $G_{ab}$ on the random variables in the $c$-th column is weak if $c \ne a,b$. 
This is manifested in the identity
\be\label{Gabc}
  G_{ab} \;=\; G_{ab}^{(c)} + \frac{G_{ac}G_{cb}}{G_{cc}}\,.
\ee
(This identity first appeared in \cite{EYY1}; see Lemma \ref{lemma: res id}
below for a precise statement and related formulas.)
Since $G_{ab}^{(c)}$ is independent of the $c$-th column, the $c$-dependence
of $G_{ab}$ is contained in the second term of \eqref{Gabc}.
This term  is naively of order $\Lambda^2$, i.e.\ smaller than the main term
(accepting that $G_{cc}$ in the denominator is harmless; in fact it turns out to be
bounded from above and below by universal positive constants).
 Computing the variance of (I) results in a double sum $\sum_a \sum_c$. We shall see 
that, since the first term of \eqref{Gabc} is independent of $c$, the leading order 
contribution to the variance 
in fact comes from the second term. This yields an improvement of one $\Lambda$ over the naive
bound $\Lambda^2$. These ideas lead to a bound of order $\Lambda^3$ for both (I) and (II).
The idea of using averaging to improve a trivial bound on resolvent
entries by  an extra factor $\Lambda$ 
 was  central in \cite{EYY2}. In that paper this idea was applied
to a specific quantity analogous to
\begin{equation}\label{diagaver}
 \frac 1 N \sum_a  (1 - \E_a) \frac { 1} { G_{aa}}\,.
\end{equation}

When we compute a high  moment of
the quantities in \eqref{34}, we successively use  formulas \eqref{Gabc}
and \eqref{GAB} and take partial expectation in the expanded indices.
The result is the average of a high-order monomial of resolvent matrix entries.
Whether this averaging reduces the naive size depends
on the precise structure of the monomial. For example,
\be
 \sum_c s_{\mu c} G_{bc}G_{cb}^* \;=\;  \sum_c s_{\mu c} |G_{bc}|^2 = O(\Lambda^2)
\label{GGst}
\ee
and this estimate is optimal, while
\be
\sum_c s_{\mu c} G_{bc} G_{cb} \;=\; O(\Lambda^3).
\label{GGst1}
\ee
It turns out that average of the high-order monomial obtained from 
computing a high moment of (I) in \eqref{34} 
contains several summations of the type \eqref{GGst1}, while
the analogous formula for (II) contains only summations of the type \eqref{GGst}
(at least to leading order).
Whether the additional gain is present or not depends on the precise structure of the 
original monomial, in particular on how many times the averaging index appears in an entry of $G$ or $G^*$. In this regard the expressions (I) and (II) differ, which is the reason why their sizes differ.
Our main result (Theorem~\ref{theorem: Z lemma}) expresses the precise 
relation between the maximal gain and the structure of the monomial. 
As it turns out, this dependence is quite subtle. The main purpose of this paper is to give a systematic rule, applicable to arbitrary monomials
 in the resolvent entries, which determines the gain from all indices over which an average is taken.
 In particular, averaging over certain indices yields an improvement of order $\Lambda^2$; this is a novel phenomenon. This observation is crucial
 in the application of our results to the problem of quantum diffusion in random band matrices \cite{EKYY3}.

Finally, we shortly explain the improvement from the naive size $\Lambda^2$ to $\Lambda^3$ for the left-hand side of \eqref{GGst1}.
It follows from the estimate of order $\Lambda^3$ on (II) in \eqref{34} and from the fact that $\E_c  G_{ac}G_{cb} = O(\Lambda^3)$ for any $a,b$.
That the expectation $\E_c  G_{ac}G_{cb}$ itself is smaller than its naive size $\Lambda^2$ may be seen by expanding $G_{ac}G_{cb}$ in the index $c$ using formulas of the type \eqref{GAB}.
It turns out that $\E_c  G_{ac}G_{cb}$, viewed as a vector  indexed by $c$
and keeping $a$ and $b$ fixed, satisfies a stable self-consistent
vector equation (see \eqref{vself}). The analysis of this equation leads to the improved bound on $\E_c G_{ac}G_{cb}$ of order $\Lambda^3$.

Bounds on averages of resolvents of random matrices have played an essential role
in establishing 
the  local semicircle law with an optimal error bound.  We recall that 
in the simplest case of Wigner matrices, where $s_{ij}=N^{-1}$, the trace
of the resolvent
$$
    m_N(z) \;\deq\; \frac{1}{N}\tr G(z) \;=\; \frac{1}{N}\sum_a G_{aa}(z)
$$
is well approximated by the Stieltjes transform of the celebrated Wigner semicircle law
$$
   m(z) \;\deq\; \frac{1}{2\pi}\int_{-2}^2 \frac{\sqrt{4-x^2}}{x-z} \, \dd x \,.
$$
The optimal bound is
\begin{equation} \label{old zlemma}
   \abs{m(z)-m_N(z)} \;\lesssim\; \frac{1}{N\eta}
\end{equation}
with high probability (see \cite{EYY3} for the precise statement and the
history of this result). One of the main steps in proving this optimal
bound is to exploit that $G_{aa}$ and $G_{a'a'}$ are only  weakly correlated
for $a\ne a'$. Hence the average of $G_{aa}$ in $a$ in the definition of $m_N(z)$ 
fluctuates on a smaller scale than the fluctuations of $G_{aa}$. 
Various forms of this fluctuation averaging  were formulated in
\cite {EYY2, EYY3, EKYY1}. 
They were the key inputs to prove \eqref{old zlemma} and
its analogue for the sample covariance matrices in \cite{PY}.
In Proposition~\ref{prop: warm-up},  we present a simple special case of 
our main result,  Theorem~\ref{theorem: Z lemma}. This proposition
yields generalizations of estimates analogous to the previous fluctuation averaging bounds
with a more streamlined proof. A somewhat different simplification was given in \cite{PY}.

On the one hand, Theorem~\ref{theorem: Z lemma} is more general than its predecessors
since it is applicable to arbitrary monomials in $G$ and $G^*$, and also holds for universal
 Wigner matrices with nonconstant variances.
On the other hand, and more importantly, Theorem~\ref{theorem: Z lemma} gives
a stronger bound because  it exploits the additional cancellation effect
explained in connection with the different bounds on the two quantities
in \eqref{34}. This extra cancellation mechanism was not present
in \cite{EYY2, EYY3, EKYY1, PY}.

In a separate paper \cite{EKYY3} 
we apply the stronger bound
\be\label{y1}
    \sum_a  s_{\mu a} \pb{ |G_{ab}|^2 - \E_a |G_{ab}|^2} \;=\; O(\Lambda^4)
\ee
to derive a lower bound on the localization length of random band matrices. 
Extensions of the methods of \cite{EYY2, EYY3, EKYY1, PY} would have yielded only
\be
    \sum_a  s_{\mu a} \pb{ |G_{ab}|^2 - \E_a |G_{ab}|^2} \;=\; O(\Lambda^3)\,.
\label{weaker}
\ee
Had we had only \eqref{weaker} available in \cite{EKYY3}, the resulting
estimate on the localization length would not have improved the previously known results \cite{EK1, EK2} on eigenvector delocalization.

We conclude this section with a roadmap of the paper. In Section \ref{sec: setup} we define our main objects and introduce notation used throughout the paper. Our main result is Theorem~\ref{theorem: Z lemma} 
in Section~\ref{sec:multZ}. Before stating it in full generality, we first present a  special case, Proposition~\ref{prop: examples}, in Section~\ref{section: examples}. In order to motivate the concepts underlying Theorem~\ref{theorem: Z lemma},
we not only state this special case but also give a sketch of its proof, in Section~\ref{sect: sketch of proof of examples}. This is done before the main theorem is stated. A reader who prefers an inductive presentation should follow our sections in sequential
order. A reader who wants to jump quickly to the main result may skip
Section~\ref{sect: sketch of proof of examples}. However, some concepts
introduced in Section~\ref{sect: sketch of proof of examples} are needed
later in the proof (but not in the statement) of Theorem~\ref{theorem: Z lemma}.
The full proof of Theorem \ref{theorem: Z lemma} is presented in Sections~\ref{section: Z3}--\ref{section: simplifications},
following Section~\ref{section: Z1} where we give an outline of the proof and explain how Sections~\ref{section: Z3}--\ref{section: simplifications} are related.

\section{Setup} \label{sec: setup}
Let $(h_{ij} \col i \leq j)$ be a family of independent, complex-valued random variables $h_{ij} \equiv h_{ij}^{(N)}$ satisfying
$\E h_{ij} = 0$  and $h_{ii} \in \R$ for all $i$.
For $i > j$ we define $h_{ij} \deq \ol h_{ji}$, and denote by $H \equiv H_N = (h_{ij})_{i,j = 1}^N$ the $N \times N$ matrix with entries $h_{ij}$. By definition, $H$ is Hermitian: $H = H^*$.
We abbreviate
\begin{equation} \label{variance of h}
s_{ij} \;\deq\; \E \abs{h_{ij}}^2\,, \qquad M \;\equiv\; M_N \;\deq\; \frac{1}{\max_{i, j} s_{ij}}\,.
\end{equation}
In particular, we have the bound
\begin{equation} \label{s leq W}
s_{ij} \;\leq\; M^{-1}
\end{equation}
for all $i$ and $j$. We introduce the $N \times N$ symmetric matrix $S \equiv S_N = (s_{ij})_{i,j = 1}^N$. We assume that $S$ is (doubly) stochastic:
\begin{equation} \label{S is stochastic}
\sum_j s_{ij} \;=\; 1
\end{equation}
for all $i$. We shall always assume the bounds
\begin{equation} \label{lower bound on W}
N^\delta \;\leq\; M \;\leq\; N
\end{equation}
for some fixed $\delta > 0$.

\begin{example}[Band matrix] \label{example: band matrix}
Fix $d \in \N$. Let $f$ be a bounded and symmetric (i.e.\ $f(x) = f(-x)$) probability density on $\R^d$.
Let $L$ and $W$ be integers satisfying
\begin{equation*}
L^{\delta'} \;\leq\; W \;\leq\; L
\end{equation*}
for some fixed $\delta' > 0$. Define the $d$-dimensional discrete torus
\begin{equation*}
\bb T^d_L \;=\; [-L/2, L/2)^d \cap \Z^d\,.
\end{equation*}
Thus, $\bb T^d_L$ has $N = L^d$ lattice points; and we may identify $\bb T_L^d$ with $\{1, \dots, N\}$. We define the canonical representative of $i \in \Z^d$ through
\begin{equation*}
[i]_L \;\deq\; (i + L \Z^d) \cap \bb T^d_L\,.
\end{equation*}
Then $H$ is a \emph{$d$-dimensional band matrix} with band width $W$ and profile function $f$ if
\begin{equation*}
s_{ij} \;=\; \frac{1}{Z_{L}} \, f \pbb{\frac{[i - j]_L}{W}}\,.
\end{equation*}
It is not hard to see that $M = \pb{W^d + O(W^{d - 1})} / \norm{f}_\infty$ as $L \to \infty$.
The rows and columns of $H$ are thus indexed by the lattice points  in $\bb T^d_L$,
i.e.\ they are equipped with the geometry of $\Z^d$. For $d=1$, assuming that $f$ is compactly supported, the matrix entry $h_{ij}$ vanishes if
$|i-j|$ is larger than $CW$, i.e.\ $H$ is a band matrix in the traditional sense.
\end{example}

It is often convenient to use the normalized entries
\begin{equation*}
\zeta_{ij} \;\deq\; (s_{ij})^{-1/2} h_{ij}\,,
\end{equation*}
which satisfy $\E \zeta_{ij} = 0$ and $\E \abs{\zeta_{ij}}^2 = 1$. (If $s_{ij} = 0$ we set for convenience $\zeta_{ij}$ to be a normalized Gaussian, so that these relations remain valid. Of course in this case the law of $\zeta_{ij}$ is immaterial.) We assume that the random variables $\zeta_{ij}$ have finite moments, uniformly in $N$, $i$, and $j$, in the sense that for all $p \in \N$ there is a constant $\mu_p$ such that
\begin{equation} \label{finite moments}
\E \abs{\zeta_{ij}}^p \;\leq\; \mu_p
\end{equation}
for all $N$, $i$, and $j$. We make this assumption to streamline notation in the statements of results such as Theorem \ref{theorem: Z lemma} and the proofs. In fact, our results hold, with the same proof, provided \eqref{finite moments} is valid for some large but fixed $p$. See Remark \ref{remark: finite moments} below for a more precise statement.

Throughout the following we use a spectral parameter $z \in \C$ satisfying $\im z > 0$. We shall use the notation
\begin{equation*}
z \;=\; E + \ii \eta
\end{equation*}
without further comment.
The Stieltjes transform of Wigner's semicircle law is defined by
\begin{equation} \label{definition of msc}
m \;\equiv\; m(z) \;\deq\; \frac{1}{2 \pi} \int_{-2}^2 \frac{\sqrt{4 - \xi^2}}{\xi - z} \, \dd \xi\,.
\end{equation}
To avoid confusion, we remark that the Stieltjes transform $m$ was denoted by $m_{sc}$ in the papers \cite{ESY1, ESY2, ESY3, ESY4,ESY5,ESY6, ESY7, ESYY, EYY1, EYY2, EYY3, EKYY1, EKYY2}, in which $m$ had a different meaning from \eqref{definition of msc}. It is well known that the Stieltjes transform $m$ satisfies the identity
\begin{equation} \label{identity for msc}
m(z) + \frac{1}{m(z)} + z \;=\; 0\,.
\end{equation}
We define the \emph{resolvent} of $H$ through
\begin{equation*}
G(z) \;\deq\; (H - z)^{-1}\,,
\end{equation*}
and denote its entries by $G_{ij}(z)$. 
We also write $G^*(z) \deq (G(z))^* = (H - \bar z)^{-1}$. We often drop 
the argument $z$ and write $G \equiv G(z)$ as well as $G^* \equiv G^*(z)$.

\begin{definition}[Minors] \label{def: minors}
For $ T \subset \{1, \dots, N\}$ we define $H^{( T)}$ by
\begin{equation*}
(H^{( T)})_{ij} \;\deq\; \ind{i \notin  T} \ind{j \notin  T} h_{ij}\,.
\end{equation*}
Moreover, we define the resolvent of $H^{( T)}$ through
\begin{equation*}
G^{( T)}_{ij}(z) \;\deq\;  (H^{( T)} - z)^{-1}_{ij}\,.
\end{equation*}
We also set
\begin{equation*}
\sum_i^{( T)} \;\deq\; \sum_{i \col i \notin  T}\,.
\end{equation*}
When $ T = \{a\}$, we abbreviate $(\{a\})$ by $(a)$ in the above definitions; similarly, we write $(ab)$ instead of $(\{a,b\})$.
\end{definition}

\begin{definition}[Partial expectation and independence] \label{definition: P Q}
Let $X \equiv X(H)$ be a random variable. For $i \in \{1, \dots, N\}$ define the operations $P_i$ and $Q_i$ through
\begin{equation*}
P_i X \;\deq\; \E(X | H^{(i)}) \,, \qquad Q_i X \;\deq\; X - P_i X\,.
\end{equation*}
We call $P_i$ \emph{partial expectation} in the index $i$.
Moreover, we say that $X$ is \emph{independent of $T \subset \{1, \dots, N\}$} if $X = P_i X$ for all $i \in T$.
\end{definition}

The following definition introduces a notion of a high-probability bound that is suited for our purposes.
\begin{definition}[Stochastic domination]\label{def:stocdom}
Let $X = \pb{X^{(N)}(u) \col N \in \N, u \in U^{(N)}}$ be a family of random variables, where $U^{(N)}$ is a possibly $N$-dependent parameter set. Let $\Psi = \pb{\Psi^{(N)}(u) \col N \in \N, u \in U^{(N)}}$ be a deterministic family satisfying $\Psi^{(N)}(u) \geq 0$. We say that $X$ is \emph{stochastically dominated by $\Psi$, uniformly in $u$,} if for all (small) $\epsilon > 0$ and (large) $D > 0$ we have
\begin{equation*}
\sup_{u \in U^{(N)}} \P \qB{\absb{X^{(N)}(u)} > N^\epsilon \Psi^{(N)}(u)} \;\leq\; N^{-D}
\end{equation*}
for large enough $N\ge N_0(\e, D)$. Unless stated otherwise, 
throughout this paper the stochastic 
domination will always be uniform in all parameters apart from the parameter $\delta$ in \eqref{lower bound on W} and the sequence of constants $\mu_p$ in \eqref{finite moments}; thus, $N_0(\e, D)$ also depends on $\delta$ and $\mu_p$.
If $X$ is stochastically dominated by $\Psi$, uniformly in $u$, we use the equivalent notations
\begin{equation*}
X \;\prec\; \Psi \qquad \text{and} \qquad X \;=\; O_\prec(\Psi)\,.
\end{equation*}
\end{definition}

For example, using Chebyshev's inequality and \eqref{finite moments} one easily finds that 
\begin{equation}\label{hsmallerW}
h_{ij} \;\prec\; (s_{ij})^{1/2} \;\prec\; M^{-1/2}\,,
\end{equation}
so that we may also write $h_{ij} = O_\prec((s_{ij})^{1/2})$.
The relation $\prec$ satisfies the familiar algebraic rules of order relations. For instance if $A_1 \prec \Psi_1$ and $A_2 \prec \Psi_2$ then $A_1 + A_2 \prec \Psi_1 + \Psi_2$ and $A_1 A_2 \prec \Psi_1 \Psi_2$. Moreover, if $A \prec \Psi$ and there is a constant $C > 0$ such that $\Psi \geq N^{-C}$ and $\abs{A} \leq N^C$ almost surely, then $P_i A \prec \Psi$ and $Q_i A \prec \Psi$. More general statements in this spirit are given in Lemma \ref{lemma: basic properties of prec} below.

Let $\gamma > 0$ be a fixed small positive constant and let $(\f S^{(N)})$ be a sequence of domains satisfying
\begin{equation*}
\f S^{(N)} \;\subset\; \hb{z \in \C \col -10 \leq E \leq 10\,,\; M^{-1 + \gamma} \leq \eta \leq 10}\,.
\end{equation*}
As usual, we shall systematically omit the index $N$ on $\f S$.

\begin{definition}
A positive $N$-dependent deterministic function $\Psi \equiv \Psi^{(N)}$ on $\f S$ is called a \emph{control parameter}.
The control parameter $\Psi$ is \emph{admissible} if there is a constant $c > 0$ such that
\begin{equation} \label{admissible Psi}
M^{-1/2} \;\leq\; \Psi(z) \;\leq\; M^{-c}
\end{equation}
for all $N$ and $z \in \f S$.
\end{definition}

In this paper we always consider families $X^{(N)}(u) = X^{(N)}_i(z)$ indexed by $u = (z,i)$, where $z \in \f S$ and $i$ takes on values in some finite (possibly $N$-dependent or empty) index set.

We slightly modify the definition \eqref{Gmax} to include a control
on the diagonal entries of $G$ in addition to the off-diagonal entries. For the rest of the paper, we define the ($z$-dependent) random variable
\begin{equation*}
\Lambda(z) \;\deq\; \max_{x,y} \absb{G_{xy}(z) - \delta_{xy} m(z)}\,.
\end{equation*}
The variable $\Lambda$ will play the role of a \emph{random} control parameter.
If $\Psi$ is an admissible control parameter, the lower bound on $\Psi$ in \eqref{admissible Psi} together with 
\eqref{hsmallerW} imply that
\begin{equation} \label{h prec Psi}
h_{ij} \;\prec\; \Psi\,.
\end{equation}

\section{Simple examples and ingredients of the proof} \label{section: examples}

In this section we give an informal overview of fluctuation averaging, by stating and sketching the proofs of a few simple, yet representative, cases. Our starting point will always be an admissible control parameter $\Psi$ that controls $\Lambda$, i.e.\ $\Lambda \prec \Psi$. In addition to $\Psi$, we introduce the secondary control parameter
\begin{equation} \label{definition of Phi}
\Phi \;\equiv\; \Phi_\Psi \;\deq\; \min \hb{ \varrho \pb{\Psi + M^{-1/2} \Psi^{-1}}\,,\, 1}\,,
\end{equation}
where we defined the coefficient\footnote{Here we use the notation
$ \norm{A}_{\ell^\infty \to \ell^\infty} =\max_{i} \sum_j |A_{ij}|$ for the operator norm on $\ell^\infty(\C^N)$.}
\begin{equation} \label{def of rho}
\varrho  \;\deq\;
\normb{(1 - m^2 S)^{-1}}_{\ell^\infty \to \ell^\infty} \,.
\end{equation}
Thus, $\Phi$ is defined in terms of the primary control parameter $\Psi$, although we usually do not indicate this explicitly.

\begin{remark}
We use the somewhat complicated definitions \eqref{definition of Phi} and \eqref{def of rho} because they emerge naturally from our argument, and do not require us to impose any further conditions on the matrix $H$ or the spectral parameter $z$. The parameter $\Phi$ will describe the gain associated with a charged vertex or a chain vertex; see Definitions \ref{def: charged} and \ref{def: chains} below.

In the motivating example of band matrices (Example \ref{example: band matrix}), the parameter $\Phi$ may be considerably simplified. Indeed, in that case there is a positive constant $C$ such that
\begin{equation}
\varrho \;\leq\; \frac{C \log N}{(\im m)^2}\,,
\end{equation}
as proved in Proposition \ref{prop: rho for band} below. For most applications, we are interested in the bulk spectrum of the band matrix, i.e.\ $E \in [-2 + \kappa,2 - \kappa]$ for some fixed $\kappa > 0$. 
In that case the relation $\im m(z) \asymp \sqrt{\eta + 2 - \abs{E}}$
(proved e.g.\ in \cite[Lemma 4.2]{EYY2}) 
yields $\im m \geq c$ for some positive constant $c$ depending on $\kappa$. We conclude that $1 \leq \varrho \leq C \log N$; the logarithmic factor in the upper bound is irrelevant, since $\Phi$ will always be used as a deterministic control parameter in Definition \ref{def:stocdom}. In summary: for the bulk spectrum of a band matrix, we may replace $\Phi$ with $\Psi + M^{-1/2} \Psi^{-1}$.

Moreover, in typical applications the imaginary part $\eta$ of the spectral parameter $z$ is small enough that $\Psi \geq M^{-1/4}$. In this case $\Phi$ and $\Psi$ are comparable (in the bulk spectrum), and hence interchangeable as control parameters in Definition \ref{def:stocdom}.
\end{remark}

\begin{remark} \label{rem: lower bound on Phi}
We have the lower bound
\begin{equation} \label{lower bound on Phi}
1/2 \;\leq\; |1 - m^2|^{-1} \;\leq\; \varrho\,,
\end{equation}
where the first inequality follows from \eqref{m is bounded} below, and the second from the identity $(1 - m^2 S)^{-1} \f e = (1 - m^2)^{-1} \f e$ with the vector $\f e \deq (1, \dots, 1)$.
We therefore have the bounds $\Psi \leq 2\Phi \leq 2$.
\end{remark}

In this section we sketch the proof of the following result.

\begin{proposition}[Simple examples] \label{prop: examples}
Suppose that $\Lambda \prec \Psi$ for some admissible control parameter $\Psi$. Then we have
\begin{equation}\label{GGnoQ}
\frac{1}{N} \sum_{a}^{(\mu)} G_{\mu a} G_{a \mu} \;\prec\; \Psi^2 \Phi \,, \qquad
\frac{1}{N} \sum_{a}^{(\mu)} G_{\mu a} G_{a \mu}^* \;\prec\; \Psi^2
\end{equation}
as well as
\begin{equation}\label{GGwithQ}
\frac{1}{N} \sum_{a}^{(\mu)} Q_a (G_{\mu a} G_{a \mu}) \;\prec\; \Psi^3 \,, \qquad
\frac{1}{N} \sum_{a}^{(\mu)} Q_a(G_{\mu a} G_{a \mu}^*) \;\prec\; \Psi^3 \Phi\,.
\end{equation}
In addition, we have the bounds
\begin{equation} \label{avgG}
\frac{1}{N} \sum_a (G_{aa} - m) \;\prec\; \Psi \Phi\,, \qquad \frac{1}{N} \sum_a Q_a G_{aa} \;\prec\; \Psi^2\,.
\end{equation}
\end{proposition}

\begin{remark} \label{rem: W in examples}
As explained after \eqref{definition of Phi}, typically $\Phi$ and $\Psi$ are comparable. In this case the right-hand sides of the estimates in \eqref{GGnoQ} can be replaced with $\Psi^3$ and $\Psi^2$, those of \eqref{GGwithQ} with $\Psi^3$ and $\Psi^4$, and those of \eqref{avgG} with $\Psi^3$ and $\Psi^3$. Thus we may keep track of the improving effect of the average using a simple power counting in the single parameter $\Psi$, replacing each $\Phi$ with a $\Psi$.
\end{remark}

The significance of Proposition \ref{prop: examples} is the following. The trivial bound $G_{a\mu}\prec \Psi$ (which follows immediately from $\Lambda \prec \Psi$)
implies, for example, that $\frac{1}{N} \sum_{a}^{(\mu)} G_{\mu a} G_{a \mu}\prec \Psi^2$.
The first estimate in \eqref{GGnoQ} represents an improvement from $\Psi^2$
to $\Psi^2 \Phi$. This improvement
is due to the averaging over the index $a$ of fluctuating quantities
with almost vanishing expectation. We shall refer to such vertices as \emph{charged}; see Definition \ref{def: charged} below.
In contrast, there is no
such improvement in the second estimate of \eqref{GGnoQ}, since
$G_{\mu a}G_{a\mu}^* = \abs{G_{\mu a}}^2$ is always positive. 
If we subtract the expectation (for technical reasons, we subtract only the
partial expectation, i.e.\ take $Q_a = 1 - P_a$), then the averaging becomes effective and it
improves the average of $G_{\mu a}G_{a\mu}^*$ by two orders, from $\Psi^2$ to $\Psi^3 \Phi$.
Interestingly, subtracting the expectation in the average of $G_{\mu a}G_{a\mu}$
does not improve the estimate further; compare the first bounds
in \eqref{GGnoQ} and \eqref{GGwithQ}. (In fact, we get the only slightly stronger bound $\Psi^3$ instead of $\Psi^2 \Phi$.) These examples indicate that
the improving effect of the averaging heavily depends on the structure of the resolvent monomials.

We shall be concerned with averages of more general expressions. Roughly, we consider arbitrary monomials in the resolvent entries $(G_{ij})$. Some of the indices are summed. The summation is always performed with respect to a \emph{weight}, a nonnegative quantity which sums to one. In the examples of Proposition \ref{prop: examples}, the weight was $N^{-1}$. Generally, we want to allow weights consisting of factors $N^{-1}$ as well as $s_{ij}$; recall that $\sum_{j} s_{ij} = \sum_j N^{-1} = 1$.  Thus, in addition to \eqref{GGnoQ}, \eqref{GGwithQ}, and \eqref{avgG} we have for example the bounds
\begin{equation} \label{examples with s}
\sum_{a}^{(\mu)} s_{\nu a} G_{\mu a} G_{a \mu} \;\prec\; \Psi^2 \Phi\,, \qquad \sum_{a}^{(\mu)} s_{\nu a} Q_a(G_{\mu a} G_{a \mu}^*) \;\prec\; \Psi^3 \Phi\,, \qquad \sum_{a} s_{\nu a} (G_{a a} - m) \;\prec\; \Psi \Phi\,.
\end{equation}
A slightly more involved average is
\begin{equation} \label{example Z-lemma object}
\sum_{a,b} s_{\mu a} s_{\rho b} \, Q_b\pb{G_{\mu a} G_{a b} G_{b \nu}^* G_{a b}^* G_{\nu a}}
\end{equation}
where $\mu$, $\nu$, and $\rho$ are fixed. In Theorem \ref{theorem: Z lemma} 
we shall see that \eqref{example Z-lemma object} is stochastically dominated by $\Psi^6 \Phi^2$.
This means that the double averaging and the effect of one $Q$-operation amounts
to an improvement of a power three, from the trivial bound $\Psi^5$ to $\Psi^6 \Phi^2$. 
It may be tempting to think that each average and each factor $Q$ improves the
trivial bound by one power of $\Psi$ or $\Phi$, but this naive rule already fails in
the some of the simplest examples in \eqref{GGnoQ} and \eqref{GGwithQ}. The relation
between the averaging structure and the improved power of $\Psi$ and $\Phi$ is
more intricate.  Our final goal (see Theorem \ref{theorem: Z lemma}) is to establish an
optimal result for general monomials, which takes into account the precise
effect of all averages.

More generally, we shall be interested in averaging arbitrary monomials $\cal Z_{\f a}$ in the resolvent entries. Each such monomial contains a family of \emph{summation indices} $\f a$ and \emph{external indices} $\f \mu$. In the example \eqref{example Z-lemma object}, we have
\begin{equation} \label{example Z}
\cal Z_{\f a} \;=\; G_{\mu a} G_{a b} G_{b \nu}^* G_{a b}^* G_{\nu a} \,, \qquad \f a = (a,b)\,, \qquad \f \mu = (\mu, \nu)\,.
\end{equation} 
The most convenient way to define such a monomial $\cal Z_{\f a}$ is using a graph. The vertices are associated with the summation and external indices, and a resolvent entry $G_{xy}$ is represented as a directed edge from vertex $x$ to vertex $y$. We draw an edge associated with a resolvent entry $G_{xy}$ with a solid line, and an edge associated with a resolvent entry $G_{xy}^*$ with a dashed line. See Figure \ref{figure: simple graph example}. As it turns out, the gain in powers of $\Psi$ resulting from the averaging has a simple expression in terms of such graphs. Moreover, this graphical representation is a key tool in our proofs.
\begin{figure}[ht!]
\begin{center}
\includegraphics{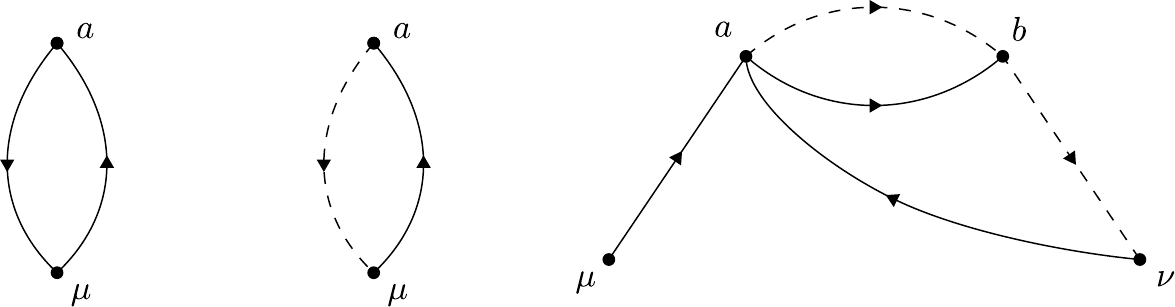}
\end{center}
\caption{Graphs associated with the monomials (from left to right) $G_{\mu a} G_{a \mu}$, $G_{\mu a} G_{a \mu}^*$, and $G_{\mu a} G_{a b} G_{b \nu}^* G_{a b}^* G_{\nu a}$ (from \eqref{example Z}). \label{figure: simple graph example}}
\end{figure}

Note that neither the $Q$-factors nor the averaging weights are encoded in the graphical structure. Later we shall give
a more precise definition of the class of weights we consider, but as an orientation
to the reader, we emphasize that they play a secondary role. As long as
the weights ensure an effective averaging over at least $M$ values of each summation
index, their final role is simply accounted for in the additional factor $M^{-1/2} \Psi^{-1}$ in the definition of $\Phi$.
The key improvement on the power of $\Psi$ in the final estimate is solely determined
by the structure of $\cal Z_{\f a}$ and by the locations of the $Q$-factors.

\subsection{Preliminaries}
In this subsection we collect some basic facts that will be used throughout the paper. We use $C$ to denote a generic large positive constant, which may depend on some fixed parameters and whose value may change from one expression to the next. For two positive quantities $A_N$ and $B_N$ we use the notation $A_N \asymp B_N$ to mean $C^{-1} A_N \leq B_N \leq C A_N$.

\begin{lemma} \label{lemma: msc}
There is a constant $c > 0$ such that for $E \in [-10, 10]$ and $\eta \in (0, 10]$
\begin{equation} \label{m is bounded}
c \;\leq\; \abs{m(z)} \;\leq\; 1\,.
\end{equation}
\end{lemma}
\begin{proof}
See Lemma 4.2 in \cite{EYY2}. 
\end{proof}

The following lemma collects basic algebraic properties of stochastic domination $\prec$.

\begin{lemma} \label{lemma: basic properties of prec}
\begin{enumerate}
\item
Suppose that $X(u,v) \prec \Psi(u,v)$ uniformly in $u \in U$ and $v \in V$. If $\abs{V} \leq N^C$ for some constant $C$ then
\begin{equation*}
\sum_{v \in V} X(u,v) \;\prec\; \sum_{v \in V} \Psi(u,v)
\end{equation*}
uniformly in $u$.
\item
Suppose that $X_{1}(u) \prec \Psi_{1}(u)$ uniformly in $u$ and $X_{2}(u) \prec \Psi_{2}(u)$ uniformly in $u$. Then
\begin{equation*}
X_{1}(u) X_{2}(u) \;\prec\; \Psi_{1}(u) \Psi_{2}(u)
\end{equation*}
uniformly in $u$.
\item
Suppose that $\Psi(u) \geq N^{-C}$ for all $u$ and that for all $p$ there is a constant $C_p$ such that $\E \abs{X(u)}^p \leq N^{C_p}$ for all $u$. Then, provided that $X(u) \prec \Psi(u)$ uniformly in $u$, we have
\begin{equation*}
P_a X(u) \;\prec\; \Psi(u) \qquad \text{and} \qquad Q_a X(u) \;\prec\; \Psi(u)
\end{equation*}
uniformly in $u$ and $a$.
\end{enumerate}
\end{lemma}
\begin{proof}
The claims (i) and (ii) follow from a simple union bound. The claim (iii) follows from Chebyshev's inequality, using a high-moment estimate combined with Jensen's inequality for partial expectation. We omit the details.
\end{proof}

We shall frequently make use of Schur's well-known complement formula, which we write as
\begin{equation} \label{schur}
\frac{1}{G_{ii}^{(T)}} \;=\; h_{ii} - z - \sum_{k,l}^{(T i)} h_{ik} G_{kl}^{(T i)} h_{li}\,,
\end{equation}
where $i \notin T \subset \{1, \dots, N\}$.

The following resolvent identities form the backbone of all of our calculations. The idea behind them is that a resolvent  matrix entry $G_{ij}$
depends strongly on the $i$-th and $j$-th columns of $H$, but weakly on all other columns. The first set of identities (called Family A) determines how to make
a resolvent matrix entry $G_{ij}$ independent of an additional index $k \neq i,j$.
The second set of identities (Family B) expresses the dependence of a resolvent matrix entry $G_{ij}$ on the matrix entries in the $i$-th or in the $j$-th column of $H$.

\begin{lemma}[Resolvent identities] \label{lemma: res id}
For any Hermitian matrix $H$ and $T \subset \{1, \dots, N\}$ the following identities hold.
\begin{description}
\item[Family A.] For $i,j,k \notin T$ and $k \neq i,j,$ we have
\begin{equation} \label{resolvent expansion type 1}
G_{ij}^{(T)} \;=\; G_{ij}^{(Tk)} + \frac{G_{ik}^{(T)} G_{kj}^{(T)}}{G_{kk}^{(T)}}\,, \qquad \frac{1}{G_{ii}^{(T)}} \;=\; \frac{1}{G_{ii}^{(Tk)}} - \frac{G_{ik}^{(T)} G_{ki}^{(T)}}{G_{ii}^{(T)} G_{ii}^{(Tk)} G_{kk}^{(T)}}\,.
\end{equation}
\item[Family B.]
For $i,j \notin T$ satisfying $i \neq j$ we have
\begin{subequations}
\label{resolvent expansion type 2}
\begin{align} \label{res exp 2b}
G_{ij}^{(T)} &\;=\; - G_{ii}^{(T)} \sum_{k}^{(Ti)} h_{ik} G_{kj}^{(Ti)} \;=\; - G_{jj}^{(T)} \sum_k^{(Tj)} G_{ik}^{(Tj)} h_{k j}
\\ \label{res exp 2b iterated}
G_{ij}^{(T)} &\;=\;  G_{ii}^{(T)} G_{jj}^{(Ti )} \pbb{- h_{ij} + \sum_{k,l}^{(Tij)} h_{ik} G_{kl}^{(Tij)} h_{lj}}\,,
\\ \label{res exp 2c}
\frac{1}{G^{(T)}_{ii}} &\;=\; \frac{1}{m} - \pb{- h_{ii} + Z_i^{(T)} + U_i^{(Ti)}}\,,
\end{align}
\end{subequations}
 where we defined
\begin{equation} \label{def of Zi and Ji}
Z_i^{(T)} \;\deq\; Q_i \sum_{k,l}^{(T i)} h_{ik} G_{kl}^{(T i)} h_{li} \,, \qquad U_i^{(S)} \;\deq\; \sum_{k}^{(S)} s_{ik} G_{kk}^{(S)} - m\,.
\end{equation}
\end{description}
\end{lemma}
\begin{proof}
The first identity of \eqref{resolvent expansion type 1} was proved in Lemma 4.2 of \cite{EYY1}. The second identity of \eqref{resolvent expansion type 1} is an immediate consequence of the first. The identities \eqref{res exp 2b} were proved in Lemma 6.10 of \cite{EKYY2}, and \eqref{res exp 2b iterated} follows by iterating \eqref{res exp 2b} twice. Finally, \eqref{res exp 2c} (together with \eqref{def of Zi and Ji}) follows easily from \eqref{schur}, \eqref{identity for msc}, the partition $1 = Q_i + P_i$, and the definition \eqref{variance of h}.
\end{proof}

Next, we record a simple estimate on resolvent entries of minors. For $T \subset \{1, \dots, N\}$ define the random variable
\begin{equation*}
\Lambda^{(T)}(z) \;\deq\; \max_{i,j \notin T} \absB{G_{ij}^{(T)}(z) - \delta_{ij} m(z)}\,.
\end{equation*}
\begin{lemma}[Bound on $\Lambda^{(T)}$] \label{lemma: Lambda T}
Suppose that $\Lambda \prec \Psi$ for some admissible control parameter $\Psi$. 
Then for any fixed $\ell \in \N$ we have
\begin{equation} \label{rough bound 2}
\Lambda^{(T)} \;\prec\; \Psi
\end{equation}
provided that $\abs{T} \leq \ell$. (The threshold $N_0(\e, D)$
in Definition~\ref{def:stocdom}  may also depend on $\ell$).
\end{lemma}
\begin{proof}
See Appendix \ref{section: proof of res id}.
\end{proof}

In particular, if $\Lambda \prec \Psi$ for some admissible $\Psi$, then Lemmas \ref{lemma: Lambda T} and \ref{lemma: msc} imply that for any fixed $\ell \in \N$ we have
\begin{equation} \label{1/G prec 1}
\frac{1}{G_{ii}^{(T)}} \;\prec\; 1
\end{equation}
provided that $\abs{T} \leq \ell$. We conclude this section with rough bounds on the entries of $G$, which will be used to deal with exceptional, low-probability events. 
\begin{lemma}[Rough bounds on $G$] \label{lemma: rough bounds on G}
Suppose that $\Lambda \prec \Psi$ for some admissible control parameter $\Psi$.
\begin{enumerate}
\item
We have
\begin{equation} \label{rough bound 1}
\absb{G_{ij}^{(T)}(z)} \;\leq\; \eta^{-1} \; \leq \; M
\end{equation}
for all $z \in \f S$, $T \subset \{1, \dots, N\}$, and $i,j \notin T$.
\item
For every $p \in \N$ and $\ell \in \N$ there is a constant $C_{p,\ell}$ such that
\begin{equation} \label{rough bound 3}
\E \absB{1/G_{ii}^{(T)}(z)}^p \;\leq\; C_{p,\ell}
\end{equation}
for all $T \subset \{1, \dots, N\}$ satisfying $\abs{T} \leq \ell$, all $z \in \f S$, and all $i \notin T$.
\end{enumerate}
\end{lemma}
\begin{proof}
See Appendix \ref{section: proof of res id}.
\end{proof}

\subsection{Some ingredients of the proof of Proposition \ref{prop: examples}} \label{sect: sketch of proof of examples}
A reader interested only in our main theorem (Theorem \ref{theorem: Z lemma}) may skip this section and proceed to Section~\ref{sec:multZ} directly. 
Here we sketch the proof of Proposition \ref{prop: examples}. Our goal
 is to motivate some concepts underlying our main theorem, and
to give an impressionistic overview of some ideas in its proof. The actual
proof of  Proposition \ref{prop: examples} will not be needed, since Theorem \ref{theorem: Z lemma} implies Proposition \ref{prop: examples} as a special case.

To avoid needless complications in our proof, we additionally assume that we are dealing with one of the two classical symmetry classes of random matrices: 
real symmetric and complex Hermitian. For \emph{real symmetric band matrices} we assume
\begin{equation} \label{RS}
\zeta_{ij} \in \R \quad \text{for all} \quad i \leq j \,.
\end{equation}
For {\it complex Hermitian band matrices}
we assume
\begin{equation} \label{CH}
\E \zeta_{ij}^2 = 0 \quad \text{for all} \quad i < j\,.
\end{equation}
A common way to satisfy \eqref{CH} is to choose the real and imaginary parts of $\zeta_{ij}$ to be independent with identical variance.
In Remark \ref{rem: relaxing symmetry} below we explain how to remove the assumption that \eqref{RS} or \eqref{CH} holds, i.e.\ how to remove the assumption $\E \zeta_{ij}^2 = 0$ in the case \eqref{CH}.

The second estimate of \eqref{GGnoQ} follows trivially from $\abs{G_{\mu a}} \leq \Lambda \prec \Psi$. We shall sketch the proofs of the remaining inequalities in the following order:
\begin{itemize}
\item[(A)]
 first estimate of \eqref{GGwithQ} and second estimate of \eqref{avgG};
\item[(B)]
 first estimate of \eqref{GGnoQ} and first estimate of \eqref{avgG};
\item[(C)]
second estimate of \eqref{GGwithQ}.
\end{itemize}
This order corresponds to an increasing degree of complication of the proofs. These three steps thus serve as simple examples in which to introduce four basic concepts underlying our proof. More specifically, in the language of the full proof (Sections \ref{section: Z1} -- \ref{section: simplifications}), (A) requires only the simple high-moment estimate from Section \ref{section: Z3}, (B) requires in addition the inversion of a stable self-consistent equation (Section \ref{sec: chains with no Q}), and (C) requires in addition a priori bounds on chains (Sections \ref{sec: chains with no Q} and \ref{section: Z4}) as well as the procedure of vertex resolution (Section \ref{section: Z5}).

\subsubsection{Proof of (A)} \label{sec: sketch of A}
 We focus first on the first estimate of \eqref{GGwithQ}. We derive the stochastic bound from high-moment bounds and Chebyshev's inequality. To simplify the presentation, we only estimate the variance
\begin{equation} \label{variance of noQ}
\E \absbb{\frac{1}{N} \sum_{a}^{(\mu)} Q_a (G_{\mu a} G_{a \mu})}^2 \;=\; \frac{1}{N^2} \sum_{a,b}^{(\mu)} \E \, Q_a \pB{G_{\mu a} G_{a \mu}} Q_b \ol{\pB{G_{\mu b} G_{b \mu}}}\,.
\end{equation}
Our goal is to prove that \eqref{variance of noQ} is bounded by $C \Psi^6$. We partition the summation into the cases $a = b$ and $a \neq b$. For the case $a = b$, we easily get from Lemmas \ref{lemma: basic properties of prec} and \ref{lemma: rough bounds on G}  the bound $C N^{-1} \Psi^4 \leq C \Psi^6$, where we used \eqref{lower bound on W} and the fact that $\Psi$ satisfies Definition \ref{admissible Psi}.

Let us therefore focus on the case $a \neq b$. We use \eqref{resolvent expansion type 1} to get
\begin{align}
&\mspace{-20mu} \E \, Q_a \p{G_{\mu a} G_{a \mu}} Q_b \ol{\p{G_{\mu b} G_{b \mu}}}
\notag \\
&\;=\; \E \, Q_a \qBB{\pbb{G_{\mu a}^{(b)} + \frac{G_{\mu b} G_{ba}}{G_{bb}}} \pbb{G_{a \mu}^{(b)} + \frac{G_{a b} G_{b \mu}}{G_{b b}}}} 
Q_b \ol{\qBB{\pbb{G_{\mu b}^{(a)} + \frac{G_{\mu a} G_{ab}}{G_{aa}}} \pbb{G_{b \mu}^{(a)} + \frac{G_{ba} G_{a \mu}}{G_{aa}}}}}
\notag \\ \label{simple example for weak Zlemma}
&\;=\; \E \, Q_a \qBB{\pbb{G_{\mu a}^{(b)} + \frac{G_{\mu b}^{(a)} G_{ba}}{G_{bb}^{(a)}}} \pbb{G_{a \mu}^{(b)} + \frac{G_{a b} G_{b \mu}^{(a)}}{G_{b b}^{(a)}}}} 
Q_b \ol{\qBB{\pbb{G_{\mu b}^{(a)} + \frac{G_{\mu a}^{(b)} G_{ab}}{G_{aa}^{(b)}}} \pbb{G_{b \mu}^{(a)} + \frac{G_{ba} G_{a \mu}^{(b)}}{G_{aa}^{(b)}}}}} + \cdots\,,
\end{align}
where we dropped the higher order terms of the expansion. The philosophy behind this expansion is to make each resolvent entry independent of as many indices in $(a,b)$ as possible
by using \eqref{resolvent expansion type 1} iteratively.
We call such terms \emph{maximally expanded} in $(a,b)$, i.e.\ a maximally expanded resolvent entry cannot be made independent of $a$ or $b$ using the identity \eqref{resolvent expansion type 1}; the reason is that either it already has $a$ and $b$ as upper indices or an index from $(a,b)$ appears as a lower index. See Definition \ref{definition: maximally expanded} below for a precise statement. The iteration is stopped if either \eqref{resolvent expansion type 1} cannot be applied to any resolvent 
entry or if a sufficient number of resolvent entries (in our case a total of
six) have been generated. 
(In the proof of Proposition \ref{lemma: weak moment estimate} we give a precise definition of this stopping rule.)

 We now multiply everything out on the right-hand side of \eqref{simple example for weak Zlemma} to get terms of the form 
$\E Q_a(A) Q_b(B)$. The key observation is that if $B$ is independent of $a$ then the
 expectation vanishes (in fact, already the partial expectation $P_a$
renders the whole term zero). Similarly, if $A$ is independent of $b$ then the expectation vanishes. An example of a leading-order term from \eqref{simple example for weak Zlemma} that does not vanish is
\begin{equation} \label{simplest nonzero term}
\E \, Q_a \qBB{\frac{G_{\mu b}^{(a)} G_{ba}}{G_{bb}^{(a)}} G_{a \mu}^{(b)}} 
Q_b \ol{\qBB{\frac{G_{\mu a}^{(b)} G_{ab}}{G_{aa}^{(b)}} G_{b \mu}^{(a)} }}.
\end{equation}
(Note that all resolvent entries are maximally expanded in $(a,b)$.)
In this fashion each $Q$ imposes the presence of at least one additional off-diagonal entry. Since every off-diagonal resolvent entry contributes a factor $\Psi$ (see Lemma \ref{lemma: Lambda T}), we find that \eqref{simple example for weak Zlemma} is of order $\Psi^6$ instead of the naive $\Psi^4$. This concludes the sketch of the proof of the first estimate of \eqref{GGwithQ}.

The sketch of the proof of the second estimate of \eqref{avgG} 
is almost identical, and therefore omitted.

\subsubsection{Introduction of graphs} \label{sec: sketch: graphs}
Before moving on to (B) and (C), we take this opportunity to introduce a graphical language which is useful for keeping track of terms such as \eqref{simplest nonzero term}. Although not needed here, since the example in (A) is very simple,
this language will prove essential when defining more complicated expressions, as well as for the actual proof of Theorem \ref{theorem: Z lemma}.
Recall from Figure \ref{figure: simple graph example} that we can represent the expression $G_{\mu a} G_{a \mu}$ graphically by regarding $\mu$ and $a$ as vertices, and by drawing two directed edges associated with $G_{\mu a}$ and $G_{a \mu}$.  We adopt the convention given after \eqref{example Z}. Thus, an off-diagonal resolvent entry of $G_{ab}$, $a \neq b$, is represented with a directed solid line from $a$ to $b$, and the analogous entry $G^*_{ab}$ with a directed dashed line from $a$ to $b$.

\begin{convention*}
We sometimes identify a vertex with its associated summation index, and hence use the letter $a$ to denote two different things: a vertex of a graph and the value of the associated index. This allows us to avoid a proliferation of double subscripts in expressions like $G_{a_i a_j}$.
When depicting graphs, we always label a vertex using the name of the associated index.
\end{convention*}

We shall also have to deal with diagonal resolvent entries; in fact we introduce separate notations the three most common functions of them. Our graphical conventions are summarized in Figure \ref{figure: diagonal elements}.
\begin{figure}[ht!]
\begin{center}
\includegraphics{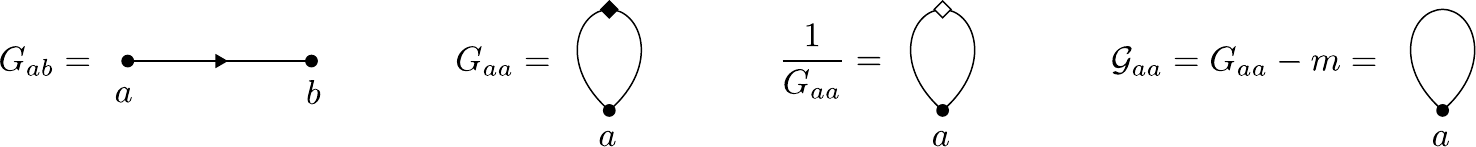}
\end{center}
\caption{The graphical representations of resolvent entries. The versions associated with $G^*$ are the same with a dashed line.  \label{figure: diagonal elements}}
\end{figure}
We may thus represent the expression on the left-hand side of \eqref{simple example for weak Zlemma}, i.e.\ $Q_a \p{G_{\mu a} G_{a \mu}} Q_b \p{G_{\mu b}^* G_{b \mu}^*}$,
See Figure \ref{figure: simple graph for variance}; note that our graphical notation does not keep track of the factors $Q$.
\begin{figure}[ht!]
\begin{center}
\includegraphics{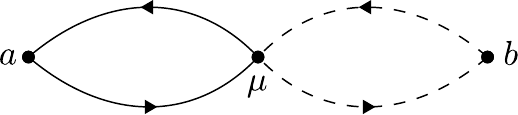}
\end{center}
\caption{Graph associated with the monomial $G_{\mu a} G_{a \mu} G_{\mu b}^* G_{b \mu}^*$. Here we draw the case $a \neq b$. \label{figure: simple graph for variance}}
\end{figure}
Having drawn the graph in Figure \ref{figure: simple graph for variance}, we start making all resolvent entries (corresponding to edges) independent of the indices $a$ and $b$, using the 
identities \eqref{resolvent expansion type 1}. As explained above, this gives rise to a sum of terms, each one of which is a fraction of resolvent entries that are maximally expanded in $(a,b)$. 
The denominator of each term contains diagonal resolvent entries, while its numerator is
a product of off-diagonal resolvent entries; this follows 
from the structure of \eqref{resolvent expansion type 1}.
A simple such example was given in \eqref{simplest nonzero term}. The associated monomial,
\begin{equation} \label{monomial for simple nonzero term}
\frac{G_{\mu b}^{(a)} G_{ba}}{G_{bb}^{(a)}} G_{a \mu}^{(b)}\,
\frac{G_{a \mu}^{(b)*} G_{ba}^*}{G_{aa}^{(b)*}} G_{\mu b}^{(a)*},
\end{equation}
may be represented graphically as in Figure \ref{figure: expanded simple graph for variance}.
\begin{figure}[ht!]
\begin{center}
\includegraphics{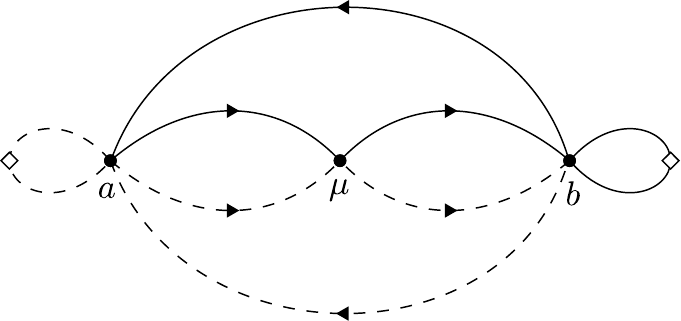}
\end{center}
\caption{Graph associated with \eqref{monomial for simple nonzero term}. Here we draw the case $a \neq b$. \label{figure: expanded simple graph for variance}}
\end{figure}

We remark that the graphs depicted in Figures \ref{figure: simple graph for variance} and \ref{figure: expanded simple graph for variance} are fundamentally different in the following sense. In Figure \ref{figure: simple graph for variance}, each edge of the graph represents a resolvent entry with no upper indices; in Figure \ref{figure: expanded simple graph for variance}, each edge of the graph represents a resolvent entry that is maximally expanded in $(a,b)$. In the language of Section \ref{section: Z3}, the former graph will be called $\gamma^2(\Delta)$ while the latter will be called $\Gamma$. It is the latter graphs that play a major role in our proofs. The former type is simply a trivial concatenation of basic graphs, and serves as an intermediate step in the construction of graphs of the latter type (i.e.\ whose edges represent maximally expanded resolvent entries). If one wanted to be more precise, one could keep track of the upper indices associated with each edge in the graphs. By definition, the edges of the graph in Figure \ref{figure: simple graph for variance} have no upper indices, and the edges of the graph in Figure \ref{figure: expanded simple graph for variance}  have
upper indices as given in \eqref{monomial for simple nonzero term}. However, these upper indices are unambiguously determined by the condition that each resolvent entry be maximally expanded in $(a,b)$. This means that $a$ appears as upper index of any edge that is not incident to $a$ (and similarly with $b$). In practice, however, we do not indicate the upper indices, as they are uniquely determined by the condition that all edges are maximally expanded in $(a,b)$.

It is possible, and indeed important for our proof, to introduce a graphical rule that generates graphs like the one depicted in Figure \ref{figure: expanded simple graph for variance} from graphs like the one depicted in Figure \ref{figure: simple graph for variance} through a sequence of graphs whose edges are not yet maximally expanded. Before the maximal expansion is achieved, we shall temporarily indicate the upper indices 
on the graph edges in parenthesis. Recall that the underlying algebra was simply governed by the identities \eqref{resolvent expansion type 1}. Figure \ref{figure: expansion of off-diag G} depicts the identity
\begin{equation} \label{exp off-diag in picture}
G_{ij} \;=\; G_{ij}^{(k)} + \frac{G_{ik} G_{kj}}{G_{kk}}\,.
\end{equation}
\begin{figure}[ht!]
\begin{center}
\includegraphics{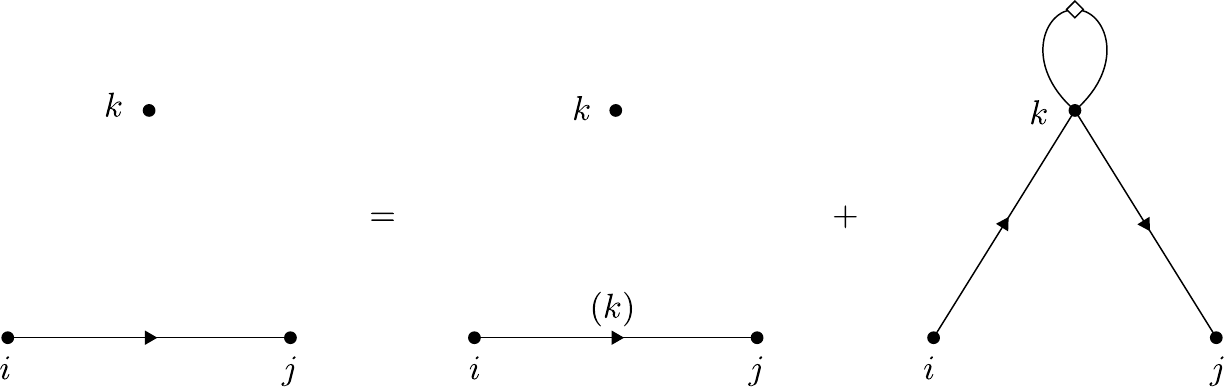}
\end{center}
\caption{The graphical representation of the formula \eqref{exp off-diag in picture}.
\label{figure: expansion of off-diag G}}
\end{figure}
Similarly, the corresponding identities for the diagonal entries,
\begin{equation} \label{exp in picture}
\frac{1}{G_{ii}} \;=\; \frac{1}{G_{ii}^{(j)}} - \frac{G_{ij} G_{ji}}{G_{ii} G_{ii}^{(j)} G_{jj}}\,, \qquad
G_{ii} \;=\; G_{ii}^{(j)} + \frac{G_{ij} G_{ji}}{G_{jj}}\,,
\end{equation}
are depicted in Figure \ref{figure: expansion of diag G}.
\begin{figure}[ht!]
\begin{center}
\includegraphics{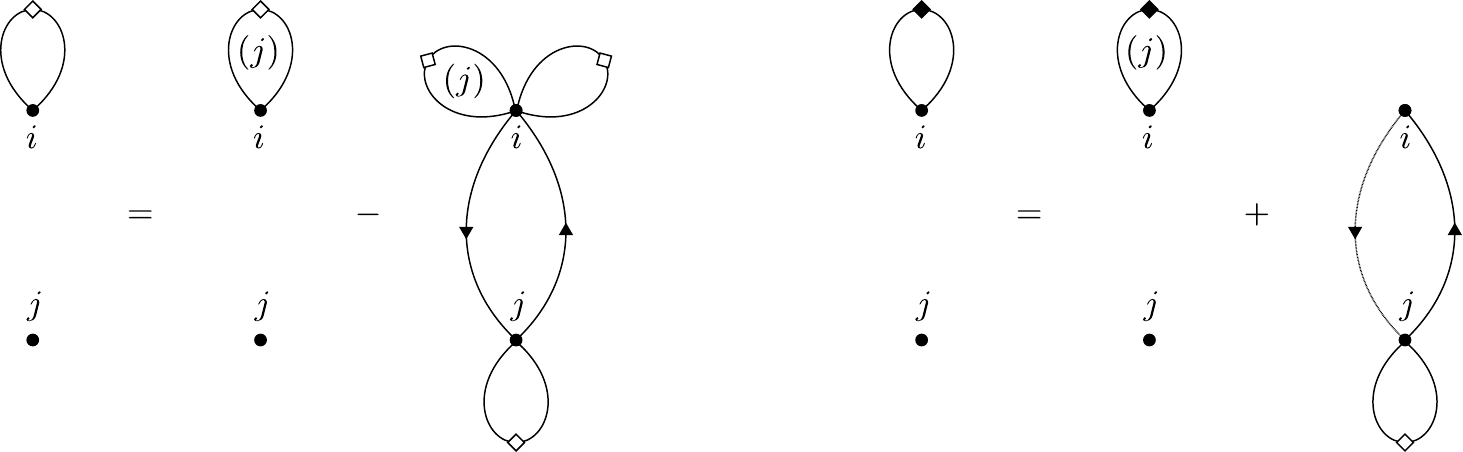}
\end{center}
\caption{Adding an upper index $j$ to the diagonal entries $1/G_{ii}$ and $G_{ii}$. These pictures correspond to \eqref{exp in picture}. We exceptionally also mark the upper indices associated with each edge. \label{figure: expansion of diag G}}
\end{figure}
Applying the graphical rules of Figures \ref{figure: expansion of off-diag G} and \ref{figure: expansion of diag G} to Figure \ref{figure: simple graph for variance} results e.g.\ in Figure \ref{figure: expanded simple graph for variance} (and many others). To be precise, we should keep track of the upper indices associated with each edge at each step, as is done in Figure \ref{figure: expansion of diag G}. When all edges are maximally expanded, we stop the application of the rules of Figures \ref{figure: expansion of off-diag G} and \ref{figure: expansion of diag G}. However, as explained above, we usually omit the explicit indication of upper indices in graphs after the maximal expansion is achieved. For future use, we record the following definition associated with the operations depicted in Figures \ref{figure: expansion of off-diag G} and \ref{figure: expansion of diag G}.

\begin{definition} \label{def: linking}
We refer to the second graph on the right-hand side of Figure \ref{figure: expansion of off-diag G} as arising from \emph{linking} the edge $(ij)$ with the vertex $k$. We also say that the vertex $k$ was \emph{linked to} by the edge $(ij)$. Similarly, in both connected graphs in Figure \ref{figure: expansion of diag G}, the vertex $j$ was linked to by the edge $(ii)$.
\end{definition}

The argument underlying \eqref{simple example for weak Zlemma} may now be formulated graphically as follows. We start from Figure \ref{figure: simple graph for variance} and apply the identities from Figures \ref{figure: expansion of off-diag G} and \ref{figure: expansion of diag G} until all resolvent entries associated with the edges are maximally expanded in $(a,b)$. Since these identities can be applied in various orders, this procedure is
not unique. This lack of uniqueness does not concern us, however: we need only a maximally expanded representation.
By the argument given after \eqref{simple example for weak Zlemma}, we know that only those graphs in which both $a$ and $b$ have been linked to by an edge survive. Such graphs (as the one from Figure \ref{figure: expanded simple graph for variance}) have (at least) two additional edges as compared to the one from Figure \ref{figure: simple graph for variance}. This results in a size $O_\prec(\Psi^6)$.

\subsubsection{Sketch of the proof of (B)} \label{sec: proof of (B)}
 We focus first on the first estimate of \eqref{GGnoQ}.
The idea is to derive a stable self-consistent equation for the quantity
\begin{equation}\label{sce}
\frac{1}{N} \sum_{a}^{(\mu)} G_{\mu a} G_{a \mu}\,.
\end{equation}
We do this by introducing the partition $1 = P_a + Q_a$ inside the summation. The second resulting term was estimated in (A). The first resulting term may be written as
\begin{align*}
\frac{1}{N} \sum_{a}^{(\mu)} P_a \pb{G_{\mu a} G_{a \mu}} &\;=\; \frac{1}{N} \sum_{a}^{(\mu)} P_a \pbb{ \frac{m^2}{G_{aa}^2}G_{\mu a} G_{a \mu}} + O_\prec(\Psi^3)
\\
&\;=\; m^2 \, \frac{1}{N} \sum_{a}^{(\mu)} P_a \pbb{\sum_{x,y}^{(a)} G_{\mu x}^{(a)} h_{xa} h_{ay} G_{y \mu}^{(a)}} + O_\prec(\Psi^3)
\\
&\;=\; m^2 \, \frac{1}{N} \sum_{a}^{(\mu)} \sum_{x}^{(a)} s_{ax} G_{\mu x}^{(a)} G_{x \mu}^{(a)} + O_\prec(\Psi^3)
\\
&\;=\; m^2 \, \frac{1}{N} \sum_{a}^{(\mu)} \sum_{x}^{(a)} s_{ax} G_{\mu x} G_{x \mu} + O_\prec(\Psi^3)
\\
&\;=\; m^2 \, \frac{1}{N} \sum_{x}^{(\mu)} G_{\mu x} G_{x \mu} + O_\prec(\Psi^3 + N^{-1})\,.
\end{align*}
In the first step we used \eqref{1/G prec 1}. In the second step we used the identity \eqref{res exp 2b} (note the usefulness of smuggling in $G_{aa}$ in the previous step). In the third step we used the identity $P_a  h_{xa} h_{ay} =\E h_{xa} h_{ay} = s_{ax}\delta_{xy}$, as follows from the definition of $H$ and the fact that $G^{(a)}$ is independent of $a$. In the fourth step we used the identity \eqref{resolvent expansion type 1} to remove the upper indices. In the fifth step we used a simple analysis of coinciding indices together with the estimates \eqref{s leq W} and \eqref{m is bounded}. Together with the bound from (A), we therefore get for the quantity \eqref{sce} the self-consistent equation
\begin{align*}
\frac{1}{N} \sum_{a}^{(\mu)} G_{\mu a} G_{a \mu} &\;=\; \frac{1}{N} \sum_{a}^{(\mu)} P_a \pb{G_{\mu a} G_{a \mu}} + O_\prec(\Psi^3)
\\
&\;=\; m^2 \frac{1}{N} \sum_{x}^{(\mu)} G_{\mu x} G_{x \mu} + O_\prec \pb{\Psi^3 + M^{-1}}
\\
&\;=\; m^2 \frac{1}{N} \sum_{x}^{(\mu)} G_{\mu x} G_{x \mu} + O_\prec \pB{\Psi^2 \pb{\Psi + M^{-1/2} \Psi^{-1}}}\,,
\end{align*}
where in the last step we used \eqref{admissible Psi}.
Using \eqref{lower bound on Phi} and the trivial bound $\frac{1}{N} \sum_{a}^{(\mu)} G_{\mu a} G_{a \mu} \prec \Psi^2$ we therefore get
\begin{equation*}
\frac{1}{N} \sum_{a}^{(\mu)} G_{\mu a} G_{a \mu}  \;\prec\; \min \hbb{ \Psi^2 \, , \,
 \frac{\Psi^2 \pb{\Psi + M^{-1/2} \Psi^{-1}}}{|1-m^2|}} \;\le\; \min \hb{ \Psi^2 \, , \,
 \varrho \Psi^2 \pb{\Psi + M^{-1/2} \Psi^{-1}}} \;=\; \Psi^2 \Phi\,,
\end{equation*}
which is the first estimate of \eqref{GGnoQ}.

The proof of the first estimate of \eqref{avgG} is similar, except that we derive the self-consistent equation using \eqref{res exp 2c} instead of \eqref{res exp 2b}. Using the second estimate of \eqref{avgG} we find  
\begin{equation}\label{eq:selfcon1}
\frac{1}{N} \sum_a (G_{aa} - m) \;=\; \frac{1}{N} \sum_a P_a (G_{aa} - m) + \frac{1}{N} \sum_a Q_a G_{aa} \;=\; \frac{1}{N} \sum_a P_a(G_{aa} - m) + O_\prec(\Psi^2)\,.
\end{equation}
Next, from a simple large deviation estimate (see the first paragraph in the proof of Lemma \ref{lemma: diagonal estimates} in Appendix \ref{section: proof of res id}) we find $Z_a \prec \Psi$. Moreover, Lemma \ref{lemma: Lambda T}, \eqref{S is stochastic}, \eqref{s leq W}, and \eqref{admissible Psi} readily imply that $U_a^{(a)} \prec \Psi$.
Recalling the estimate \eqref{h prec Psi}, we may therefore expand the identity \eqref{res exp 2c} using \eqref{identity for msc} to get
\begin{equation*}
G_{aa} \;=\; m + m^2 (-h_{aa} + Z_a  + U_a^{(a)}) + O_\prec(\Psi^2)\,.
\end{equation*}
Using $P_a h_{aa} = 0$ and $P_a Z_a = 0$ we therefore find
\begin{align*}
\frac{1}{N} \sum_a P_a (G_{aa} - m) &\;=\; m^2 \frac{1}{N} \sum_a P_a U_a^{(a)} + O_\prec(\Psi^2)
\\
&\;=\; m^2 \frac{1}{N} \sum_a \sum_x^{(a)} s_{ax} (G_{xx}^{(a)} - m) + O_\prec(\Psi^2)
\\
&\;=\; m^2 \frac{1}{N} \sum_x (G_{xx} - m) + O_\prec(\Psi^2)\,,
\end{align*}
 where in the second step we recalled the definition \eqref{def of Zi and Ji} and used \eqref{s leq W} as well as \eqref{S is stochastic} to write $m = \sum_x^{(a)} s_{ax} m + O(M^{-1})$ with $M^{-1} = O(\Psi^2)$ by \eqref{admissible Psi}, and in the third \eqref{resolvent expansion type 1} to get rid of the upper index $a$ as well as \eqref{S is stochastic}. Thus, together with \eqref{eq:selfcon1}, we get the self-consistent equation
\begin{equation} \label{sceq for Gaa}
\frac{1}{N} \sum_a (G_{aa} - m) \;=\; m^2 \frac{1}{N} \sum_x (G_{xx} - m) + O_\prec(\Psi^2)\,,
\end{equation}
from which we easily conclude the first estimate of \eqref{avgG} as before.

In both of the above examples the averaging was performed with respect to the uniform weight $w_a = N^{-1}$. We conclude by sketching the differences in the case of a nontrivial weight, e.g.\ $w_a = s_{\nu a}$. Consider for example the average $\sum_a s_{\nu a} (G_{aa} - m)$
from \eqref{examples with s}. Repeating the above derivation of \eqref{sceq for Gaa}, we find the self-consistent \emph{system of  equations}
\begin{equation*}
\sum_a s_{\nu a} (G_{aa} - m) \;=\; m^2 \sum_{a} s_{\nu a}\sum_x
 s_{ax} (G_{xx} - m) + E_\nu\,
\end{equation*}
for each $\nu$. Here the error satisfies $E_\nu = O_\prec(\Psi^2)$.
Introducing the vectors  $\f v = (v_a)_{a = 1}^N$ defined by $v_a \deq \sum_x s_{ax} (G_{xx} - m)$
and $\f E =(E_\nu)_{\nu=1}^N$, we have
\begin{equation*}
\f v \;=\; m^2 S \f v + {\f E}\,.
\end{equation*}
 Thus we find
\begin{equation*}
\f v \;=\; (1 - m^2 S)^{-1} {\f E}\,,
\end{equation*}
from which we conclude that $v_a \prec \varrho \Psi^2 \leq \Psi \Phi$.

\subsubsection{Sketch of the proof of (C)} \label{sec: sketch of resolution}
As in (A), the proof is based on a high-moment estimate. We again restrict our attention to the variance
\begin{equation} \label{variance of G G*}
\E \absbb{\frac{1}{N} \sum_{a}^{(\mu)} Q_a(G_{\mu a} G_{a \mu}^*)}^2 \;=\; \frac{1}{N^2} \sum_{a,b}^{(\mu)} \E \, Q_a(G_{\mu a} G_{a \mu}^*) \, Q_b(G_{\mu b} G_{b \mu}^*)\,.
\end{equation}
Our goal is to derive the stochastic bound $\Psi^6 \Phi^2$ for \eqref{variance of G G*}.
The case $a = b$ yields the bound
\begin{equation} \label{variance for coinciding indices}
\frac{1}{N^2} \sum_a^{(\mu)} \abs{G_{\mu a}}^4 \;\prec\; \Psi^4 N^{-1} \;\leq\; \Psi^6 \Phi^2\,.
\end{equation}
Let us therefore assume for the following that $a \neq b$. The first part of the argument follows precisely the proof of (A) above. We expand all resolvent entries 
of
\begin{equation} \label{monomial for simple resolution}
\E \, Q_a(G_{\mu a} G_{a \mu}^*) \, Q_b(G_{\mu b} G_{b \mu}^*)
\end{equation}
using \eqref{resolvent expansion type 1} and obtain a sum of monomials whose resolvent entries are maximally expanded in $(a,b)$. A typical example of a nonvanishing term arising from the expansion of \eqref{variance of G G*} is
\begin{equation*}
\E \, Q_a\pBB{\frac{G_{\mu b}^{(a)} G_{b a}}{G_{bb}^{(a)}} G_{a \mu}^{(b)*}} \, Q_b\pBB{G_{\mu b}^{(a)} \frac{G_{ba}^* G_{a \mu}^{(b)*}}{G_{aa}^{(a)*}}}\,.
\end{equation*}
As in the proof of (A), this immediately gives the stochastic bound $\Psi^6$. See Figure \ref{figure: linking for variance of G G*} for a graphical summary of the argument in this context.
\begin{figure}[ht!]
\begin{center}
\includegraphics{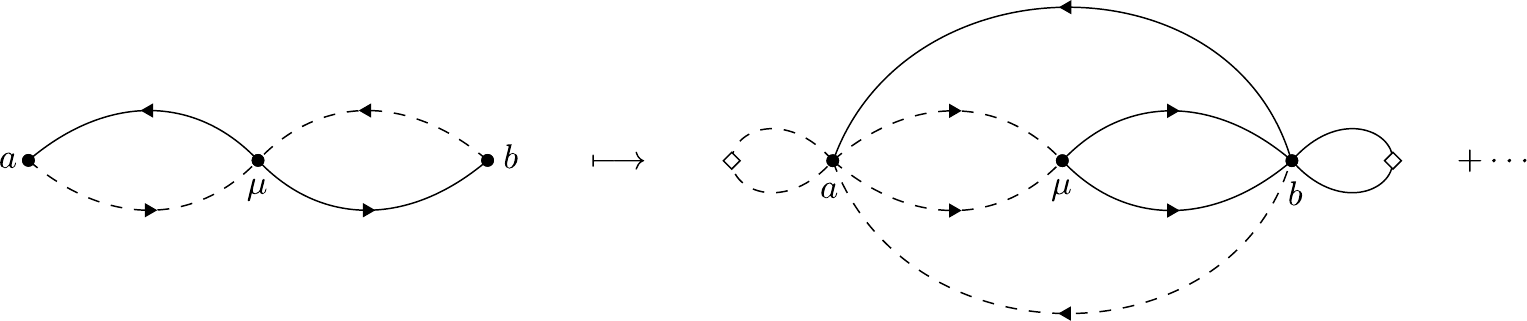}
\end{center}
\caption{The process of making all edges of the graph associated with \eqref{monomial for simple resolution} maximally expanded in $(a,b)$. \label{figure: linking for variance of G G*}}
\end{figure}

The bound $\Psi^6$ is not enough, however. In order to improve this to $\Psi^6 \Phi^2$, we introduce a new operation which we call \emph{vertex resolution}. In order to simplify the presentation, in the following we systematically replace any diagonal entry $G_{aa}^{(T)}$ by $m$. The resulting error terms are small by definition of $\Lambda$ (of course, they have to be dealt with, which is done Section \ref{section: Z6} of the full proof below). Thus, we have to estimate the expression
\begin{equation} \label{simplified expression for variance}
\E \, Q_a\pb{G_{\mu b}^{(a)} G_{b a} G_{a \mu}^{(b)*}} \, Q_b\pb{G_{\mu b}^{(a)} G_{ba}^* G_{a \mu}^{(b)*}}
\end{equation}
for $a \neq b$.
We begin by expanding all resolvent entries using the Family B identity \eqref{res exp 2b}, again neglecting the diagonal prefactors in \eqref{res exp 2b}. This gives
\begin{equation} \label{simple vertex resolution}
\sum_{x,y,z,w}^{(ab)}  \sum_{x',y',z',w'}^{(ab)} \E \, Q_a \pb{G_{\mu x}^{(ab)} h_{xb} h_{by} G_{yz}^{(ab)} h_{za} h_{a w} G_{w \mu}^{(ab)*}} \,
Q_b\pb{G_{\mu x'}^{(ab)} h_{x' b} h_{b y'} G_{y'z'}^{(ab)*} h_{z' a} h_{a w'} G_{w' \mu}^{(ab)*}}\,.
\end{equation}
(Here we also ignored a few special cases of coinciding indices when expanding both $a$ and $b$ in $G_{ab}$ using \eqref{res exp 2b}. As usual, the resulting terms are subleading and unimportant for this sketchy discussion.) The idea behind \eqref{simple vertex resolution} is to expand all of the randomness that depends on $a$ and $b$ explicitly (i.e.\ in entries of $H$), so that partial expectations may be explicitly taken. Note that all resolvent entries in \eqref{simple vertex resolution} are independent of $a$ and $b$. We may now take the expectation $\E$ in \eqref{simple vertex resolution}; more precisely, we reorganize \eqref{simple vertex resolution} as
\begin{equation} \label{simple vertex resolution organized}
\sum_{x,y,z,w}^{(ab)}  \sum_{x',y',z',w'}^{(ab)} \E \, G_{\mu x}^{(ab)} G_{yz}^{(ab)}G_{w \mu}^{(ab)*} G_{\mu x'}^{(ab)} G_{y'z'}^{(ab)*} G_{w' \mu}^{(ab)*} \, P_a \qB{Q_a(h_{za} h_{a w}) h_{z' a} h_{a w'}}  P_b \qB{h_{xb} h_{by} Q_b(h_{x' b} h_{b y'})}
\,.
\end{equation}
The two square brackets in \eqref{simple vertex resolution} may be computed explicitly. Since $\E h_{uv}=0$, each matrix entry $h_{uv}$ must (at least) be paired with another copy of the same factor or its conjugate $\bar h_{uv} = h_{vu}$. Assume first that we are dealing with a complex Hermitian band matrix (condition \eqref{CH}). In that case, each $h_{uv}$ must be paired with its conjugate $h_{vu}$ since $\E h_{uv}^2 = 0$. Of course, it may happen that more than two entries have coinciding indices, but this leads to a term that is subleading by a factor $M^{-1/2}$, and which we neglect here.
Thus, $h_{za}$ in \eqref{simple vertex resolution organized}
may be paired with $h_{aw}$ (resulting in $z=w$) or with  $h_{aw'}$ (resulting in $z = w'$). However, the
pairing $h_{za}$ with $h_{aw}$ gives a vanishing contribution owing to the presence of $Q_a$, since
\begin{equation*}
\E_{za} Q_a (h_{za}h_{az}) \;=\; 0
\end{equation*}
where $\E_{za}$ denotes partial expectation with respect to $h_{za}$. In other words,  $Q_a$ forbids the pairing of $h_{za}$ with $h_{aw}$, and similarly $Q_b$ the pairing of $h_{x' b}$ with $h_{b y'}$. Thus the leading order term resulting from the square brackets in \eqref{simple vertex resolution organized}, on which we focus here, is the pairing
\begin{equation} \label{simple pairing in case H}
s_{az} s_{aw} s_{bx} s_{by} \, \delta_{w z'} \delta_{z w'} \delta_{x'y} \delta_{x y'}\,.
\end{equation}
In the real symmetric case (condition \eqref{RS}), where $\E h_{uv}^2$ does not vanish, $h_{za}$ can also be paired with $h_{z' a}$. (Note that $Q_a$ still forbids the pairing of $h_{za}$ with $h_{aw}$.) This yields the three further allowed pairings
\begin{equation} \label{example pairings for S}
s_{az} s_{a w} s_{bx} s_{by} \delta_{z z'} \delta_{w w'} \delta_{x'y} \delta_{x y'} \,, \qquad s_{az} s_{aw} s_{bx} s_{by}  \delta_{w z'} \delta_{zw'} \delta_{x x'} \delta_{y y'}\,, \qquad s_{az} s_{a w} s_{bx} s_{by} \delta_{z z'} \delta_{w w'} \delta_{x x'} \delta_{y y'}\,.
\end{equation}
Assuming again condition \eqref{CH}, only \eqref{simple pairing in case H} contributes, and we get the expression (up to lower order error terms in $M^{-1/2}$)
\begin{multline} \label{resolution after pairing}
\sum_{x,y,z,w}^{(ab)} s_{bx} s_{by} s_{az} s_{aw}  \, \E \, G_{\mu x}^{(ab)} G_{yz}^{(ab)} G_{w \mu}^{(ab)*} G_{\mu y}^{(ab)} G_{xw}^{(ab)*} G_{z \mu}^{(ab)*}
\\
=\; \E \pBB{\sum_{x,w} s_{aw} s_{bx} G_{\mu x}^{(ab)} G_{xw}^{(ab)*} G_{w \mu}^{(ab)*}}
\pBB{\sum_{y,z} s_{by} s_{az}  G_{\mu y}^{(ab)} G_{yz}^{(ab)} G_{z \mu}^{(ab)*}}\,.
\end{multline}
Now each of the expressions in the parentheses is stochastically bounded by $\Psi^3 \Phi$. Indeed, an argument very similar to the proof of (B) above yields
\begin{equation} \label{most primitive chain}
\sum_{y} s_{by} G_{\mu y}^{(ab)} G_{yz}^{(ab)} \;\prec\; \Psi^2 \Phi\,.
\end{equation}
(The additional upper indices $(ab)$ are unimportant.) Thus, from the summation over $y$ in \eqref{resolution after pairing} we gain an additional factor $\Phi$ (and, similarly, from the summation over $w$). We therefore find that \eqref{monomial for simple resolution} is stochastically bounded by $\Psi^6 \Phi^2$, which was the claim of (C).

The use of graphs greatly clarifies the mechanism underlying the above sketch of the proof of (C). 
In order to depict vertex resolution, we need a graphical notation for edges associated with matrix entries $h_{uv}$ and $s_{ab}$; we represent the former using dotted lines and the latter using wiggly lines. See Figure \ref{figure: h and s}.
\begin{figure}[ht!]
\begin{center}
\includegraphics{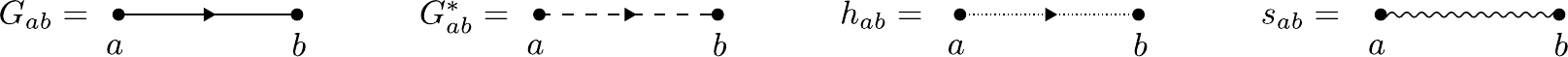}
\end{center}
\caption{The graphical notation for entries of $G$, $G^*$, $H$, and $S$. Since $S$ is symmetric, the edge associated with $s_{ab}$ is undirected. \label{figure: h and s}}
\end{figure}
As seen above, the starting point for the operation of vertex resolution is \eqref{simplified expression for variance}.
The expanded expression \eqref{simple vertex resolution} may be graphically represented as in Figure \ref{figure: detailed vertex resolution}. Thus, the vertex $a$ is ``resolved'' into four (i.e.\ the degree of $a$) new vertices, which are drawn in white and are connected to their parent vertex $a$ by dotted lines (corresponding to a matrix entry of $H$). White vertices are either \emph{incoming} or \emph{outgoing}, depending on the orientation of the dotted edge that joins them to their parent vertex $a$. Similarly, the vertex $b$ is resolved into four new vertices.
\begin{figure}[ht!]
\begin{center}
\includegraphics{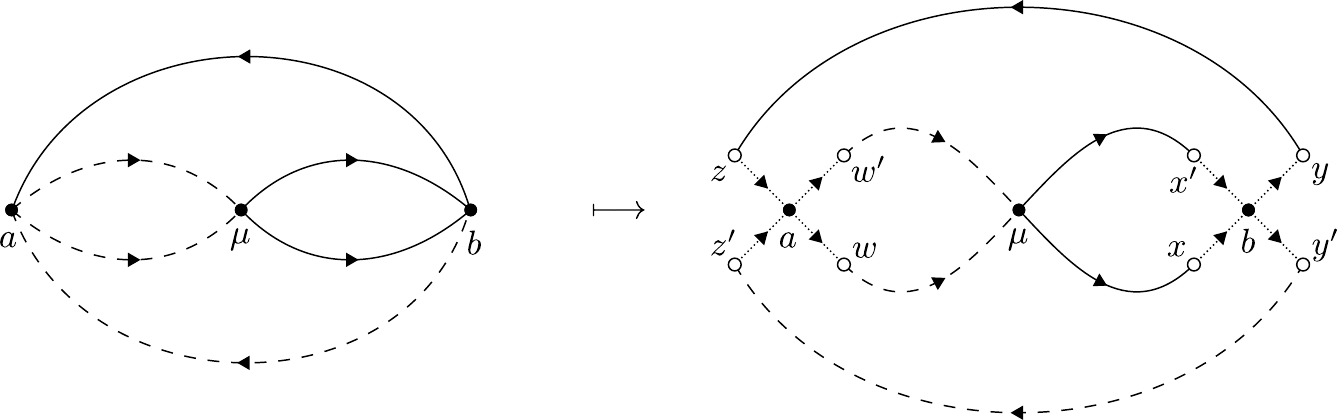}
\end{center}
\caption{Resolving the vertices $a$ and $b$. In accordance with \eqref{simple vertex resolution}, we ignore the loops associated with diagonal terms. (These lead to corrections that are higher order in $\Psi$.) \label{figure: detailed vertex resolution}}
\end{figure}
Note that each solid or dashed edge in the right-hand graph of Figure \ref{figure: detailed vertex resolution} represents a resolvent matrix entry that is independent of $a$ and $b$. The expression \eqref{resolution after pairing} was obtained from \eqref{simple vertex resolution} by computing the partial expectations $P_a$ and $P_b$ of the associated entries of $H$. Graphically, this amounts to a pairing of the white vertices surrounding each black parent vertex. (Note that the factors $Q$, which yielded constraints on the allowed pairings, are not visible in the graphs. This is not a problem, however, as the ensuing bounds will hold for all pairings, even if these restrictions are relaxed.) The pairing of two dotted lines gives rise to a wiggly line, in accordance with the identity $\E_a \abs{h_{ax}}^2 = s_{ax}$. See Figure \ref{figure: completed vertex resolution} for a graphical representation of the pairing in \eqref{resolution after pairing}. In Figure \ref{figure: completed vertex resolution} we represented the pairing \eqref{simple pairing in case H}, which is the only one in the complex Hermitian case \eqref{CH}. In this case, the orientation of the edges must be matched when pairing white vertices, i.e.\ an incoming white vertex can only be paired with an outgoing one. This is an immediate consequence of the condition $\E h_{ax}^2 = 0$, as explained after \eqref{simple pairing in case H}. In the real symmetric case \eqref{RS}, where $\E_a \abs{h_{ax}}^2 = \E h_{ax}^2 = s_{ax}$, the other pairings \eqref{example pairings for S} are also possible. Graphically, this means that, when pairing white vertices, there are no constraints on the orientation of the incident edges. In other words, the arrows on the dotted edges may be ignored. 
\begin{figure}[ht!]
\begin{center}
\includegraphics{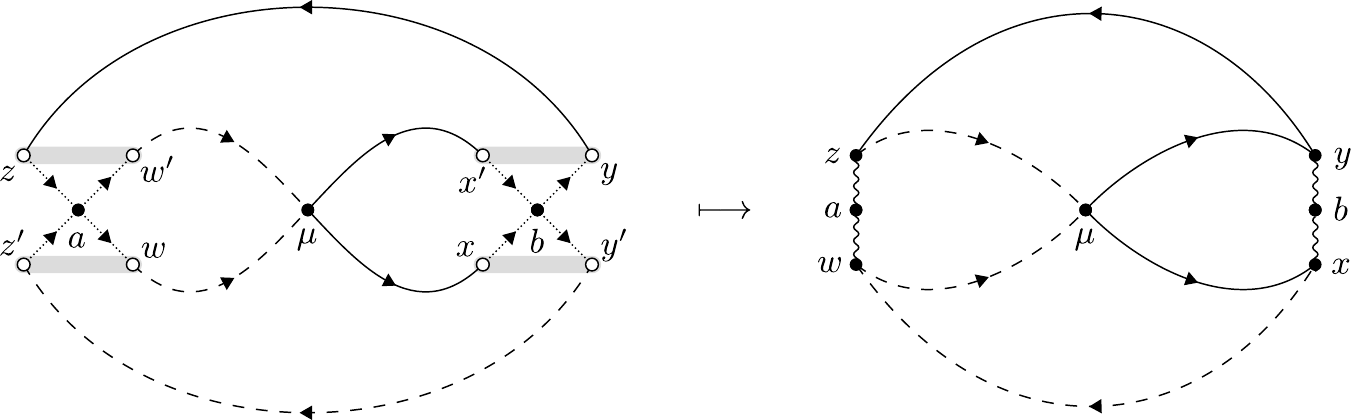}
\end{center}
\caption{Taking a pairing of the white vertices to get the completed resolution. \label{figure: completed vertex resolution}}
\end{figure}

The result of the vertex resolution is graphically evident when comparing the first graph in Figure \ref{figure: detailed vertex resolution} and second graph of Figure \ref{figure: completed vertex resolution}: the vertex $a$, of degree four, has been split (or ``resolved'') into two vertices of degree two. (The same happened for $b$). This \emph{resolution} also entails the creation of new summation indices, $z$ and $w$. Each one of them is connected to the original vertex $a$ by a factor $s_{az}$ (respectively $s_{aw}$), which implies that the summation over the larger family of summation indices is still performed with respect to a normalized weight. Generally, vertex resolution splits vertices of high degree into several vertices of degree two. The reason why this helps is that we can gain an extra factor $\Phi$ from any summation vertex of degree two whose incident edges are of the same ``colour'' (solid or dashed). We shall use the name \emph{marked vertex} (see Definition \ref{definition of marked vertex} below) to denote a vertex whose resolution yields at least one new summation vertex whose (two) incident edges are of the same colour. The mechanism behind the gain of a factor $\Phi$ from a newly created (via resolution) index is roughly the content of (B), and was used in \eqref{most primitive chain}. In our case, we gain from the summations over $y$ and $w$ (but not $z$ or $x$). Generally, the process of vertex resolution yields long ``chains'' (i.e.\ subgraphs whose vertices have degree two), each vertex of which yields an extra factor $\Phi$ provided both incident edges have the same colour. In fact, establishing such estimates for chains is an important step in our proof (see Proposition \ref{lemma: weak Z lemma} below). This concludes our overview of the proof of Proposition \ref{prop: examples}.

\section{General monomials and main result}\label{sec:multZ}

In this section we state the fluctuation averaging theorem in full generality. To that end, we introduce a general class of monomials which we shall average. We consider monomials in the variables
\begin{equation*}
\cal G_{ij}(z) \;\deq\; G_{ij}(z) - \delta_{ij} m(z)\,,
\end{equation*}
which yield a more consistent power counting for diagonal resolvent entries. Indeed, by definition $\abs{\cal G_{ij}} \leq \Lambda$ for all $i$ and $j$. As we saw in Section \ref{section: examples}, monomials in the resolvent entries are best described using graphs; see \eqref{example Z} and Figure \ref{figure: simple graph example}.

We may now define the graphs $\Delta$ used to describe monomials.

\begin{definition}[Admissible graphs] \label{definition: Delta}
\begin{enumerate}
\item
Let $V_s$ and $V_e$ be finite disjoint sets.
Let  $V \deq V_s \sqcup V_e$ be their disjoint union\footnote{Here, and throughout the following, we use the symbol $\sqcup$ to denote disjoint union.} 
and $E$ be a subset of the ordered pairs $V\times V$. 
The quadruple
\begin{equation*}
\Delta \;=\; (V_s, V_e, E, \xi)
\end{equation*}
is an \emph{admissible graph} if it 
is a directed, edge-coloured, multigraph with set of vertices $V$.
 The edges are ordered pairs of vertices with multiplicity, i.e.\ we allow loops and multiple edges.
We shall also use the notation $V_s \equiv V_s(\Delta)$, $V_e \equiv V_e(\Delta)$, and $E \equiv E(\Delta)$.

More formally, we can view $E(\Delta)$ as an arbitrary finite set
equipped with maps $\alpha, \beta \col E(\Delta) \to V(\Delta)$. Here $\alpha(e)$ and $\beta(e)$ represent the source and target vertices of the edge $e \in E(\Delta)$. The colouring $\xi \col E(\Delta) \to \{1, *\}$ is a mapping that assigns one of two ``colours'', $1$ or $*$, to each edge. If no confusion is possible with the multiplicity of an edge $e \in E(\Delta)$, we shall identify it with the ordered pair $(\alpha(e), \beta(e))$.
\item
We denote by $\fra Z$ the set of admissible graphs $\Delta$ on arbitrary $V_s$ and $V_e$.
\item
The \emph{degree} of $\Delta \in \fra Z$ is
\begin{equation*}
\deg(\Delta) \;\deq\; \abs{E(\Delta)}\,.
\end{equation*}
\end{enumerate}
\end{definition}

The set $V_s(\Delta)$ will label the family of summation indices ($(a,b)=(a_i)_{i\in V_s(\Delta)}$ in the example \eqref{example Z-lemma object}), and $V_e(\Delta)$ the set of external
 indices ($(\mu, \nu)=(\mu_i)_{i\in V_e(\Delta)}$ in the example \eqref{example Z-lemma object}). We use the notation 
\begin{equation}\label{indexsplitting}
\f u \;=\; (\f a, \f \mu) \,, \qquad \f a \;=\; (a_i)_{i \in V_s(\Delta)}\,, \qquad \f \mu \;=\; (\mu_i)_{i \in V_e(\Delta)}
\end{equation}
for the matrix indices. Generally, we try to use Latin letters $a,b,c,d,x,y,z,\dots$ for summation indices and Greek letters $\mu,\nu\,\dots$ for external indices.

Although our statements and proofs hold for any admissible graph $\Delta$, in order to
avoid trivial cases in our applications we shall always consider graphs without isolated
vertices and with the property that each edge is incident to 
at least one vertex $V_s(\Delta)$, i.e.\ every resolvent entry contains at least
one summation index.

Next, we introduce the monomials in $(\cal G_{xy})$ whose average we shall estimate.

\begin{definition}[Monomials] \label{definition: fra Z} 
Let $\Delta \in \fra Z$ be an admissible graph and let $\f \mu \in \{1, \dots, N\}^{V_e(\Delta)}$ be a collection of external indices.
We define the monomial
\begin{equation} \label{definition of Z}
\cal Z_{\f a} \;\equiv\; \cal Z_{\f a}^{\f \mu}(\Delta) \;\deq\; 
\prod_{e \in E(\Delta)} \cal G_{u_{\alpha(e)} u_{\beta(e)}}^{\xi_e}\,
\end{equation}
which is regarded as a function of the summation indices $\f a$, recalling
the splitting of the indices \eqref{indexsplitting}. We also denote by
\begin{equation*}
\cal Z \;=\; \pb{\cal Z_{\f a}^{\f \mu}(\Delta) \col \f a \in \{1, \dots, N\}^{V_s(\Delta)}}
\end{equation*}
the family of monomials associated with $(\Delta, \f \mu)$ by \eqref{definition of Z}, and say that $\Delta$ \emph{encodes} the monomial $\cal Z_{\f a}$.
\end{definition}

Note that $\deg(\Delta)$ is the degree of the monomial $\cal Z_{\f a}$ encoded by $\Delta$. Throughout the following we shall frequently drop the explicit dependence of $\cal Z_{\f a}$ on $\f \mu$ and $\Delta$.

The averaging over $\f a$ will be performed with respect to a weight $w(\f a)$. In the example \eqref{example Z-lemma object}, this weight was $w(a,b) = s_{\mu a} s_{\rho b}$. A typical example of a weight is
\begin{equation} \label{example weight}
w(a,b,c) \;=\; \frac{1}{N} \sum_{d} s_{\mu d} s_{d b} s_{bc}\, \qquad \mbox{for} \quad \f a = (a,b,c)\,.
\end{equation}
In order to define a general class of weights, the following notion of partitioning of summation indices is helpful.

\begin{definition}[Partition of indices] \label{def: partition}
Let $I$ be a finite index set. For $\f a  = (a_i)_{i \in I} \in \{1, \dots, N\}^I$ we denote by $\cal P(\f a)$ the partition of $I$ defined by 
 the equivalence relation $k \sim l$ if and only if $a_k = a_l$.
\end{definition}
Generally, we consider weights satisfying the following definition; when reading it, it is good to keep examples of the type \eqref{example weight} in mind. 

\begin{definition}[Weights] \label{definion: w}
A map $w \col \{1, \dots, N\}^{V_s(\Delta)} \to [0,1]$ is a \emph{weight adapted to} $\Delta \in \fra Z$ if it satisfies the following condition. Let  $V_s(\Delta) = I \sqcup J$ be a (possibly trivial) partition of $V_s(\Delta)$ into two disjoint subsets, inducing a splitting $\f a =(\f a_I, \f a_J)$ of the summation indices.
Then we require that, for any  partition $P$ of $J$, we have
\begin{equation} \label{definition of weight}
\max_{\f a_I} \sum_{\f a_J} \indb{\cal P(\f a_J) = P} \, w(\f a_I, \f a_J) \;\leq\; M^{\abs{P} - \abs{V_s(\Delta)}}\,,
\end{equation}
where $\abs{P}$ denotes the number of blocks in $P$.
\end{definition}

The interpretation of \eqref{definition of weight} is that the left-hand side of \eqref{definition of weight} has $\abs{P}$ free summation indices; the remaining summation indices have been either frozen (i.e.\ they belong to $\f a_I$) or merged with others (i.e.\ they belong to a nontrivial block of $P$). Then \eqref{definition of weight} simply states that each suppressed summation yields a factor $M^{-1}$.
In particular, with $J=V_s(\Delta)$ and the trivial atomic partition $P$ we have
\begin{equation*}
\sum_{\f a} w(\f a) \;\leq\; 1\,,
\end{equation*}
i.e.\ the total sum of all weights is always bounded by one.

When estimating averages such as \eqref{example Z-lemma object}, we shall always impose that all indices that have distinct names also have distinct values. In the case that two indices have the same value, we give them the same name. Thus, for example we write
\begin{multline*}
\frac{1}{N^2} \sum_{a,b} G_{\mu a} G_{ab} G_{b \mu}
\\
=\; \frac{1}{N^2}
\sum_{a,b}^{(\mu)*} G_{\mu a} G_{ab} G_{b \mu} + \frac{1}{N^2} \sum_{a}^{(\mu)} G_{\mu a} G_{aa} G_{a \mu} + \frac{1}{N^2} \sum_a^{(\mu)} G_{\mu a} G_{a\mu} G_{\mu \mu} + \frac{1}{N^2} \sum_b^{(\mu)} G_{\mu \mu} G_{\mu b} G_{b \mu} + \frac{1}{N^2} G_{\mu \mu}^3\,,
\end{multline*}
where a star on top of a summation means that all summation indices are constrained to be distinct. (Recall also the notation $\sum^{(S)}$ for $S \subset \{1, \dots, N\}$ from Definition \ref{def: minors}.)

We may now define our central quantity. Let $\Delta \in \fra Z$ and $\f \mu \in \{1, \dots, N\}^{V_e(\Delta)}$ be a collection of external indices. Let $F \subset V_s(\Delta)$ and $w$ be a weight adapted to $\Delta$. We define
\begin{equation} \label{general sum of Gs}
X_F^w(\Delta) \;\equiv\; X_F^{w, \f \mu}(\Delta) \;\deq\; \sum_{\f a}^{(\f \mu)*} w(\f a) \, \qBB{\prod_{i \in F} Q_{a_i}} \cal Z_{\f a}^{\f \mu}(\Delta)\,.
\end{equation}
Thus, $F$ denotes the set of summation indices that come with an operator $Q$. As explained above, the symbol $(\f \mu)$ on top of them sum means that $a_i \neq \mu_j$ for all $i \in V_s(\Delta)$ and $j \in V_e(\Delta)$, and the star means that $a_i \neq a_j$ for all distinct $i,j \in V_s(\Delta)$. Throughout the following, we shall frequently drop the explicit dependence of $X_F^{w, \f \mu}(\Delta)$ on $\f \mu$.

\begin{remark}
In \eqref{general sum of Gs} each operator $Q$ acts on all resolvent entries in $\cal Z_{\f a}$. We make this choice to simplify the presentation; also, this is sufficient for all of our current applications. However, our results may be easily extended to more complicated quantities, in which each $Q$ acts only on a subset of the resolvent entries in $\cal Z_{\f a}$. Thus, in general, there a resolvent entry is either \emph{outside} or \emph{inside} $Q_{a_i}$, for each $i \in F$. We require that each resolvent entry outside $Q_{a_i}$ have no index $a_i$, and at least one resolvent entry inside $Q_{a_i}$ have an index $a_i$. Then our proof carries over with merely cosmetic changes. For example, expressions such as
\begin{equation*}
\sum_{a,b,c,d}^{(\mu \nu)*} s_{\mu a} s_{\rho b} s_{b c} \, Q_a \pB{G_{\mu a} Q_b(G_{ab} G^*_{b \nu}) Q_c(G^*_{ac} G_{cd}) G_{d \mu}}
\end{equation*}
may be estimated in this fashion.
\end{remark}

From Lemmas \ref{lemma: basic properties of prec} and \ref{lemma: rough bounds on G}, we find the trivial bound
\begin{equation}\label{trivial bound}
X^w_F(\Delta) \;\prec\; \Psi^{\deg(\Delta)}
\end{equation}
for any adapted weight $w$, provided that $\Lambda \prec \Psi$. We call \eqref{trivial bound} trivial because we also have the bound
\begin{equation*}
\cal Z_{\f a}^{\f \mu}(\Delta) \;\prec\; \Psi^{\deg(\Delta)}\,.
\end{equation*}
Hence the estimate \eqref{trivial bound} has not been improved by the averaging over $\f a$.

Next, we define indices which count the gain in the size of $X_F^w(\Delta)$ resulting from the averaging over $\f a$ and from the factors $Q$.
\begin{definition} \label{def: incidence indices}
Let $\Delta$ be an edge-coloured graph as in Definition \ref{definition: Delta}. For $i \in V(\Delta)$ we set
\begin{align*}
\nu_i(\Delta) &\;\deq\; \sum_{(j,k) \in E(\Delta)} \ind{\xi_{(j,k)} = 1} \qB{\ind{i = j} + \ind{i = k}}\,,
\\
\nu_i^*(\Delta) &\;\deq\; \sum_{(j,k) \in E(\Delta)} \ind{\xi_{(j,k)} = *} \qB{\ind{i = j} + \ind{i = k}}\,.
\end{align*}
Informally, $\nu_i(\Delta)$ is the number of legs of colour $1$ incident to $i$, and $\nu_i^*(\Delta)$ the number of legs of colour $*$ incident to $i$.
\end{definition}

We shall use $\deg(i) \equiv \deg_\Delta(i)$ to denote the degree of the vertex $i \in V(\Delta)$. It is sometimes important to emphasize that this degree is computed with respect to the graph $\Delta$, which we indicate using the subscript\footnote{Of course, $\deg_\Delta$ is not the same as $\deg(\Delta)$. In fact, we have $\deg(\Delta) = \frac{1}{2} \sum_{i \in V(\Delta)} \deg_\Delta(i)$.} $\Delta$. By definition, $\deg_\Delta(i)$ is the number of legs incident to $i$, i.e.\ a loop at $i$ counts twice. In particular, $\deg_\Delta(i) = \nu_i(\Delta) + \nu_i^*(\Delta)$.

In terms of the monomials $\cal Z$ encoded by $\Delta$, the index $\nu_i(\Delta)$ (respectively $\nu_i^*(\Delta)$) is the number of resolvent entries of $\cal G$ (respectively of $\ad {\cal G}$) in which the index $a_i$ appears. (Note that if the index $a_i$ appears twice in a resolvent entry, this entry is counted twice.)

\begin{definition}[Charged vertex] \label{def: charged}
We call a summation vertex $i \in V_s(\Delta)$ \emph{charged} if either
\begin{enumerate}
\item
$i \notin F$ and $\nu_i \neq \nu_i^*$, or
\item
$i \in F$ and $\abs{\nu_i - \nu_i^*} \neq 2$.
\end{enumerate}
We denote by $V_c(\Delta) \subset V_s(\Delta)$ the set of charged vertices.
\end{definition}

We may now state our main result. 

\begin{theorem}[Averaging theorem] \label{theorem: Z lemma}
Suppose that $\Lambda \prec \Psi$ for some admissible control parameter $\Psi$. Let $\Delta \in \fra Z$ (recall Definitions \ref{definition: Delta} and \ref{definition: fra Z}) and $F \subset V_s(\Delta)$.
Then
\begin{equation} \label{main Z lemma estimate}
X_F^{w, \f \mu}(\Delta) \;\prec\; \Psi^{\deg(\Delta) + \abs{F}} \, \Phi^{\abs{V_c(\Delta)}}
\end{equation}
for any $\f \mu$ and weight $w$ adapted to $\Delta$ (recall Definition \ref{definion: w}).
\end{theorem}

Thus, Theorem \ref{theorem: Z lemma} states that we gain a factor $\Psi$ from each $Q$ and a factor $\Phi$ 
from each charged vertex. The rationale behind the name ``charged'' is that, in the vertex resolution process from the proof of Theorem \ref{theorem: Z lemma}, a charged vertex gives rise, in leading order, to a collection vertices of degree two, at least one of which will be a \emph{chain vertex} (see Definition \ref{def: chains}) and hence yield a factor $\Phi$ using the a priori bounds of Section \ref{sec: chains}.

\begin{remark}\label{remark simple bound}
The right-hand side of \eqref{main Z lemma estimate} can be estimated from above by
\begin{equation*}
(\Psi + M^{-1/4})^{\deg(\Delta) + \abs{F} + \abs{V_c(\Delta)}}\,,
\end{equation*}
which gives a simple power counting in terms of the quantity $\Psi + M^{-1/4}$. From each summation index $a_i$ without an associated $Q_{a_i}$ we gain a factor $\Psi + M^{-1/4}$ if $\nu_i \neq \nu_i^*$. If there is a $Q_{a_i}$ then we gain at least a factor
 $\Psi + M^{-1/4}$, and, provided that $\abs{\nu_i - \nu_i^*} \neq 2$, one additional factor $\Psi + M^{-1/4}$.
Note that we gain at most two additional factors $\Psi + M^{-1/4}$ from each summation index.
\end{remark}

\begin{remark} \label{remark: bad example}
As explained after \eqref{GGwithQ}, the additional term $M^{-1/2}\Psi^{-1}$ in the definition of $\Phi$
is a (necessary) technical nuisance and should be thought of as a lower
order term in typical applications. In general, however, it cannot be eliminated,
and Theorem \ref{theorem: Z lemma} cannot be formulated in terms of powers of $\Psi$ alone. This may be seen for instance from the variance calculation of the quantity $\frac{1}{N} \sum_{a}^{(\mu)} Q_a(G_{\mu a} G_{a \mu}^*)$. Indeed, as is apparent from \eqref{variance for coinciding indices}, the term arising from $a = b$ is of order $N^{-1} \Psi^4$, which is in general not bounded by $\Psi^8$.
\end{remark}

\begin{remark} \label{remark: finite moments}
The requirement that \eqref{finite moments} hold for all $p$ can be easily relaxed. Indeed, Theorem \ref{theorem: Z lemma} has the following variant. Fix $\epsilon > 0$ and $D > 0$. Then there exists a $p(\epsilon, D) \in \N$ such that the following holds. Suppose that the hypotheses of Theorem \ref{theorem: Z lemma} hold, and that \eqref{finite moments} holds for $p(\epsilon, D)$. Then
\begin{equation*}
\P\qB{\abs{X_F^{w, \f \mu}(\Delta)} > N^\epsilon \, \Psi^{\deg(\Delta) + \abs{F}} \, \Phi^{\abs{V_c(\Delta)}}} \;\leq\; N^{-D}\,,
\end{equation*}
for all $z \in \f S$, all $\f \mu$, and all weights $w$ adapted to $\Delta$.

This variant is an immediate consequence of the proof of Theorem \ref{theorem: Z lemma}, using the observation that, for any fixed $\epsilon$ and $D$, the estimate on $X_F^w(\Delta)$ consists of a finite number of steps $s$, each of them using a bound on $\E \abs{\zeta_{ij}}^{p_s}$ for some finite $p_s$. As $\epsilon \to 0$ or $D \to \infty$, the number of these steps tends to infinity. Moreover, as the step index $s$ tends to infinity, the exponent $p_s$ in $\E \abs{\zeta_{ij}}^{p_s}$ also tends to infinity.
\end{remark}

\begin{remark} \label{rem: G in denominator}
Our result applies verbatim if (some or all) diagonal entries of the form 
$\cG_{ii}=G_{ii} - m$  in the monomial \eqref{definition of Z}
are replaced by $1 /  G_{ii} - 1 / m$. (This would be a mere notational complication in the statement of Theorem \ref{theorem: Z lemma}). 
After a little algebra (multiplying out a product of terms of the form $1 /  G_{ii} - 1 / m$), we consequently find that our result applies to monomials divided by diagonal entries $G_{ii}$, i.e.\ expressions of the form
\begin{equation*}
\frac{\cG_{xy}\cG_{uv} \cG_{wz} \cdots}{G_{aa} G_{bb} G_{cc} \cdots}\,,
\end{equation*}
where the indices can be either summation or external indices.
This extension may be proved in two ways.

The first way is to observe that if we replace the identity
\begin{equation*}
G_{ii} - m \;=\; \frac{1}{1/m - \pb{- h_{ii} + Z_i + U_i^{(i)}}} - m \;=\; m^2 \pb{- h_{ii} + Z_i + U_i^{(i)}} + m^3 \pb{- h_{ii} + Z_i + U_i^{(i)}}^2 + \cdots\,,
\end{equation*}
used in our proof by \eqref{res exp 2c} for the quantity $1/G_{ii} - 1/m$, the proof of Theorem \ref{theorem: Z lemma} carries over unchanged.

The second way is to write
\begin{equation*}
\frac{1}{G_{ii}} - \frac{1}{m} \;=\; \frac{m - G_{ii}}{m^2} + \frac{(m - G_{ii})^2}{m^3} + \frac{(m - G_{ii})^3}{m^3 G_{ii}}\,.
\end{equation*}
This induces a splitting of $\cal Z$ into three parts, which are treated separately. It is a simple matter to check that Theorem \ref{theorem: Z lemma} may be applied to the first two parts. The third part is treated trivially, by freezing the index $i$; in this case we already get a factor $\Psi^3$ from the index $i$, and hence the averaging effect of the summation over $i$ is not needed, since we already gained the maximal two 
additional factors of $\Psi$ from $i$.
\end{remark}

\begin{remark} \label{rem: relaxing symmetry}
As in Section~\ref{section: examples}, in our proofs we shall assume that either \eqref{RS} or \eqref{CH} holds (see Section~\ref{sect: sketch of proof of examples}).
We impose these conditions in order to simplify the derivation and analysis of self-consistent equations such as the ones in Sections \ref{sec: proof of (B)} and \ref{section: Z4}. Without them, however, our core argument remains unchanged. For instance, when estimating $\sum_a s_{ba} G_{\mu a} G_{\mu a}$, we instead consider the quantity $V_a \deq P_a G_{\mu a} G_{\mu a}$. Using \eqref{res exp 2b}, we may do a calculation similar to the one following \eqref{sce}, and get a self-consistent equation for $V_a$. Solving the self-consistent equation entails the analysis of the Hermitian operator $R = (r_{ij})$ where $r_{ij} \deq \E h_{ij}^2$. Using $\abs{r_{ij}} \leq s_{ij}$, the spectral analysis from the end of Section \ref{sec: chains with no Q} and Appendix \ref{section: proof of res id} carries over with minor modifications. We omit the extraneous details of this generalization.
\end{remark}

\begin{remark} \label{remark: 1/G for warm-up}
In \cite[Lemma 5.2]{EYY2}, a fluctuation averaging theorem of the form
\begin{equation} \label{old Z-lemma}
\frac{1}{N} \sum_i Q_i \sum_{k,l}^{(i)} h_{ik} G_{kl}^{(i)} h_{li} \;\prec\; \Psi^2
\end{equation}
was proved. This result was further generalized in \cite{EYY3, EKYY1, PY}. The estimate \eqref{old Z-lemma} also follows  from  Theorem \ref{theorem: Z lemma}. To see this, we use Schur's formula \eqref{schur} to get
\begin{equation} \label{link to old Z-lemma}
\frac{1}{N} \sum_i Q_i \frac{1}{G_{ii}} \;=\; \frac{1}{N} \sum_i h_{ii} - \frac{1}{N} \sum_i Q_i \sum_{k,l}^{(i)} h_{ik} G_{kl}^{(i)} h_{li}\,.
\end{equation}
The first term on the right-hand side of \eqref{link to old Z-lemma} is easily proved to be stochastically bounded by $N^{-1} \leq \Psi^2$. The second term on the right-hand side of \eqref{link to old Z-lemma} is the left-hand side of \eqref{old Z-lemma}. Moreover, the left-hand side of \eqref{link to old Z-lemma} is stochastically bounded by $\Psi^2$, as follows from Theorem \ref{theorem: Z lemma}; see Remark \ref{rem: G in denominator}. In fact, the left-hand side of \eqref{link to old Z-lemma} may be estimated using the much simpler Proposition \ref{prop: warm-up} (whose proof trivially holds for expressions like the one the left-hand side of \eqref{link to old Z-lemma}). In particular, Proposition \ref{prop: warm-up} and this remark provide 
a simpler proof than \cite{EYY2, EYY3, EKYY1, PY} of the previously known estimate \eqref{old Z-lemma}.
\end{remark}

Theorem \ref{theorem: Z lemma} has the following, simpler, variant in which the averaging with respect to a weight $w$ is replaced with partial expectation. 
\begin{theorem}[Averaging using partial expectation] \label{theorem: Z lemma variant}
Suppose that $\Lambda \prec \Psi$ for some admissible control parameter $\Psi$. Let $\Delta \in \fra Z$. and $F = \emptyset$.
Then
\begin{equation} \label{main Z lemma estimate variant}
\prod_{a \in \f a} P_a \, \cal Z_{\f a}^{\f \mu}(\Delta) \;\prec\; \Psi^{\deg(\Delta)} \, \Phi^{\abs{V_c(\Delta)}}
\end{equation}
for all $\f a$ and $\f \mu$ such that all indices of the collection $(\f a, \f \mu)$ are distinct.
\end{theorem}

Thus in Theorem \ref{theorem: Z lemma variant} we set $F = \emptyset$, i.e.\ there are no factors $Q$, whose presence would be nonsensical because the identity $P_a Q_a = 0$ implies that the partial expectation of any monomial preceded by a factor $Q$ vanishes. The condition  $F = \emptyset$ is still used indirectly in the theorem since the definition of $V_c(\Delta)$ depends on $F$.

\begin{remark}
It is possible to combine Theorems \ref{theorem: Z lemma} and \ref{theorem: Z lemma variant} by splitting $\f a = (\f a', \f a'')$, and averaging over $\f a'$ with respect to a weights $w(\f a')$
and taking the partial expectation $\prod_{a \in \f a''} P_a$ over $\f a''$. We omit the details.
\end{remark}

\begin{remark}
Remarks \ref{remark simple bound} -- \ref{rem: relaxing symmetry} also apply to Theorem \ref{theorem: Z lemma variant} with the obvious modifications.
\end{remark}

\section{Outline of proof} \label{section: Z1}

We now outline the strategy behind the proof of Theorem \ref{theorem: Z lemma}. The first part of the proof relies on an inductive argument to prove the claim of Theorem \ref{theorem: Z lemma} for a special class of $\Delta$'s (the \emph{chains}) that encode monomials containing only factors $\cal G$ and not $\cal G^*$ (or the other way around). 
These $\Delta$'s act as building blocks which are used to estimate the error terms arising in the estimate of arbitrary $\Delta$'s, in the second part of the proof. The need to have a priori bounds on chains was already hinted at in Section \ref{sect: sketch of proof of examples}. Indeed, the estimate \eqref{most primitive chain} is the simplest prototype of a chain estimate, and was used to estimate quantities arising from the process of vertex resolution. This is in fact a general phenomenon: a priori bounds on chains will be used used in combination with vertex resolution.

\begin{definition}[Chains] \label{def: chains}
Let $\Delta \in \fra Z$.
\begin{enumerate}
\item
We call a vertex $i \in V_s(\Delta)$ a \emph{chain vertex} if $i$ is not adjacent to itself, $i$ has degree two, and both incident edges have the same colour.
We denote by $c(\Delta)$ the number of chain vertices in $\Delta$.
\item
We call $\Delta$ an \emph{open (undirected) chain} if all vertices $i \in V_s(\Delta)$ are chain vertices, $\abs{V_e(\Delta)} = 2$, and $\deg(i) = 1$ for both $i \in V_e(\Delta)$.
\item
We call $\Delta$ a \emph{closed (undirected) chain} if all vertices $i \in V_s(\Delta)$ are chain vertices, $\abs{V_e(\Delta)} \leq 1$, and $\deg(i) = 2$ for $i \in V_e(\Delta)$.
\item
A chain vertex $i \in V_s(\Delta)$ is \emph{directed} if one incident edge is incoming and the other outgoing. A chain is \emph{directed} if every $i \in V_s(\Delta)$ is directed.
\end{enumerate}
\end{definition}

Figure \ref{figure: chains} gives a few examples of chains. The notion of a directed chain will be used in the complex Hermitian case \eqref{CH}, in which all chains that arise in our proof will be directed. In the real symmetric case \eqref{RS}, there is no such restriction.

\begin{figure}[ht!]
\begin{center}
\includegraphics{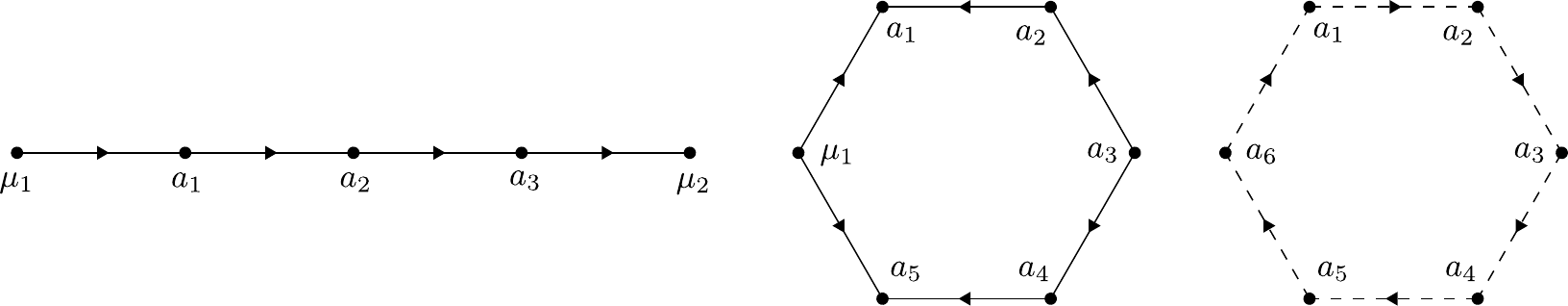}
\end{center}
\caption{From left to right: an open directed chain, a closed undirected chain with one external vertex, a closed directed chain with no external vertices. \label{figure: chains}}
\end{figure}

If $\Delta$ is a chain then by definition $X^w_F(\Delta)$ contains no diagonal entries $\cal G_{ii}$. Since $\cal G_{ij} = G_{ij}$ for $i \neq j$, we may (and shall) therefore replace all entries of $\cal G$ with entries of $G$ when $\Delta$ is a chain.

Chains are useful in combination with the following family of special weights.
\begin{definition}[Chain weights] \label{def: chain weight}
Let $n \in \N$. For any fixed $\f b = (b_1, \dots, b_n)$ define the weight
\begin{equation} \label{chain weight}
w_{\f b}(\f a) \;\equiv\; w(\f a) \;\deq\; s_{a_1 b_1} \cdots s_{a_n b_n}\,.
\end{equation}
We call weights of the form \eqref{chain weight} \emph{chain weights}.
\end{definition}

Using \eqref{s leq W}, it is easy to check that a chain weight from Definition \ref{def: chain weight} is a weight in the sense of Definition \ref{definion: w}.
The role of chains is highlighted by the two following facts.

\begin{itemize}
\item
If $\Delta$ is a chain and $w_{\f b}$ is an adapted chain weight, then the family $\pb{\sum_{\f a} w_{\f b}(\f a)\cal Z_{\f a}(\Delta)}_{b_1}$ for fixed $(b_2, \dots, b_{n - 1})$ satisfies a stable self-consistent equation; See Step $\rm I_2$ below.
\item
Proving Theorem \ref{theorem: Z lemma} (in fact, a weaker version given in Proposition \ref{lemma: weak Z lemma} below) for a chain $\Delta$ and an adapted chain weight is a key tool for proving Theorem \ref{theorem: Z lemma} for arbitrary $\Delta$.
\end{itemize}

\begin{proposition}\label{lemma: weak Z lemma}
Suppose that $\Lambda \prec \Psi$ for some admissible control parameter $\Psi$, and recall the definition \eqref{definition of Phi} of $\Phi$. Let $\Delta$ be a chain, $w$ an adapted chain weight, and $F \subset V_s(\Delta)$. 
Then we have
\begin{equation} \label{bound for closed chains}
X^w_F(\Delta) \;\prec\; \Psi^{\deg(\Delta) + \abs{F}} \Phi^{c(\Delta) - \abs{F}}
\end{equation}
for any $\f \mu$ and adapted chain weight $w$.
\end{proposition}

As an a priori bound in Sections \ref{section: Z5} and \ref{section: simplifications}, we shall always use Proposition \ref{lemma: weak Z lemma} with $F = \emptyset$.
The statement of Proposition \ref{lemma: weak Z lemma} for $F = \emptyset$ may be summarized by saying that from each chain vertex we gain a factor $\Phi$ (as compared to the trivial bound \eqref{trivial bound}). 

Next, we outline the proof Theorem \ref{theorem: Z lemma}. The argument consists of two main steps: establishing a priori bounds on chains (i.e.\ proving Proposition \ref{lemma: weak Z lemma}) and proving Theorem \ref{theorem: Z lemma} using Proposition \ref{lemma: weak Z lemma} as input.

Proposition \ref{lemma: weak Z lemma} is proved first for open chains, using a two-step induction. The induction parameter is the length of the chain $\ell \deq \deg(\Delta)$. The induction is started at $\ell = 1$, and consists of two steps, $\rm I_1$ and $\rm I_2$. It may be summarized in the form
\begin{equation*}
(\ell = 1\,,\, F = \emptyset)
\xrightarrow{\rm I_1}
(\ell = 2\,,\, F \neq \emptyset)
\xrightarrow{\rm I_2}
(\ell = 2\,,\, F = \emptyset)
\xrightarrow{\rm I_1}
(\ell = 3\,,\, F \neq \emptyset)
\xrightarrow{\rm I_2}
(\ell = 3\,,\, F = \emptyset)
\xrightarrow{\rm I_1} \cdots\,.
\end{equation*}
What follows is a sketch of steps $\rm I_1$ and $\rm I_2$.
\begin{description}
\item[Step $\rm I_1$.]
The input for Step $\rm I_2$ is the claim of Proposition \ref{lemma: weak Z lemma} with $F = \emptyset$, for all open chains $\Delta'$ satisfying $\deg(\Delta') < \ell$. Using a high moment expansion, we estimate $X_F^w(\Delta)$, where $\Delta$ is an open chain, $\deg(\Delta) = \ell$, and $F \neq \emptyset$. The details are carried out in Section \ref{section: Z4}.

\item[Step $\rm I_2$.]
We fix an open chain $\Delta$ and prove the claim of Proposition \ref{lemma: weak Z lemma} for $F = \emptyset$, under the assumption that the claim of Proposition \ref{lemma: weak Z lemma} has been established for
\begin{enumerate}
\item
$\Delta$ with $F \neq \emptyset$;
\item
all open chains $\Delta'$ satisfying $\deg(\Delta') < \deg(\Delta)$ with $F = \emptyset$.
\end{enumerate}
The proof is based on a self-consistent equation for the family $\pb{\sum_{\f a} w_{\f b}(\f a)\cal Z_{\f a}}_{b_1}$ for fixed $b_2, \dots, b_n$. This self-consistent equation will be stable provided $E = \re z$ lies away from the spectral edges $\pm 2$. This stability is ensured by the fact that $\cal Z$ only contains factors $G$ and not $G^*$. The details are carried out in Section \ref{sec: chains with no Q}.
\end{description}

The induction is started by noting that Proposition \ref{lemma: weak Z lemma} holds trivially for the open chain of length 1 (which has no chain vertex), encoding the monomial $G_{\mu \nu} \prec \Psi$. After Steps $\rm I_1$ and $\rm I_2$ are complete, the induction argument outlined above completes the proof of Proposition \ref{lemma: weak Z lemma} for open chains. The proof for closed chains is almost identical, except that no induction is needed; the only required assumption is that Proposition \ref{lemma: weak Z lemma} hold for \emph{open} chains of arbitrary degree.

Once Proposition  \ref{lemma: weak Z lemma} has been proved, we use it as input to prove Theorem \ref{theorem: Z lemma} for a general $\Delta \in \fra Z$. Similarly to Step $\rm I_2$, we use a high-moment expansion. The estimates are considerably more involved than in Step $\rm I_2$, however. (In the language of Sections \ref{sec: sketch of resolution} and \ref{section: Z5}, we use vertex resolution to gain extra powers of $\Psi$ from the charged vertices.) The details are carried out in Sections \ref{section: Z5} -- \ref{section: simplifications}.

We record the following guiding principle for the entire proof of Theorem \ref{theorem: Z lemma}. 
It is a \emph{basic power counting} that can be summarized as follows.
The size of $X_F^w(\Delta)$ is given by a product of three main ingredients:
\begin{itemize}
\item[(a)]
The naive size $\Psi^{\deg(\Delta)}$, which is simply the number of entries of $\cal G$ in $X_F^w(\Delta)$ (obtained by a trivial power counting and $\Lambda \prec \Psi$).
\item[(b)]
The smallness arising from $F$, i.e.\ $\Psi^{\abs{F}}$ (obtained from the linking imposed by the factors $Q$).
\item[(c)]
The smallness arising from the charged vertices, i.e.\ $\Phi^{\abs{V_c(\Delta)}}$ (obtained from vertex resolution and the a priori bounds of Proposition \ref{lemma: weak Z lemma} applied to chain vertices).
\end{itemize}
We shall frequently refer to the factors $\Psi^{\abs{F}}$ and $\Phi^{\abs{V_c(\Delta)}}$ from (b) and (c) as \emph{gain} over the naive size $\Psi^{\deg(\Delta)}$.
It is very important for the whole proof that the mechanism of this gain is \emph{local} in the graph,  i.e.\ operates on the level of individual vertices. Each factor gained in the case (c) can be associated with a charged vertex. In the case (b), a linking results in an additional
edge adjacent to the vertex on which a linking was performed. There will be some technical complications which somewhat obscure this picture, such as 
occasionally coinciding indices. We shall always analyse these exceptional situations by comparing them to the basic power counting dictated by the generic situation.  We remark that these ``exceptional'' situations sometimes in fact lead to leading-order error terms, which is for instance the reason why the parameter $\Phi$ cannot in general be replaced with $\Psi$ in \eqref{main Z lemma estimate}.

Figure \ref{figure: roadmap} contains a diagram summarizing all key steps of the proof.
\begin{figure}[ht!]
\begin{center}
\includegraphics{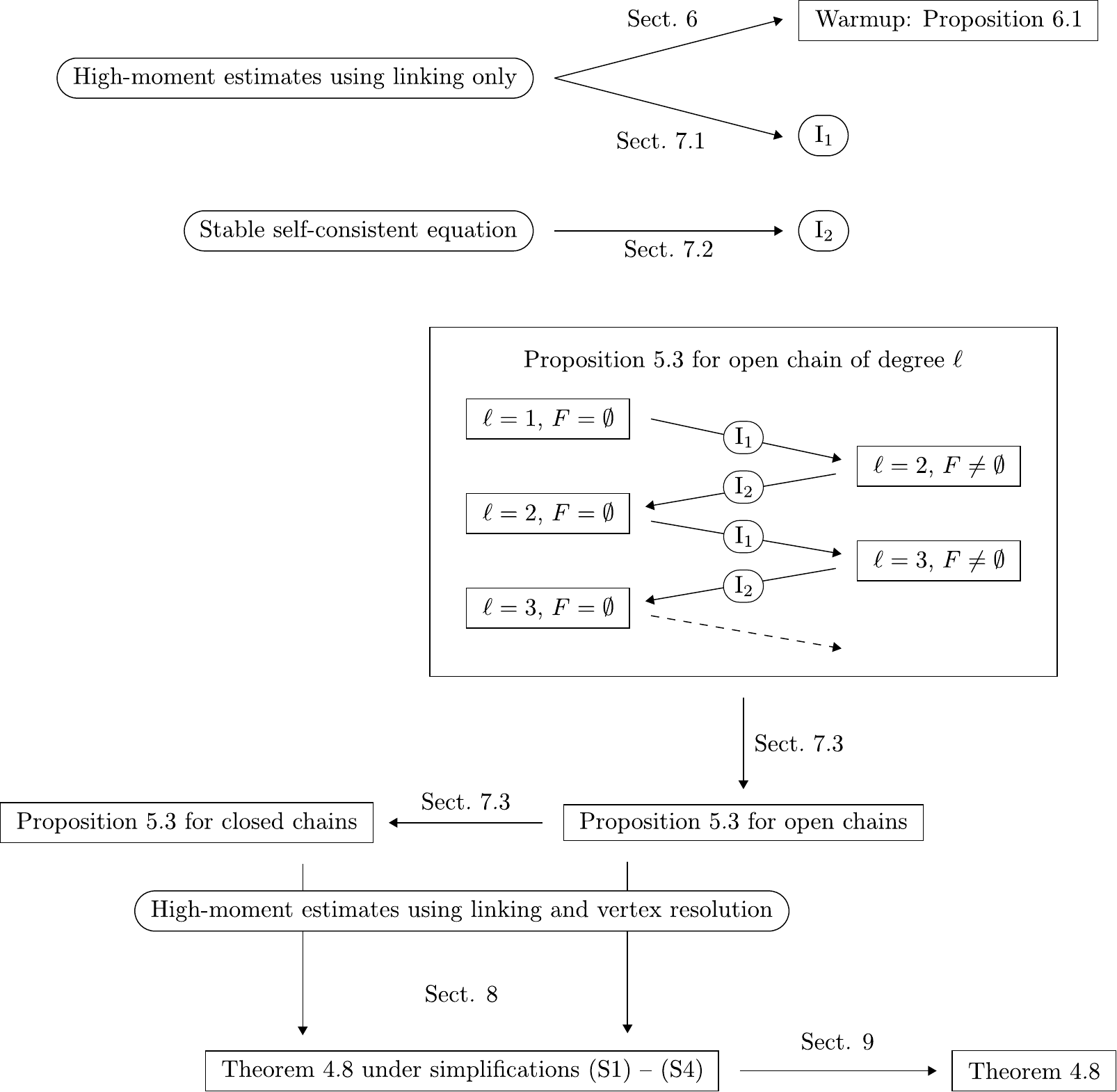}
\end{center}
\caption{The structure of the proof of Theorem \ref{theorem: Z lemma}. Concepts and arguments are displayed in rounded boxes, statements and results in rectangular boxes. \label{figure: roadmap}}
\end{figure}

We conclude this section with an outline of Sections \ref{section: Z3} -- \ref{section: simplifications}. In Section \ref{section: Z3} we present a simple high-moment estimate that only uses the process of \emph{linking} (see Definition \ref{def: linking}); more algebraically, the argument of Section \ref{section: Z3} only uses Family A identities (and not Family B). The result is Proposition \ref{prop: warm-up}, which obtains a gain of a factor $\Psi$ from each $Q$ but no gain from charged vertices (see Definition \ref{def: charged}). The goal of Section \ref{section: Z3} is twofold, the first goal being pedagogical. It provides a complete but vastly simplified proof of a special case of Theorem \ref{theorem: Z lemma}, thereby illustrating the process of linking. In addition, it lays the ground for Step $\rm I_1$ used to derive a priori bounds on chains, as well as for the more complicated high-moment estimates used in the full proof of Theorem \ref{theorem: Z lemma}.

Section 6 is devoted to chains; its goal is to prove Proposition \ref{lemma: weak Z lemma}. Step $\rm I_1$ is proved in Section \ref{section: Z4} and Step $\rm I_2$ in Section \ref{sec: chains with no Q}. The induction, and hence the proof of Proposition \ref{lemma: weak Z lemma}, is completed in Section \ref{sect: completion of induction}. In Section \ref{section: Z5} we prove Theorem \ref{theorem: Z lemma} under four simplifying assumptions, {\bf (S1)} -- {\bf (S4)} 
listed in Sections \ref{section: Z3} and \ref{section: Z5}. These simplifications allow us to ignore some additional complications, and give a streamlined argument in which the fundamental mechanism is evident. The starting point for the argument in Section \ref{section: Z5} is the high-moment expansion using vertex linking, already introduced in Section \ref{section: Z3}. In addition, we make use of Family B identities, which leads us to the process of vertex resolution (sketched in Section \ref{sec: sketch of resolution}). In Section \ref{section: simplifications} we present the additional arguments needed to drop Simplifications {\bf (S1)} -- {\bf (S4)}, and hence prove Theorem \ref{theorem: Z lemma} in full generality. Finally, in Section \ref{sec: proof of variant} we prove Theorem \ref{theorem: Z lemma variant} as a relatively easy consequence of Theorem \ref{theorem: Z lemma}.

\section{Warmup: simple high-moment estimates} \label{section: Z3}
We now move on to the high-moment estimates which underlie our proofs. The idea is to derive high-probability bounds on $X_F^w(\Delta)$ by controlling its high moments using a graphical expansion scheme.

For pedagogical reasons, we shall throughout the following selectively ignore some complications so as to make the core strategy clearer. We shall eventually put back the complications one by one. In this section we consistently assume the following simplification.
\begin{itemize}
\item[{\bf (S1)}]
All summation indices in the expanded summation $\E \abs{X_F^w(\Delta)}^p$ (see \eqref{E X^p} below) are distinct. (I.e.\ we ignore repeated indices which give rise to a smaller combinatorics of the summation.)
\end{itemize}
In this section we present a simple argument which proves the following weaker estimate.
\begin{proposition} \label{prop: warm-up}
Suppose that $\Lambda \prec \Psi$ for some admissible control parameter $\Psi$, $\Delta \in \fra Z$, and $w$ is an adapted weight. Then for all $F \subset V_s(\Delta)$ and $\f \mu$ we have
\begin{equation} \label{weak bound without vertex res}
X_F^w(\Delta) \;\prec\; \Psi^{\deg(\Delta) + \abs{F}}\,.
\end{equation}
\end{proposition}

The estimate \eqref{weak bound without vertex res} expresses that from each $Q$ in $X_F^w(\Delta)$ one gains an additional factor $\Psi$.

\begin{remark} \label{rem: warmup with denominators}
As in Remark \ref{rem: G in denominator}, the statement of Proposition \ref{prop: warm-up} remains true if some (or all) diagonal entries of the form $\cal G_{aa} = G_{aa} - m$ are replaced by $1/ G_{aa} - 1/m$. The proof is exactly the same.
\end{remark}

The simplified argument behind the proof of Proposition \ref{prop: warm-up} uses only the Family A identities, i.e.\ \eqref{resolvent expansion type 1}. It relies on a high-moment estimate of the following form. The precise statement is somewhat complicated by the need to keep track of low-probability exceptional events. The sum over $\Gamma \in \fra G$ in Lemma \ref{lemma: weak moment estimate} will arise as a summation over graphs.
\begin{lemma} \label{lemma: weak moment estimate}
Suppose that $\Lambda \prec \Psi$ for some admissible control parameter $\Psi$, and let $p \in 2 \N$ be even. Then we have
\begin{equation} \label{summation over graphs}
\E \abs{X_F^w(\Delta)}^p \;\leq\; \sum_{\Gamma \in \fra G} \E X_\Gamma\,,
\end{equation}
where $\fra G$ is a finite set (depending on $\Delta$, $F$, and $p$) and $X_\Gamma$ is a random variable satisfying
\begin{equation} \label{claimed estimate 1}
X_\Gamma \;\prec\; \Psi^{p(\deg(\Delta) + \abs{F})}
\end{equation}
as well as the rough bound
\begin{equation} \label{claimed estimate 2}
\E \abs{X_\Gamma}^2 \;\leq\; N^{C_p}
\end{equation}
for some constant $C_p$.
\end{lemma}
Before proving Lemma \ref{lemma: weak moment estimate}, we show how it implies Proposition \ref{prop: warm-up}. 

\begin{proof}[Proof of Proposition \ref{prop: warm-up}]
Let $\epsilon > 0$ and $D > 0$ be given. Define $p$ as the smallest even number greater than $4 D / \epsilon$, and abbreviate $q \deq \deg(\Delta) + \abs{F}$. Then by Lemma \ref{lemma: weak moment estimate}, for each $\Gamma \in \fra G$ there exists an event $\Xi_\Gamma$ such that
\begin{equation*}
\abs{X_\Gamma} \ind{\Xi_\Gamma} \;\leq\; N^{\epsilon p/2} \Psi^{pq} \,, \qquad \P(\Xi_\Gamma^c) \;\leq\; N^{-C_p - pq}
\end{equation*}
for all $w$, $\f \mu$, and $z \in \f S$. Then we find, using Lemma \ref{lemma: weak moment estimate} again,
\begin{align*}
\E \abs{X_F^w(\Delta)}^p &\;\leq\; \sum_{\Gamma \in \fra G} \pB{\E \pb{X_\Gamma \ind{\Xi_\Gamma}} + \E \pb{X_\Gamma \ind{\Xi^c_\Gamma}}}
\\
&\;\leq\; \sum_{\Gamma \in \fra G} \pB{N^{\epsilon p/2} \Psi^{pq} + \pb{\E \abs{X_\Gamma}^2}^{1/2} \P(\Xi_\Gamma^c)^{1/2}}
\\
&\;\leq\; \absb{\fra G} \pB{N^{\epsilon p/2} \Psi^{pq} + N^{-pq/2}}
\\
&\;\leq\; 2 \absb{\fra G} N^{\epsilon p/2} \Psi^{pq}\,,
\end{align*}
for all $w$, $\f \mu$, and $z \in \f S$. Therefore Chebyshev's inequality gives
\begin{equation*}
\P \pB{\abs{X_F^w(\Delta)} > N^\epsilon \Psi^q} \;\leq\; 2 \absb{\fra G} N^{- \epsilon p / 2} \;\leq\; N^{-D}
\end{equation*}
for all $w$, $\f \mu$, and $z \in \f S$.
\end{proof}

The rest of this section is devoted to the proof of Lemma \ref{lemma: weak moment estimate}. All of our estimates will be uniform in $w$ and $\f \mu$, and we shall henceforth no longer mention this explicitly. Throughout this section we assume Simplification {\bf (S1)}.

\begin{proof}[Proof of Lemma \ref{lemma: weak moment estimate}]
The idea of the proof was already outlined in Section \ref{sec: sketch of A}. Let $\Delta \in \fra Z$ have $n$ summation indices, denoted by $a_1, \dots, a_n$, and $k$ external indices, denoted by $\mu_1, \dots, \mu_k$. Let $F\subset \{1, \dots, n\}$. Let $p \in 2 \N$ be even and write
\begin{equation} \label{E X^p}
\E \abs{X_F^w(\Delta)}^p \;=\; \sum^{(\f \mu) *}_{\f a^1} w(\f a^1) \cdots \sum^{(\f \mu)*}_{\f a^p} w(\f a^p) \E \prod_{j = 1}^{p/2} \qBB{\pBB{\prod_{i \in F} Q_{a_i^j}} \cal Z_{\f a^j}} \prod_{j = p/2 + 1}^{p} \ol{\qBB{\pBB{\prod_{i \in F} Q_{a_i^j}} \cal Z_{\f a^j}}}\,,
\end{equation}
where we abbreviated
\begin{equation*}
\f a^j \;\deq\; \p{a^j_i \col 1 \leq i \leq n}\,, \qquad \f a \;\deq\; \pb{a_i^j \col 1 \leq i \leq n \,,\, 1 \leq j \leq p}\,.
\end{equation*}
We now make the crucial observation that $\wh w(\f a) \deq w(\f a^1) \ldots w(\f a^j)$ is a weight on the set of indices $(i,j)$; this is an elementary consequence of the Definition \eqref{definition of weight}. In particular, $\sum_{\f a} \wh w(\f a) \leq 1$.

By Simplification {\bf (S1)}, we assume that all indices $\f a$ are distinct: in addition to the constraint $a_i^j \notin \{\mu_1, \dots, \mu_k\}$, we introduce into \eqref{E X^p} an indicator function that imposes $a_i^j \neq a_{i'}^{j'}$ if $(i,j) \neq (i',j')$.

We now make each $\cal G_{xy}$ independent of as many summation indices as possible using Family A identities. To that end, we define
\begin{equation*}
\cal G_{ij}^{(T)} \;\deq\; G^{(T)}_{ij} - \delta_{ij} \, m\,.
\end{equation*}
Using \eqref{resolvent expansion type 1} iteratively, we expand every factor $\cal G_{xy}$ appearing in \eqref{E X^p} in all the indices
\begin{equation*}
\f a_F \;\deq\; \pb{a_i^j \col i \in F \,,\, 1 \leq j \leq p}
\end{equation*}
associated with a factor $Q$. Let $\cal G_{xy}$ be a fixed entry in \eqref{E X^p}. The idea is to successively add to $\cal G_{xy}$ as many upper indices from the collection $\f a_F$ as possible. The goal is to obtain a 
quantity satisfying the following definition.
\begin{definition} \label{definition: maximally expanded}
An entry $G_{xy}^{(T)}$ or $\cal G_{xy}^{(T)}$ is \emph{maximally expanded in $S$} if $S \subset T \sqcup \{x,y\}$.
In other words, a maximally expanded resolvent entry cannot be expanded any further in the indices $S$ using \eqref{resolvent expansion type 1}.
\end{definition}

Along the expansion of each  $\cal G_{xy}$ using
 \eqref{resolvent expansion type 1}, new terms in \eqref{E X^p} appear;
each such term is a monomial of entries of $\cal G$
divided by diagonal entries of $G$. We stop expanding a term if either
\begin{itemize}
\item[(a)]
all its factors are maximally expanded in $\f a_F$, or
\item[(b)]
it contains $\deg(\Delta) + 2 p n$ entries of $\cal G$ in the numerator.
\end{itemize}
The precise recursive procedure is as follows. We start by setting $A \deq \cal G_{xy}$, where $\cal G_{xy}$ is an entry on the right-hand side of \eqref{E X^p}. 
\begin{itemize}
\item[1.]
Let $\cal G_{uv}^{(T)}$ denote an entry in $A$ and $d$ an index in $\f a_F$ such that $d \notin T \cup \{u,v\}$. (This choice is arbitrary and unimportant.) If (a) no such pair exists, or (b) $A$ contains $\deg(\Delta) + 2 p n$ factors $\cal G$ in the numerator, stop the recursion of the term $A$.
\item[2.]
Using Family A identities, write
\begin{equation} \label{family A fo cal G 1}
\cal G^{(T)}_{uv} \;=\; \cal G_{uv}^{(Td)} + \frac{\cal G_{ud}^{(T)} \cal G_{dv}^{(T)}}{G_{dd}^{(T)}}
\end{equation}
if $\cal G_{uv}^{(T)}$ is a resolvent entry in the numerator and
\begin{equation} \label{family A fo cal G 2}
\frac{1}{G_{uu}^{(T)}} \;=\; \frac{1}{G_{uu}^{(Td)}} - \frac{\cal G_{ud}^{(T)} \cal G_{du}^{(T)}}{G_{uu}^{(T)} G_{uu}^{(Td)} G_{dd}^{(T)}}
\end{equation}
if $G_{uv}^{(T)} = G_{uu}^{(T)}$ is a diagonal resolvent entry in the denominator. This yields the splitting $A = A' + A''$,
where  both terms have the form of a product of entries in the numerator
and diagonal entries in the denominator.
 Repeat step 1 for both $A'$ and $A''$ (playing the role of $A$ in step 1).
\end{itemize}
It is not hard to see that the stopping rule defined by the conditions (a) or (b) ensures that the recursion terminates after a finite number of steps. Indeed, the quantity ``number of entries of $\cal G$'' + ``number of upper indices'' must remain bounded by the stopping rules (a) and (b).

The result is of the form
\begin{equation*}
\cal G_{xy} \;=\; \sum_{\alpha} H_\alpha + R\,,
\end{equation*}
where each summand $H_\alpha$ is a fraction with entries of $\cal G$ in the numerator and diagonal entries of $G$ in the denominator, all of them maximally expanded in $\f a_F$. Here the rest term $R$ satisfies 
\begin{equation} \label{estimate on rest term}
R \;\prec\; \Psi^{\deg(\Delta) + 2pn}\,, \qquad \E \abs{R}^2 \;\leq\; N^{C_{p,\Delta}}
\end{equation}
for some constant $C_{p,\Delta}$. The first estimate of \eqref{estimate on rest term} follows from \eqref{rough bound 2} combined with \eqref{m is bounded} and Lemma \ref{lemma: basic properties of prec}, and the second estimate of \eqref{estimate on rest term} from \eqref{rough bound 1} and \eqref{rough bound 3}.

We then multiply the resulting sums on the right-hand side of \eqref{definition of Z} out to get
\begin{equation} \label{expansion of Y}
\cal Z_{\f a^j} \;=\; \sum_{\alpha} Y^{j, \alpha}_{\f a}\,,
\end{equation}
where $Y^{j, \alpha}_{\f a}$ is a monomial and $\alpha$ a counting index ranging over some finite set. Each term $Y_{\f a}^{j, \alpha}$ is a fraction with entries of $\cal G$ in the numerator and diagonal entries of $F$ in the denominator. Moreover, either (i) all entries of $Y_{\f a}^{j,\alpha}$ are maximally expanded in $\f a_F$ or (ii) $Y_{\f a}^{j,\alpha} \prec \Psi^{\deg(\Delta) + 2pn}$ (the latter arises if $Y_{\f a}^{j,\alpha}$ contains one or more rest terms $R$). We now multiply out the expectation in \eqref{E X^p} as
\begin{equation} \label{fully expanded term}
\E \qBB{\pBB{\prod_{i \in F} Q_{a_i^1}} \cal Z_{\f a^1}} \cdots \qBB{\pBB{\prod_{i \in F} Q_{a_i^p}} \ol{\cal Z_{\f a^p}}} \;=\; \sum_{\alpha_1, \dots, \alpha_p}
\E \qBB{\pBB{\prod_{i \in F} Q_{a_i^1}} Y_{\f a}^{1,\alpha_1}} \cdots \qBB{\pBB{\prod_{i \in F} Q_{a_i^p}} \ol{Y_{\f a}^{p,\alpha_p}}}\,.
\end{equation}
We plug this into \eqref{E X^p} and pull out the summation over $\alpha_1, \dots, \alpha_p$.  This gives rise to the summation in \eqref{summation over graphs}, indexed by the set $\fra G = \{(\alpha_1, \dots, \alpha_p)\}$. If, for some $(\alpha_1, \dots, \alpha_p)$, one or more of $Y^{1, \alpha_1}_{\f a}, 
\dots, Y^{p, \alpha_p}_{\f a}$ is not maximally expanded in $\f a_F$, 
 it is easy to see that
\begin{equation} \label{definition of X_alpha}
X_{\alpha_1 \cdots \alpha_p} \;\deq\; \sum^{(\f \mu) *}_{\f a^1} w(\f a^1) \cdots \sum^{(\f \mu)*}_{\f a^p} w(\f a^p) \qBB{\pBB{\prod_{i \in F} Q_{a_i^1}} Y_{\f a}^{1,\alpha_1}} \cdots \qBB{\pBB{\prod_{i \in F} Q_{a_i^p}} \ol{Y_{\f a}^{p,\alpha_p}}}
\end{equation}
satisfies \eqref{claimed estimate 1} and \eqref{claimed estimate 2}. Indeed, each term  $Y_{\f a}^{j,\alpha_j}$ contains at least $\deg(\Delta)$ entries of $\cal G$; thus the trivial bound $Y_{\f a}^{j,\alpha_j} \prec \Psi^{\deg(\Delta)}$ always holds
by Lemma \ref{lemma: rough bounds on G}. Using Lemma \ref{lemma: basic properties of prec} we can multiply these estimates. Recalling \eqref{estimate on rest term}, \eqref{rough bound 1}, and \eqref{rough bound 3}, we find \eqref{claimed estimate 1} and \eqref{claimed estimate 2}.

It therefore suffices to consider products of $Y_{\f a}^{j,\alpha_j}$'s 
in \eqref{fully expanded term} which are all maximally expanded in $\f a_F$ (i.e.\ terms which are products of $H_\alpha$'s only and not $R$'s). The presence of $Q$'s leads to the following crucial restriction on terms yielding a nonzero contribution to \eqref{fully expanded term}. For each $i \in F$, we claim that at least one of $Y_{\f a}^{2,\alpha_2}, \dots, Y_{\f a}^{p,\alpha_p}$ is not independent of $a_i^1$. This follows from the observation that generally $\E [Q_a(X) Y] = 0$ if $Y$ is independent of $a$.
More generally, we require that, for any $i \in F$ and $j = 1, \dots, p$, at least of one of
\begin{equation*}
Y_{\f a}^{1,\alpha_1}, \dots, \widehat{Y_{\f a}^{j,\alpha_j}}, \dots, Y_{\f a}^{p,\alpha_p}
\end{equation*}
is not independent of $a_i^j$ (hat indicates omission
from the list).  This imposes a constraint on the terms that survive the expansion.

Moreover, the term that is not independent of $a_i^j$
contains at least one additional entry of $\cal G$, since
at some point the formula \eqref{family A fo cal G 1} or \eqref{family A fo cal G 2} had to be applied
with $d=a_i^j$ and the second term of \eqref{family A fo cal G 1} or \eqref{family A fo cal G 2}
contains at least one additional entry of $\cal G$. Since we assumed Simplification {\bf (S1)},
i.e.\ all $a_i^j$'s are different, it is a general fact
that each $Q$ gives rise to an additional off-diagonal entry of $\cal G$
and contributes a factor $\Lambda \prec \Psi$ to \eqref{E X^p}. In other words, any $X_{\alpha_1 \cdots \alpha_p}$ yielding a nonzero contribution to \eqref{E X^p} has at least $p(\deg(\Delta) + \abs{F})$ entries of $\cal G$ in the numerator. Recalling Lemma \ref{lemma: basic properties of prec} and Lemma \ref{lemma: rough bounds on G}, we find that any term $X_{\alpha_1 \cdots \alpha_p}$ yielding a nonzero contribution to \eqref{E X^p} satisfies \eqref{claimed estimate 1} and \eqref{claimed estimate 2}. This concludes the proof of Lemma \ref{lemma: weak moment estimate}.
\end{proof}

\subsection{Graphical representation}
The phenomenon behind the proof of Lemma \ref{lemma: weak moment estimate} in fact has a simple graphical representation, which will prove essential for later, more intricate, estimates. We illustrate its usefulness by applying it to the proof of Lemma \ref{lemma: weak moment estimate}. We recall the basic graphical notation introduced in Section \ref{sec: sketch: graphs}.

The quantity whose expectation we are estimating, $\abs{X_F^w(\Delta)}^p =\big[ X_F^w(\Delta)\big]^{p/2} \big[\ov{X_F^w(\Delta)}\big]^{p/2}$, has a natural representation in terms of a multigraph, which  we call $\gamma^p(\Delta)$ and which is essentially a $p$-fold copy of the graph $\Delta$ encoding $\cal Z$. The graph $\gamma^p(\Delta)$ is obtained as follows.
\begin{enumerate}
\item
Take $p/2$ copies of $\Delta$ and $p/2$ copies of $\Delta$ whose edges have inverted direction and colour arising from the relation $\overline{\cal G}_{ab}= \cal G_{ba}^*$. More precisely, this inversion means that each edge $e \in E(\Delta)$ gives rise to an inverted edge $e'$ satisfying
\begin{equation*}
\xi_{e'} \;=\;
\begin{cases}
1 & \text{if } \xi_e = *
\\
* & \text{if } \xi_e = 1\,,
\end{cases}
\qquad
\alpha(e') \;=\; \beta(e) \,, \qquad \beta(e') \;=\; \alpha(e)\,.
\end{equation*}
\item
For each external vertex $i \in V_e(\Delta)$ merge all $p$ copies of $i$ to form a single vertex.
\end{enumerate}
Note that the set $F$ is not depicted in $\gamma^p(\Delta)$. The vertex set of $\gamma^p(\Delta)$ consists of summation vertices and external vertices (this classification is inherited from the vertices of $\Delta$ in the obvious way), so that we may write $V(\gamma^p(\Delta)) = V_s(\gamma^p(\Delta)) \sqcup V_e(\gamma^p(\Delta))$.
\begin{definition}[Projection $\pi$] \label{definition: pi}
We introduce the $p$-to-one canonical projection $\pi \col V(\gamma^p(\Delta)) \to V(\Delta)$, defined as $\pi(i) = j$ if $i$ is a copy of $j$ in the construction of $\gamma^p(\Delta)$.
\end{definition}
\begin{figure}[ht!]
\begin{center}
\includegraphics{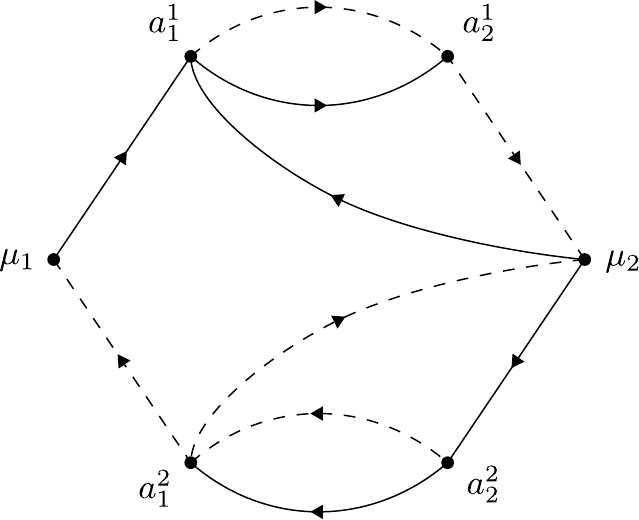}
\end{center}
\caption{The graph $\gamma^{2}(\Delta)$ that encodes $\E \abs{X_F^w(\Delta)}^2$, where $\cal Z_{a_1 a_2} \;=\; \cal G_{\mu_1 a_1} \cal G_{a_1 a_2} \cal G_{a_2 \mu_2}^* \cal G_{a_1 a_2}^* \cal G_{\mu_2 a_1}$. \label{figure: variance graph}}
\end{figure}

We start with the graph $\gamma^p(\Delta)$ (see Figure \ref{figure: variance graph}). We shall construct a set $\wt {\fra G}_F^p(\Delta)$ of graphs, denoted by $\Gamma$, on the same vertex set $V(\gamma^p(\Delta))$. The algorithm that generates $\wt {\fra G}_F^p(\Delta)$ is precisely the one given after Definition \ref{definition: maximally expanded}. On the level of graphs, this algorithm consists of a repeated application of the graphical rules in Figures \ref{figure: expansion of off-diag G} and \ref{figure: expansion of diag G}. (Note that the second identity of Figure \ref{figure: expansion of diag G} is also valid for $\cal G$ instead of $G$, i.e.\ without the black diamonds.) As indicated in Figures \ref{figure: expansion of off-diag G} and \ref{figure: expansion of diag G}, we keep track of the upper indices associated with an edge by attaching a list of upper indices to each edge. The algorithm terminates when either all edges are maximally expanded in $\f a_F$ or there are $\deg(\Delta) + pn$ edges that do not bear a diamond (i.e.\ that contribute a factor $\Psi$). Indicating these upper indices may be more precisely implemented using decorated edges, but we shall not need such formal constructions.

Recall from Definition \ref{def: linking} that choosing the second graph on the right-hand side of any identity in Figures \ref{figure: expansion of off-diag G} and \ref{figure: expansion of diag G} is called \emph{linking} (an edge with a vertex). As shown above, any graph $\Gamma \in \wt {\fra G}_F^p(\Delta)$ whose edges are not maximally expanded yields a small enough contribution by a trivial power counting. In the following we shall therefore only consider the remaining graphs, i.e.\ we shall assume that all edges of $\Gamma \in \wt {\fra G}_F^p(\Delta)$ are maximally expanded in $\f a_F$. Moreover, the upper indices of an edge are uniquely determined by the constraint that the edge be maximally expanded: the entry encoded by the edge $(x,y)$ is $\cal G_{xy}^{(\f a_F \setminus \{x,y\})}$. Thus, we shall consistently drop the upper indices associated with edges from our graphs $\Gamma$.

We introduce some convenient notions when dealing with graphs in $\wt {\fra G}_F^p(\Delta)$.

\begin{definition} \label{definition: basics of Gamma}
Let $\Gamma \in \wt {\fra G}_F^p(\Delta)$.
\begin{enumerate}
\item
We denote the vertex set of $\Gamma$ by $V(\Gamma) = V_s(\Gamma) \sqcup V_e(\Gamma)$, where $V_s(\Gamma)$ denotes the summation vertices and $V_e(\Gamma)$ external (fixed) vertices. By definition, all three sets are the same as those of $\gamma^p(\Delta)$.
\item
We denote the set of edges of $\Gamma$ by $E(\Gamma)$; the set $E(\Gamma)$ has a colouring $\xi \col E(\Gamma) \to \{1, *\}$.
\item
The $p$-to-one canonical projection $\pi \col V(\Gamma) \to V(\Delta)$ is taken over from Definition \ref{definition: pi}.
\end{enumerate}
\end{definition}
Thus, $\f a_F = \pb{a_i \col i \in \pi^{-1}(F)}$.
In this manner we write the sum of maximally expanded terms on the right-hand side of \eqref{fully expanded term} as a sum of graphs $\Gamma \in \wt {\fra G}^p_F(\Delta)$. By definition, each vertex in $\pi^{-1}(F)$ has been linked with at least one edge, possibly more. Each such linking adds an edge to the graph, and hence contributes a factor $\Lambda \prec \Psi$ to its size. This concludes the graphical discussion behind the proof of \eqref{weak bound without vertex res}. Figure \ref{figure: example linking} shows two sample graphs from $\wt {\fra G}_{F}^2(\Delta)$ for the graph $\Delta$ from Figure \ref{figure: variance graph}, where $F$ consists of a single vertex associated with the summation variable $a_1$.
\begin{figure}[ht!]
\begin{center}
\includegraphics{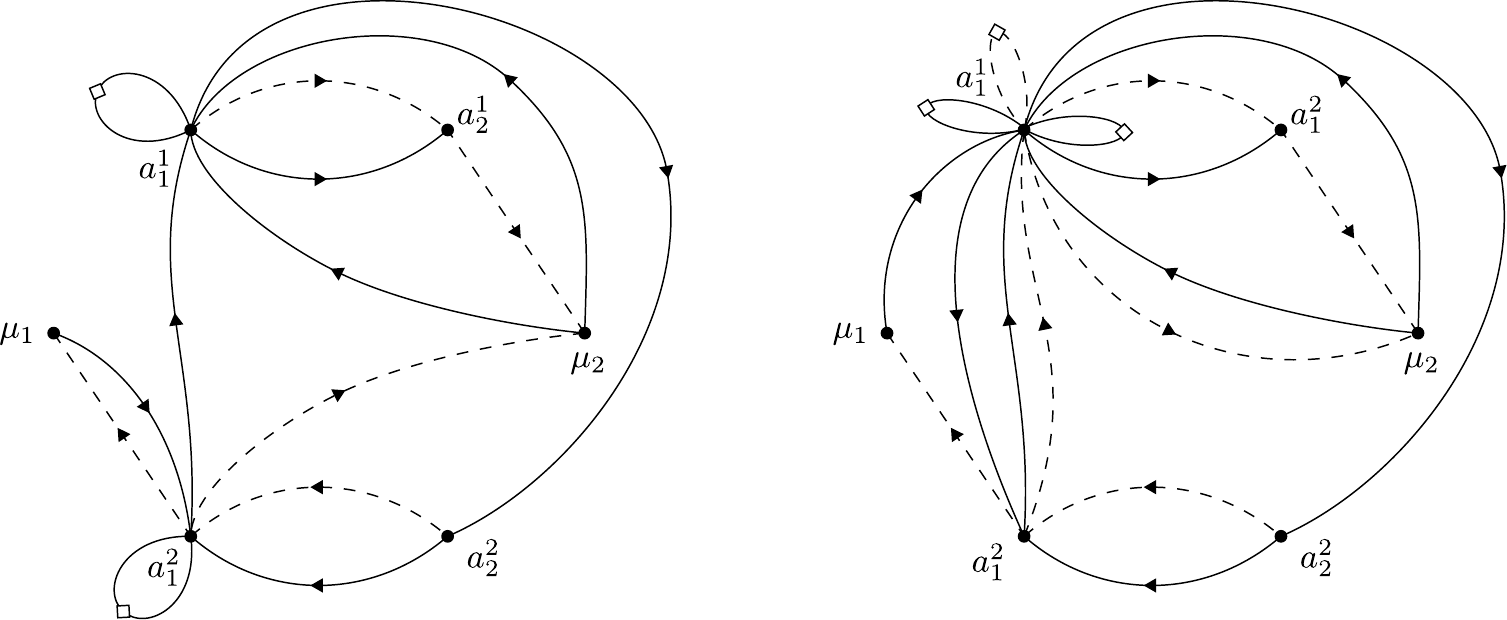}
\end{center}
\caption{Left: a graph obtained from the one in Figure \ref{figure: variance graph}
with $|F|=1$ by linking both vertices in $\f a_F = (a_1^1, a_1^2)$ with an edge; $a_1^1$ was linked with the edge $(\mu_2, a_2^2)$
and $a_1^2$ was linked with $(\mu_1, a_1^1)$. This graph contains the minimal number of edges to yield a nonzero contribution after taking the expectation. Hence it is of leading order. Right: a graph obtained by linking two further edges with $a_1^1$, namely 
the edges $(a_1^2,\mu_2)$ and $(\mu_1, a_1^2)$. Its size is subleading. \label{figure: example linking}}
\end{figure}

\section{Chains} \label{sec: chains}

In this section we derive the a priori estimate on chains, Proposition \ref{lemma: weak Z lemma}.

\subsection{Step $\rr I_1$: chains with $F \neq \emptyset$} \label{section: Z4}
Step $\rm I_1$ is an application of the simple high-moment expansion method from Section \ref{section: Z3}. It is formulated in the following proposition. In Section \ref{sect: completion of induction}, it will be used in conjunction with Proposition \ref{prop:step1} below to complete the induction and hence the proof of Proposition \ref{lemma: weak Z lemma}.

\begin{proposition}[Induction Step $\rm I_1$] \label{prop:step2}
Suppose that $\Lambda \prec \Psi$ for some admissible control parameter $\Psi$, and let $\ell \geq 2$. Suppose that
\begin{equation} \label{Step 1 bound}
X^w_F(\Delta) \;\prec\; \Psi^{\deg(\Delta) + \abs{F}} \Phi^{c(\Delta) - \abs{F}}
\end{equation}
holds for any open chain $\Delta$ of degree strictly less than $\ell$, $F = \emptyset$, and any adapted chain weight $w$. 
Then \eqref{Step 1 bound} holds for any open chain $\Delta$ of degree $\ell$, $F \neq \emptyset$, and any adapted chain weight $w$.
\end{proposition}

In this section we continue to assume Simplification {\bf (S1)} (see the beginning
of Section~\ref{section: Z3}).

\begin{proof}[Proof of Proposition \ref{prop:step2}]
For simplicity of notation, we focus on the case where $\Delta$ is a directed chain of degree $\ell$; the undirected case is proved in the same way. The argument is best understood in a representative example,
\begin{equation} \label{example of Z star}
\cal Z_{\f a} \;=\; G_{\mu_1 a_1} G_{a_1 a_2} G_{a_2 a_3} G_{a_3 a_4} G_{a_4 \mu_2}\,,
\end{equation}
which is encoded by the graph $\Delta$ depicted in Figure \ref{figure: example Z star}.
\begin{figure}[ht!]
\begin{center}
\includegraphics{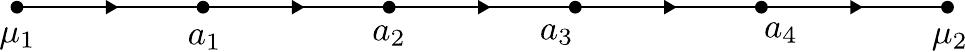}
\end{center}
\caption{The graph $\Delta$ that encodes $\cal Z$ defined in \eqref{example of Z star}. \label{figure: example Z star}}
\end{figure}
Let us take $F = \{1\}$ and compute the variance of $X_F^w(\Delta)$. In the following we use the terminology and notation of Section \ref{section: Z3} without further comment. In Section \ref{section: Z3} is was shown that the only graphs $\Gamma \in \wt {\fra G}^2_F(\Delta)$ that contribute are those in which the vertices $a_1^1$ and $a_1^2$ have both been linked to some edge.
\begin{figure}[ht!]
\begin{center}
\includegraphics{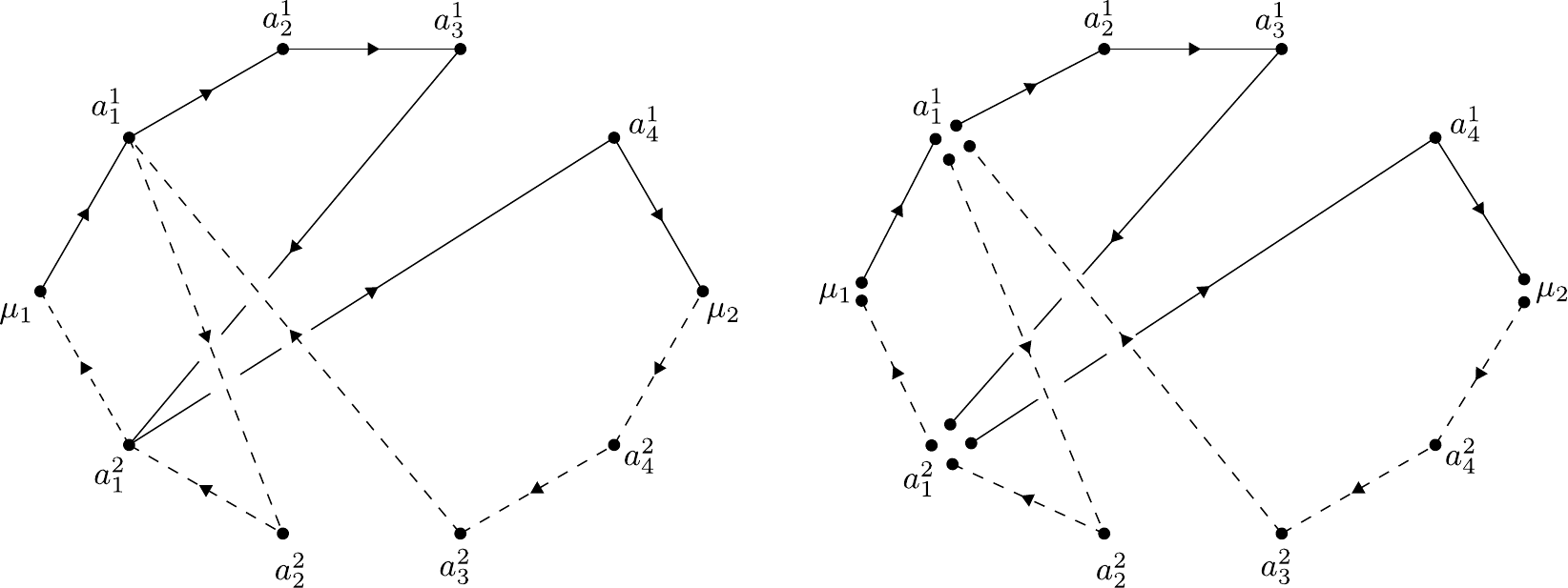}
\end{center}
\caption{Left: a graph $\Gamma$ of leading order in $\wt {\fra G}_F^2(\Delta)$ with $\Delta$ defined through \eqref{example of Z star}, and $F = \{1\}$. We do not draw the loops that encode diagonal resolvent entries in the denominator. Right: the same graph broken down to chains.
 \label{figure: square of Z star}}
\end{figure}
Figure \ref{figure: square of Z star} shows such a graph $\Gamma$ of leading order. Since the two vertices $a_1^1$ and $a_1^2$ have been linked, they each contribute a factor $\Lambda \prec \Psi$ (two edges were added by the linking process). Now we break the graph in Figure \ref{figure: square of Z star} down to its chains, i.e.\ we freeze all those summation vertices, $a_1^1$ and $a_1^2$, that were linked to. What remains is a collection of chains, each shorter than the original chain $\Delta$. In this example there are four nontrivial subchains:
\begin{equation*}
a_1^1 \to a_2^1 \to a_3^1 \to a_1^2\,, \qquad a_1^2 \to a_4^1 \to \mu_2\,, \qquad \mu_2 \to a_4^2 \to a_3^2 \to a_1^1\,, \qquad a_1^1 \to a_2^2 \to a_1^2\,.
\end{equation*}
Moreover, the monomial encoded by each subchain lies either inside $Q_{a_1^1}(\cdot)$, inside $Q_{a_1^2}(\cdot)$, or inside neither.
Thus the monomials encoded by the first two subchains lie inside $Q_{a_1^1}(\cdot)$, and the monomials encoded by the two last subchains inside $Q_{a_1^2}(\cdot)$.

Now we may invoke the induction assumption (i.e.\ Proposition \ref{lemma: weak Z lemma} for $F = \emptyset$) on each of the four subchains. We use that they all have degree strictly less than $\deg(\Delta)$. To be precise, before invoking Proposition \ref{lemma: weak Z lemma}, we have to get rid of the upper indices using \eqref{resolvent expansion type 1}; see below for details.

Moreover, we ignore some minor technicalities associated with coinciding indices. By Simplification {\bf (S1)}, we assumed that all summation indices of $\Gamma$ were distinct. In particular, the indices associated with different subchains of $\Gamma$ are distinct, which implies that the subchains of $\Gamma$ are coupled. This coupling is manifested in the fact that summation indices within a subchain are subject to additional restrictions that are unrelated to that subchain: these summation indices cannot take on values of indices in other subchains. This means that  summations cannot be performed independently within each subchain. Hence we may not strictly speaking invoke Proposition \ref{lemma: weak Z lemma} for each subchain; in order to do so, we first have to \emph{decouple} the subchains so as to get a product of terms associated with the subchains. In order to achieve this decoupling we have to allow indices associated with different subchains to coincide. This decoupling is a simple inclusion-exclusion argument whose details are postponed to Lemma \ref{lemma: factor chains} in Section \ref{section: Z8}.

Summarizing this example, we obtain an estimate of order $\Psi^{12} \Phi^6$ for the graph depicted in Figure \ref{figure: square of Z star}. Since, by Proposition \ref{lemma: weak Z lemma}, the contribution of an open subchain of degree $d$ is $\Psi^{d} \Phi^{d - 1}$,
the four non-trivial subchains yield a contribution $\Psi^3 \Phi^2 \, \Psi^2 \Phi \, \Psi^3 \Phi^2 \, \Psi^3 \Phi  = \Psi^{10} \Phi^6$.
There are also two trivial subchains, thus resulting in a total contribution $\Psi^{12} \Phi^6$.
Another way to think about such estimates is to count the additional factors of $\Psi$ and $\Phi$ gained along the proof. The naive size of the original
graph, before linking, was $\Psi^{2 \deg (\Delta)} = \Psi^{10}$ since
$\cal Z$ in \eqref{example of Z star} contains five factors and we consider its second moment (i.e.\ set $p=2$). Since $\abs{F}=1$,
we gain an additional $\Psi^{2\abs{F}} = \Psi^2$ from the linking; this
step increases the number of edges from 10 to 12 in the graph on
the left-hand side of Figure~\ref{figure: square of Z star}. Moreover, we gain an additional $\Phi$ factor from each internal summation vertex
in the subchains, in this example we gain a factor $\Phi$ from each of the six vertices $a_2^1$, $a_3^1$, $a_4^1$, $a_4^2$, $a_3^2$, and $a_2^2$. Thus we recover the bound $\Psi^{12} \Phi^6$.

Let us now give the general argument, which is in fact a trivial generalization of the above example. We start with a graph $\Gamma \in \wt {\fra G}_F^p(\Delta)$, as constructed in Section \ref{section: Z3}.
We split the summation indices $\f a = (\f a', \f a'')$, where $\f a'$ consists of the chain vertices of $\Gamma$. Thus, $\f a''$ contains in particular the indices associated with vertices which have been linked to an edge. By the argument of Section \ref{section: Z3}, $\f a''$ contains all indices of $\f a_F$, so that $ \abs{\f a''} \geq \abs{\f a_F}= p \abs{F}$. Since each linked vertex is incident to an additional edge resulting from linking, the graph $\Gamma$ contains at least $p \deg(\Delta) + \abs{\f a_F}
\geq p(\deg(\Delta)+|F|)$ edges. So far we have simply repeated the argument of Section \ref{section: Z3} and reproved the bound \eqref{weak bound without vertex res}.

In order to gain an additional factor $\Phi$ from each of the summation indices in $\f a'$, we use the induction assumption.  The assumption is used on open chains of vertices, i.e.\ subgraphs of $\Gamma$ which are open chains. We fix $\f a''$ and regard $\f a'$ as the summation indices. Then $\Gamma$ becomes a collection of open (sub)chains, and the vertices associated with $\f a'$ are the chain vertices of these subchains. If we can ensure that each subchain has degree strictly less than $\Delta$, we can apply the induction assumption to get an additional factor $\Phi$ from each chain vertex in represented in $\f a'$. This will give us a bound of size
\begin{equation*}
\Psi^{p \deg(\Delta) + \abs{\f a''}} \Phi^{\abs{\f a'}} \;\leq\; \Psi^{p (\deg(\Delta) + \abs{F})} \Phi^{p (c(\Delta) - \abs{F})}\,,
\end{equation*}
where we used that $\abs{\f a''} \geq p \abs{F}$, $\abs{\f a''} + \abs{\f a'} = p \, c(\Delta)$, and $\Psi \leq \Phi$.

In order to carry out this argument, we make the following observations.
\begin{enumerate}
\item
All subchains of $\Gamma$ have degree strictly less than $\deg(\Delta)$. This property is crucial for the induction. It is a consequence of the two following facts. First, the linking of vertices never produces new subchains nor lengthens pre-existing subchains. Note that vertices in $\f a''$ are fixed, and subchains terminate at them. Second, since $F \neq \emptyset$, at least one vertex of every subchain of degree $\deg(\Delta)$ in $\gamma^p(\Delta)$ will be linked to an edge, hence cutting the subchain of degree $\deg(\Delta)$ into smaller subchains.
\item
The expression $\cal Z'_{\f b}$ encoded by any subchain $\Gamma'$ of $\Gamma$ always appears in conjunction with a chain weight $w'(\f b)$. This is an immediate consequence of the fact that the weight $w(\f a^1) \cdots w(\f a^p)$ is a chain weight by assumption.
\item
Let $\cal Z'_{\f b}$ denote the monomial encoded by a subchain $\Gamma'$ of $\Gamma$.
Then any $Q_a$ has an index $a$ in $\f a''$ (i.e.\ is fixed), and acts either on all resolvent entries of $\cal Z'_{\f b}$ or none of them.
\end{enumerate}

In order to invoke the induction assumption, we still have to get rid of the upper indices in the maximally expanded resolvent entries. The procedure is almost identical to the one following Definition \ref{definition: maximally expanded}, but in the opposite direction. In particular,
the key formula \eqref{resolvent expansion type 1} should be viewed in the form
\begin{equation}\label{inverted res exp 1}
G_{ij}^{(Tk)} \;=\; G_{ij}^{(T)} - \frac{G_{ik}^{(T)} G_{kj}^{(T)}}{G_{kk}^{(T)}}\,, \qquad \frac{1}{G_{ii}^{(Tk)}} \;=\; \frac{1}{G_{ii}^{(T)}} +\frac{G_{ik}^{(T)} G_{ki}^{(T)}}{G_{ii}^{(T)} G_{ii}^{(Tk)} G_{kk}^{(T)}}\,.
\end{equation}
We start removing the upper indices one by one using \eqref{inverted res exp 1}, and stop if either all upper indices have been removed or if the number of off-diagonal resolvent entries exceeds $\deg(\Delta) + 2 p \ell$. The
size of the latter terms is already sufficiently small by the trivial bound $\Lambda \prec \Psi$. As for the former terms, they are represented by a new (but still finite) set of graphs in which every vertex is either a chain vertex or has been linked with an edge.

Now the induction assumption is applicable to each subchain, and the proof is completed by invoking Lemma \ref{lemma: basic properties of prec}. (Note that as before we ignored issues related to coinciding indices according to Simplification {\bf (S1)}; these are dealt with using the inclusion-exclusion argument of Lemma \ref{lemma: factor chains}.)
\end{proof}

\subsection{Step $\rr I_2$: chains with $F = \emptyset$} \label{sec: chains with no Q}
Step $\rr I_2$ is completed in the following proposition.

\begin{proposition}[Induction Step $\rr I_2$]\label{prop:step1}
Suppose that $\Lambda \prec \Psi$ for some admissible control parameter $\Psi$, and let $\ell \geq 2$. Suppose that
\begin{equation} \label{Step 2 bound}
X^w_F(\Delta) \;\prec\; \Psi^{\deg(\Delta) + \abs{F}} \Phi^{c(\Delta) - \abs{F}}
\end{equation}
holds for any open chain $\Delta$ of degree $\ell$, any $F \neq \emptyset$, and any adapted chain weight $w$. If $\ell \geq 3$, suppose in addition that \eqref{Step 2 bound} holds for any open chain $\Delta$ of degree strictly less than $\ell$, $F = \emptyset$, and any adapted chain weight $w$.

Then \eqref{Step 2 bound} holds for any open chain $\Delta$ of degree $\ell$, $F = \emptyset$, and any adapted chain weight $w$.
\end{proposition}

\begin{proof}
As before, we focus on the case where $\Delta$ is a directed open chain; the proof in in the undirected case is the same. Thus,
\begin{equation*}
\cal Z_{a_1 \cdots a_n} \;=\; G_{\mu_1 a_1} G_{a_1 a_2} \cdots G_{a_n \mu_2}\,, \qquad w_{\f b}(\f a) \;\equiv\; w(\f a) \;=\; s_{a_1 b_1} \cdots s_{a_n b_n}
\end{equation*}
for some $\f b = \{b_1, \dots, b_n\}$. (Recall that $G_{ij} = \cal G_{ij}$ for $i \neq j$.) Note that $n = \deg(\Delta) - 1$. We have to prove that
\begin{equation} \label{self-const equ claim}
X_\emptyset^w(\Delta) \;\prec\; \Psi^{n + 1} \Phi^n\,.
\end{equation}
(Here we used that $c(\Delta) = n$.) The main idea of the proof was given in Section \ref{sec: proof of (B)}: derive a stable self-consistent equation whose error terms may be estimated using the induction assumption. We subdivide the proof into six steps.

To simplify notation, throughout this proof we use $\cal E \equiv \cal E(\f b, \mu_1, \mu_2)$ to denote a random error term satisfying $\cal E \prec \Psi^{n + 2} \Phi^{n - 1}$. Like generic constants $C$, these error terms may change from line to line
without changing name.

Moreover, in order to keep the presentation more concise, we shall sometimes ignore unimportant subtleties arising from coinciding summation indices. These complications are harmless and will be dealt with precisely using the inclusion-exclusion argument of Lemma \ref{lemma: factor chains}. The general philosophy is the following: if we constrain a pair of indices to coincide instead of being distinct, we lose at most two factors of $\Psi$. Indeed, we lose at most one chain vertex (resulting in a loss of $\Phi \geq \Psi$), and at most one off-diagonal entry of $G$ may become diagonal (resulting in a loss of $\Psi$). Note that, since $\Delta$ is an open chain, at most one off-diagonal entry may become diagonal when setting two summation indices to be equal. (This is not true for closed chains; see Section \ref{sect: completion of induction}.) This loss of $\Psi^2$ is compensated by the factor $M^{-1} \leq \Psi^2$ we gain from the reduction in the number of summation variables.

\medskip
{\bf Step (i).} We introduce a factor $P_{a_1}$ into the summation in $X_\emptyset^w(\Delta)$. We find
\begin{align*}
X_\emptyset^w(\Delta) &\;=\; \sum_{\f a}^{(\mu_1 \mu_2)*}  w(\f a) \, P_{a_1} \cal Z_{a_1 \cdots a_n} + X_{\{1\}}^w(\Delta)
\\
&\;=\; \sum_{\f a}^{(\mu_1 \mu_2)*}  w(\f a) \, P_{a_1} \cal Z_{a_1 \cdots a_n} + \cal E\,.
\end{align*}
where the second equality follows from the induction assumption.

\medskip
{\bf Step (ii).}
We introduce a factor $m^2 / (G_{a_1 a_1})^2$ in front of $X_\emptyset^w(\Delta)$; this prefactor will be important in the fourth step below, as the factor $1 / (G_{a_1 a_1})^2$ will be used to cancel diagonal resolvent entries arising from two applications of the identity \eqref{res exp 2b}. We find
\begin{align*}
\sum_{\f a}^{(\mu_1 \mu_2)*}  w(\f a) \, P_{a_1} \cal Z_{a_1 \cdots a_n} &\;=\; \sum_{\f a}^{(\mu_1 \mu_2)*} w(\f a) P_{a_1} \qBB{\frac{(G_{a_1 a_1} - m)^2 + 2 m (G_{a_1 a_1} - m) + m^2}{(G_{a_1 a_1})^2} \cal Z_{a_1 \cdots a_n}}
\\
&\;=\; \sum_{a_1}^{(\mu_1 \mu_2)} s_{a_1 b_1} P_{a_1} \Biggl[ G_{\mu_1 a_1} \frac{(G_{a_1 a_1} - m)^2 + 2 m (G_{a_1 a_1} - m) + m^2}{(G_{a_1 a_1})^2}
\\
&\mspace{40mu} \times \sum_{a_2, \dots, a_n}^{(a_ 1\mu_1 \mu_2)*} s_{a_2 b_2} \cdots s_{a_n b_n} G_{a_1 a_2} \cdots G_{a_n \mu_2} \Biggr]\,.
\end{align*}
For $n = 1$ (i.e.\ $\deg(\Delta) = 2$) the last line is understood to be $G_{a_1 \mu_2}$.

We now show that the only the term $m^2$ in the numerator is relevant. By induction assumption and Lemma \ref{lemma: basic properties of prec}, we find that
\begin{equation} \label{self-const proof step 1}
\sum_{a_2, \dots, a_n}^{(a_1 \mu_1 \mu_2)*} s_{a_2 b_2} \cdots s_{a_n b_n} G_{a_1 a_2} \cdots G_{a_n \mu_2} \;\prec\; \Psi^{n} \Phi^{n - 1}\,.
\end{equation}
(Note that the induction assumption is only used if $n \geq 2$; for the initial value $n = 1$ \eqref{self-const proof step 1} is trivial.) Using $\sum_{a_1} s_{a_1 b_1} \leq 1$, \eqref{1/G prec 1}, and Lemma \ref{lemma: basic properties of prec} again, we get
\begin{equation} \label{self-const proof step 2}
X_\emptyset^w(\Delta) \;=\; \wt X_\emptyset^w(\Delta) + \cal E\,, \qquad \wt X_\emptyset^w(\Delta) \;\deq\; \sum_{\f a}^{(\mu_1 \mu_2)*}  w(\f a) \, P_{a_1} \qbb{\frac{m^2}{(G_{a_1 a_1})^2} \cal Z_{a_1 \cdots a_n}}\,.
\end{equation}

\medskip
{\bf Step (iii).} We make all resolvent entries which do not contain the index $a_1$ independent of $a_1$ using \eqref{resolvent expansion type 1}. Thus, we assume that $n \geq 2$; if $n = 1$ there is nothing to be done and this step is trivial. Using \eqref{resolvent expansion type 1} we find
\begin{align*}
\wt X_\emptyset^w(\Delta) &\;=\; \sum_{\f a}^{(\mu_1 \mu_2)*}  w(\f a) \, P_{a_1} \qbb{\frac{m^2}{(G_{a_1 a_1})^2} G_{\mu_1 a_1} G_{a_1 a_2} G_{a_2 a_3} G_{a_3 a_4} \cdots G_{a_n \mu_2}}
\\
&\;=\; \sum_{\f a}^{(\mu_1 \mu_2)*}  w(\f a) \, P_{a_1} \qbb{\frac{m^2}{(G_{a_1 a_1})^2}G_{\mu_1 a_1} G_{a_1 a_2} \pbb{G_{a_2 a_3}^{(a_1)} +\frac{G_{a_2 a_1} G_{a_1 a_3}}{G_{a_1 a_1}}} G_{a_3 a_4} \cdots G_{a_n \mu_2}}
\\
&\;=\; \sum_{\f a}^{(\mu_1 \mu_2)*}  w(\f a) \, P_{a_1} \qbb{\frac{m^2}{(G_{a_1 a_1})^2} G_{\mu_1 a_1} G_{a_1 a_2} G_{a_2 a_3}^{(a_1)} G_{a_3 a_4} \cdots G_{a_n \mu_2}} + \cal E\,.
\end{align*}
Here the bound on the error term
\begin{equation*}
\sum_{a_1}^{(\mu_1 \mu_2)} s_{a_1 b_1} P_{a_1} \qBB{\frac{m^2}{(G_{a_1 a_1})^3} G_{\mu_1 a_1} \sum_{a_2, \dots, a_n}^{(a_1 \mu_1 \mu_2)*}  s_{a_2 b_2} \cdots s_{a_n b_n} \, G_{a_1 a_2} G_{a_2 a_1} G_{a_1 a_3} G_{a_3 a_4} \cdots G_{a_n \mu_2}}
\end{equation*}
follows by first fixing the summation index $a_1$ and using the induction assumption combined with Lemma \ref{lemma: basic properties of prec}, similarly to Step (ii) above. The induction assumption is used on two chains: one of degree $2$ (corresponding to $G_{a_1 a_2} G_{a_2 a_1}$) and one of degree $n - 1$ (corresponding to $G_{a_1 a_3} G_{a_3 a_4} \cdots G_{a_n \mu_2}$); here $a_1$ is regarded as an external index. (Here we swept under the rug a minor technicality. Strictly speaking, the expressions encoded by different subchains do not factor, since their summations are still coupled by the constraint $a_2 \notin \{a_1, a_3, a_4, \dots, a_n\}$. As outlined above, we ignore such complications here; they are dealt with using the inclusion-exclusion argument from Lemma \ref{lemma: factor chains} in Section \ref{section: Z8} by introducing a partitioning on the values of $a_2$, which results in a decoupled expression plus a series of small error terms.)

Next, we write
\begin{multline} \label{self-const proof step 2'}
\sum_{\f a}^{(\mu_1 \mu_2)*} w(\f a) \, P_{a_1} \qbb{\frac{m^2}{(G_{a_1 a_1})^2} G_{\mu_1 a_1} G_{a_1 a_2} G_{a_2 a_3}^{(a_1)} G_{a_3 a_4} G_{a_4 a_5} \cdots G_{a_n \mu_2}}
\\
=\; \sum_{\f a}^{(\mu_1 \mu_2)*}  w(\f a) \, P_{a_1} \qbb{\frac{m^2}{(G_{a_1 a_1})^2}G_{\mu_1 a_1} G_{a_1 a_2} G_{a_2 a_3}^{(a_1)} G_{a_3 a_4}^{(a_1)} G_{a_4 a_5} \cdots G_{a_n \mu_2}} + \cal R\,,
\end{multline}
where the error term is
\begin{align*}
\cal R &\;\deq\; \sum_{\f a}^{(\mu_1 \mu_2)*}  w(\f a) \, P_{a_1} \qbb{\frac{m^2}{(G_{a_1 a_1})^2}G_{\mu_1 a_1} G_{a_1 a_2} G_{a_2 a_3}^{(a_1)} \frac{G_{a_3 a_1} G_{a_1 a_4}}{G_{a_1 a_1}} G_{a_4 a_5} \cdots G_{a_n \mu_2}}
\\
&\;=\; \sum_{\f a}^{(\mu_1 \mu_2)*}  w(\f a) \, P_{a_1} \qbb{\frac{m^2}{(G_{a_1 a_1})^2}G_{\mu_1 a_1} G_{a_1 a_2} \pbb{G_{a_2 a_3} - \frac{G_{a_2 a_1} G_{a_1 a_3}}{G_{a_1 a_1}}} \frac{G_{a_3 a_1} G_{a_1 a_4}}{G_{a_1 a_1}} G_{a_4 a_5} \cdots G_{a_n \mu_2}}\,.
\end{align*}
We may estimate this exactly as above, by first fixing $a_1$ and regarding it as an external index. The induction assumption allows us to estimate the two resulting terms
\begin{equation*}
\sum_{a_2, \dots, a_n}^{(a_1 \mu_1 \mu_2)*} s_{a_2 b_2} \cdots s_{a_n b_n} \, G_{a_1 a_2} G_{a_2 a_3} G_{a_3 a_1} G_{a_1 a_4} G_{a_4 a_5} \cdots G_{a_n \mu_2}
\end{equation*}
(a product of two subchains) and
\begin{equation*}
\sum_{a_2, \dots, a_n}^{(a_1 \mu_1 \mu_2)*} s_{a_2 b_2} \cdots s_{a_n b_n} \, G_{a_1 a_2} G_{a_2 a_1} G_{a_1 a_3} G_{a_3 a_1} G_{a_1 a_4} G_{a_4 a_5} \cdots G_{a_n \mu_2}\,,
\end{equation*}
(a product of three subchains). Note that the induction assumption is always used on subchains of degree strictly less than $\deg(\Delta) = n + 1$.  Using Lemma \ref{lemma: basic properties of prec}, one therefore finds that $\cal R$ in \eqref{self-const proof step 2'} can be replaced with an $\cal E$. 
Continuing in this manner, we eventually get
\begin{equation} \label{self-const proof step 3}
\wt X_\emptyset^w(\Delta) \;=\; \sum_{\f a}^{(\mu_1 \mu_2)*}  w(\f a) \, P_{a_1} \qbb{\frac{m^2}{(G_{a_1 a_1})^2}G_{\mu_1 a_1} G_{a_1 a_2} G_{a_2 a_3}^{(a_1)} G_{a_3 a_4}^{(a_1)} \cdots G_{a_n \mu_2}^{(a_1)}} + \cal E\,.
\end{equation}

\medskip
{\bf Step (iv).} We apply the identity \eqref{res exp 2b} to both resolvent entries with lower index $a_1$. This yields
\begin{align}
\wt X_\emptyset^w(\Delta) &\;=\; m^2 \sum_{\f a}^{(\mu_1 \mu_2)*}  w(\f a) \, P_{a_1} \qBB{ \sum_{d,d'}^{(a_1)} G_{\mu_1 d}^{(a_1)} h_{d a_1} h_{a_1 d'} G_{d' a_2}^{(a_1)} G_{a_2 a_3}^{(a_1)} G_{a_3 a_4}^{(a_1)} \cdots G_{a_n \mu_2}^{(a_1)}} + \cal E
\notag \\
&\;=\; \label{self-const proof step 4}
m^2 \sum_{\f a}^{(\mu_1 \mu_2)*} \sum_{d}^{(a_1)} w(\f a) s_{a_1 d} \, G_{\mu_1 d}^{(a_1)} G_{d a_2}^{(a_1)} G_{a_2 a_3}^{(a_1)} G_{a_3 a_4}^{(a_1)} \cdots G_{a_n \mu_2}^{(a_1)} + \cal E\,,
\end{align}
where in the second step we used that $P_{a_1} (h_{d a_1} h_{a_1 d'}) = \delta_{d d'} s_{a_1 d}$, and that all resolvent entries are independent of $a_1$.
Renaming $(a_1, d) \mapsto (d,a_1)$ and interchanging the order of summation, we find from \eqref{self-const proof step 2} and \eqref{self-const proof step 4}
\begin{equation} \label{self-const proof step 5}
X_\emptyset^w(\Delta) \;=\; m^2 \sum_d^{(\mu_1 \mu_2)} s_{b_1 d} \sum_{a_1}^{(d)} \sum_{a_2, \dots, a_n}^{(d \mu_1 \mu_2) *} s_{d a_1} s_{b_2 a_2} \cdots s_{b_n a_n} G_{\mu_1 a_1}^{(d)} G_{a_1 a_2}^{(d)} \cdots G_{a_n \mu_2}^{(d)} + \cal E .
\end{equation}

\medskip
{\bf Step (v).} Having performed the expectation, we now get rid of the upper indices $d$ in all of the resolvent entries of \eqref{self-const proof step 5} to go back to the original resolvent entries. To that end, we write
\begin{equation} \label{removing upper indices}
G_{\mu_1 a_1}^{(d)} G_{a_1 a_2}^{(d)} \cdots G_{a_n \mu_2}^{(d)} \;=\; \pbb{G_{\mu_1 a_1} - \frac{G_{\mu_1 d} G_{d a_1}}{G_{dd}}} \pbb{G_{a_1 a_2} - \frac{G_{a_1 d} G_{d a_2}}{G_{dd}}} \cdots \pbb{G_{a_n \mu_2} - \frac{G_{a_n d} G_{d \mu_2}}{G_{dd}}}
\end{equation}
in the summand of \eqref{self-const proof step 5}, and multiply everything out. As before, each term results in a collection of subchains of degree strictly less than $\deg(\Delta) = n + 1$, to which the induction assumption may be applied. More precisely, suppose that we have chosen the second term in $k \leq n + 1$ of the factors in \eqref{removing upper indices}. Then we get $k$ additional resolvent entries of $G$ (yielding a total of $n + k + 1$), as well as a collection of subchains whose total number of chain vertices is $n - k$.  Notice that the index structure of every terms after multiplying 
\eqref{removing upper indices} out is chain-like.
The result is
\begin{equation} \label{self-const proof step 6}
X_\emptyset^w(\Delta) \;=\; m^2 \sum_d^{(\mu_1 \mu_2)} s_{b_1 d} \sum_{a_1, \dots, a_n}^{(d \mu_1 \mu_2) *} s_{d a_1} s_{b_2 a_2} \cdots s_{b_n a_n} G_{\mu_1 a_1} G_{a_1 a_2} \cdots G_{a_n \mu_2} + \cal E\,.
\end{equation}
(As before, we ignore the issues related to the cases $a_1 \in \{\mu_1, \mu_2, a_2, \dots, a_n\}$; see the inclusion-exclusion argument of Lemma \ref{lemma: factor chains}.) We have the rough bound
\begin{equation} \label{rough chain bound}
\sum_{a_1, \dots, a_n}^{(d \mu_1 \mu_2) *} s_{d a_1} s_{b_2 a_2} \cdots s_{b_n a_n} G_{\mu_1 a_1} G_{a_1 a_2} \cdots G_{a_n \mu_2} \;\prec\; \Psi^{n + 1} \Phi^{n - 1}\,.
\end{equation}
Indeed, by fixing $a_1$ and using the induction assumption on the subchain of degree $n$ that encodes the expression $G_{a_1 a_2} \cdots G_{a_n \mu_2}$, \eqref{rough chain bound} follows by Lemma \ref{lemma: basic properties of prec}. It follows using $s_{b_1 d} \leq M^{-1}$ that we may replace the summation $\sum_d^{(\mu_1 \mu_2)}$ in \eqref{self-const proof step 6} by $\sum_d$ up to an error of type $\cal E$. Finally, we may replace the sum $\sum_{a_1, \dots, a_n}^{(d \mu_1 \mu_2) *}$ in \eqref{self-const proof step 6} by $\sum_{a_1, \dots, a_n}^{(\mu_1 \mu_2) *}$ up to an error type $\cal E$, by the inclusion-exclusion argument of Lemma \ref{lemma: factor chains}. The result is
\begin{equation} \label{self-const proof step 7}
X_\emptyset^w(\Delta) \;=\; m^2 \sum_d s_{b_1 d} \sum_{a_1, \dots, a_n}^{(\mu_1 \mu_2) *} s_{d a_1} s_{b_2 a_2} \cdots s_{b_n a_n} G_{\mu_1 a_1} G_{a_1 a_2} \cdots G_{a_n \mu_2} + \cal E\,.
\end{equation}

\medskip
{\bf Step (vi).}
We fix $b_2, \dots, b_n$ and regard $b_1$ as a free index. Define
\begin{equation*}
X_{\emptyset}^w(\Delta) \;\equiv\; v_{b_1} \;\deq\; \sum_{\f a}^{(\mu_1 \mu_2) *} w_{b_1 b_2 \dots b_n}(\f a)  \, \cal Z_{a_1 \dots a_n}\,.
\end{equation*}
Then \eqref{self-const proof step 7} reads
\begin{equation}\label{vself}
v_{b_1} \;=\; (m^2 S v)_{b_1} + \cal E_{b_1}\,,
\end{equation}
where $\cal E_{b_1} \prec \Psi^{n + 2} \Phi^{n - 1}$. (Here we use the notation $\cal E_{b_1}(b_2, \dots, b_n, \mu_1, \mu_2) \equiv \cal E(b_1, \dots, b_n,\mu_1,\mu_2)$, indicating that $b_1$ is the variable index and
all other indices are fixed for this argument.)
Inverting the self-consistent equation yields
\begin{equation*}
v_{b_1} \;=\; \pb{(1 - m^2 S)^{-1} \cal E}_{b_1}\,.
\end{equation*}
In order to complete the proof, we observe that if $X_j^k \prec \Psi$ uniformly in $(j,k)$, then for any matrix $A = (A_{ij})$ we have $\sum_j A_{ij} X_j^k \prec \norm{A}_{\ell^\infty \to \ell^\infty} \Psi$ uniformly in $(i,k)$.
Recalling the definition \eqref{def of rho}, we therefore get
\begin{equation*}
X_{\emptyset}^w(\Delta) \;=\; \pb{(1 - m^2 S)^{-1} \cal E}_{b_1} \;\prec\; \varrho \, \Psi^{n + 2} \Phi^{n - 1} \;\leq\; \Psi^{n + 1} \Phi^n\,.
\end{equation*}
The last inequality is valid only if $\Phi<1$, but the final bound is still correct even if $\Phi = 1$ by using the trivial bound $X_{\emptyset}^w(\Delta) \prec \Psi^{n + 1}$.
This concludes the proof.
\end{proof}

\subsection{Completion of the induction and the proof of Proposition \ref{lemma: weak Z lemma}} \label{sect: completion of induction}
We may now complete the proof of Proposition \ref{lemma: weak Z lemma}. We begin with open chains. As outlined in Section \ref{section: Z1}, the proof is by induction on $\deg(\Delta)$. The induction is started with the trivial open chain $\Delta$ corresponding to $\cal Z = G_{\mu \nu}$, for which we have the trivial bound $G_{\mu \nu} \prec \Psi$. Then \eqref{bound for closed chains} for an arbitrary open chain $\Delta$ follows by induction, using Propositions \ref{prop:step2} and \ref{prop:step1}.

In order to prove \eqref{bound for closed chains} for an arbitrary closed chain $\Delta$, we follow almost to the letter the arguments from Sections \ref{section: Z4} and \ref{sec: chains with no Q}. The proof consists of two steps, each repeated twice.
\begin{enumerate}
\item[(a)]
Prove \eqref{bound for closed chains} for $F \neq \emptyset$.
\item[(b)]
Prove \eqref{bound for closed chains} for $F = \emptyset$.
\end{enumerate}
The order of the argument is as follows. First we do step (a) for closed chains with one external vertex, then
step (b) for closed chains with one external vertex, then
step (a) again but now for closed chains with no external vertex,
and finally step (b) for closed chains with no external vertex.
Here no induction is required; the necessary input is \eqref{bound for closed chains} for arbitrary open chains. Each one of the four above steps uses the previous ones as input. The proof of either step (a) is almost identical to that of Proposition \ref{prop:step2}, and the proof of either step (b) almost identical to that of Proposition \ref{prop:step1}. The only nontrivial difference is associated with coinciding indices, where we may lose a factor $\Psi^2 \Phi^2$ if two indices coincide. Here the worst case is the closed chain of degree two with no external vertices: $\sum_{a,b} s_{\mu a} s_{\nu b} G_{ab} G_{ba}$. The associated monomial $G_{ab} G_{ba}$ is of degree two and has two chain vertices. However, setting $a = b$ yields a contribution of order $M^{-1}$, i.e.\  we lost a factor $\Psi^2$ (from the two off-diagonal entries that became diagonal) and $\Phi^2$ (from the two chain vertices). However, this loss is compensated by the gain $M^{-1}$: we get the bound $\Psi^2 \Phi^2 + M^{-1} \leq 2 \Psi^2 \Phi^2$.

Let us sketch the general cases. Consider a closed chain $\Delta$ with no external vertices. Thus, $\Delta$ has $c(\Delta) = \deg(\Delta)$ chain vertices. If we ignore coinciding indices, we get the bound $\Psi^{\deg(\Delta)} \Phi^{\deg(\Delta)}$ on its size. On the other hand, if all of the $\deg(\Delta)$ indices coincide, we get the bound $M^{-\deg(\Delta) + 1}$. Indeed, all but one of the entries of $S$ in the chain weight can be estimated using their maximum $M^{-1}$; moreover, all resolvent entries are diagonal and hence of size $1$. This yields the combined bound
\begin{equation} \label{chain estimate 1}
\Psi^{\deg(\Delta)} \Phi^{\deg(\Delta)} + M^{-\deg(\Delta) + 1} \;\asymp\; \Psi^{\deg(\Delta)} \Phi^{\deg(\Delta)}\,,
\end{equation}
where we used that $\deg(\Delta)/2 \geq \deg(\Delta) - 1$. This is \eqref{bound for closed chains}.

The case of a closed chain $\Delta$ with one external vertex is similar. In this case we have $c(\Delta) = \deg(\Delta) - 1$. Ignoring coinciding indices, we get the bound $\Psi^{\deg(\Delta)} \Phi^{\deg(\Delta) - 1}$. On the other hand, if all indices coincide we get the bound $M^{-\deg(\Delta) + 1}$ exactly as before. This yields the combined bound
\begin{equation} \label{chain estimate 2}
\Psi^{\deg(\Delta)} \Phi^{\deg(\Delta) - 1} + M^{-\deg(\Delta) + 1} \;\asymp\; \Psi^{\deg(\Delta)} \Phi^{\deg(\Delta) - 1}\,,
\end{equation}
where we used \eqref{admissible Psi}. This is \eqref{bound for closed chains}.

Note that in the bounds \eqref{chain estimate 1} and \eqref{chain estimate 2} we only considered the two extreme cases: when all summation indices are distinct, and when they all coincide. In Lemma \ref{corollary: weak Z lemma}, we prove that these bounds in fact cover all possible index configurations. The full details on coinciding indices are given in Section \ref{section: Z8}. This concludes the proof of \eqref{bound for closed chains}, and hence of Proposition \ref{lemma: weak Z lemma}.

\section{General monomials and vertex resolution} \label{section: Z5}

In this section we conclude the proof of Theorem \ref{theorem: Z lemma} for general $\Delta$ under certain simplifying assumptions. A sketch of the argument presented in this section was given in Section \ref{sec: sketch of resolution}; the key new concept is that of \emph{vertex resolution}, which relies on Family B identities.

Our starting point is the graphical expansion from Section \ref{section: Z3} as well as the chain estimates from Proposition \ref{lemma: weak Z lemma}. Throughout this section we assume Simplification {\bf (S1)} from Section \ref{section: Z3}. Moreover, we shall tacitly make use of Lemma \ref{lemma: basic properties of prec} as well as the notations and definitions from Section \ref{section: Z3}.

In order to perform the vertex resolution, it will prove necessary to expand all resolvent entries in \emph{all} of the summation indices $\f a$ instead of the smaller set $\f a_F$ (as was done in Section \ref{section: Z3}). Thus, as first step, we repeat the construction of Section \ref{section: Z3}: we start with the graph $\gamma^p(\Delta)$ that encodes the $p$-th moment of $X_F^w(\Delta)$, and perform the expansion given after Definition \ref{definition: maximally expanded}, except we now expand in the full set $\f a$ of summation indices instead of $\f a_F$. This gives rise to a family of graphs which we denote by $\fra G_F^p(\Delta)$. Note that $\fra G_F^p(\Delta) \supset \wt {\fra G}_F^p(\Delta)$, where, we recall, the set $\wt {\fra G}_F^p(\Delta)$ is the set generated in Section \ref{section: Z3} by expanding in the indices $\f a_F$ only. Thus, each graph $\Gamma \in \fra G_F^p(\Delta)$ encodes a monomial of entries of $\cal G$, the edge $(x,y)$ giving rise to the maximally expanded entry $\cal G_{xy}^{(\f a \setminus \{x,y\})}$ or $\cal G_{xy}^{(\f a \setminus \{x,y\}) *}$ depending on its colour. Here, and throughout the following, we use the phrase \emph{maximally expanded} to mean maximally expanded in $\f a$ (see Definition \ref{definition: maximally expanded}). As in Section \ref{section: Z3}, we do not keep track of the $Q$'s in our notation. (Indeed, this information will turn out to be unimportant for our proof.)

Note that Definition \ref{definition: basics of Gamma} carries over verbatim for $\Gamma \in \fra G_F^p(\Delta)$. For the following we pick and fix a $\Gamma \in \fra G_F^p(\Delta)$. Thus, $\Gamma$ encodes a monomial, whereby each edge of $\Gamma$ gives rise to a maximally expanded entry of $\cal G$. As explained in Section \ref{section: Z3}, the linking procedure used to make all entries maximally expanded ensures, thanks to the presence of the $Q$'s, that $\abs{E(\Gamma)} \geq p (\deg(\Delta) + \abs{F})$.

The main idea of vertex resolution already appeared in Section \ref{sec: sketch of resolution}. Roughly, we \emph{resolve} each summation vertex using the Family B identities \eqref{res exp 2b} and \eqref{res exp 2b iterated}, which results in a new set of summation vertices which we call fresh and draw using white dots. We call the resulting graph $\Theta$. (More precisely, from each graph $\Gamma$ we get a finite family of new graphs $\{\Theta_\alpha\}$.) Next, we take the expectation, which results in a summation over all pairings (in fact, more generally, over all lumpings) of the white vertices adjacent to the original summation vertex. Each pairing gives rise to a new graph, which we call $\Upsilon$. (As above, from each graph $\Theta$ we get a finite family of new graphs $\{\Upsilon_\alpha\}$.) Although each steps results in
an increase in the number of graphs, it is easy to see that this combinatorial factor is bounded by 
a constant depending only on  $\abs{V(\Gamma)}$,  the number of vertices in $\Gamma$. In other words, the above families $\{\Theta_\alpha\}$ and $\{\Upsilon_\alpha\}$ are finite and do not depend on $N$.

The step $\Gamma \mapsto \Theta$ is performed in Section \ref{sect: generation of fresh vertices}, and the step $\Theta \mapsto \Upsilon$ in Section \ref{sec: vertex lumping}. Figure \ref{figure: flower} contains a summary of this process, on the level of a single vertex, which is helpful to keep in mind while reading the following. The idea is that, provided the vertex being resolved arose as a copy of a charged vertex (see Definition \ref{def: charged}), the resolution process will (in leading order) generate at least one fresh chain vertex. From this chain vertex we shall gain a factor $\Phi$ by invoking Proposition \ref{lemma: weak Z lemma}, and this will conclude the proof.
\begin{figure}[ht!]
\begin{center}
\includegraphics{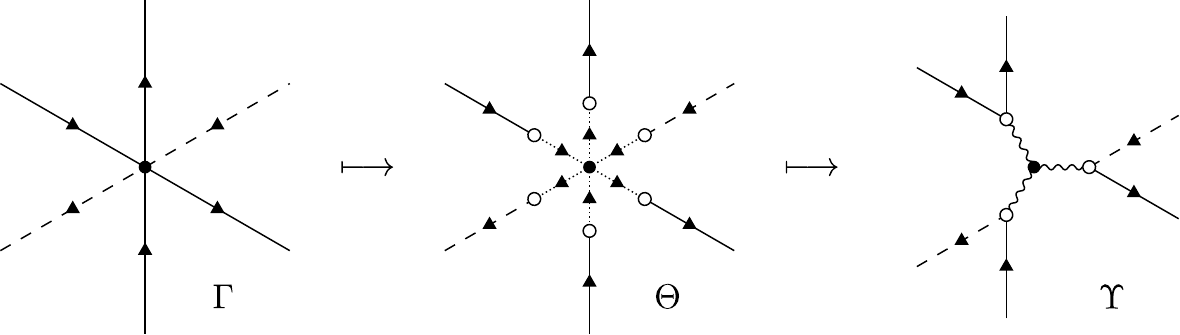}
\end{center}
\caption{A graphical overview of vertex resolution on the level of a single vertex. (We do not draw the other vertices.) In the second step
we draw only one of the possible six pairings (in the Hermitian case).
\label{figure: flower}}
\end{figure}

Note that the notion of charged vertex can be lifted from $\Delta$ to $\Gamma$ using the projection $\pi$ (see Definition \ref{definition of marked vertex} below).
The following class of vertices, called \emph{marked vertices}, plays the central role in this section. Informally, a marked vertex $i \in V_s(\Gamma)$ is a charged vertex that, in the construction of $\Gamma$ from $\gamma^p(\Delta)$, was linked to by the smallest allowed number of edges (zero if $\pi(i) \notin F$ and one if $\pi(i) \in F$).

Recall that we intend to gain an additional factor $\Phi$ from 
any charged vertex of $\Gamma$.
 However, the possibility of doing this may be destroyed if a charged
vertex was linked to in the construction of $\Gamma$ from $\gamma^p(\Delta)$. In this case we gain from the linking, as always, but we may not additionally
gain from the vertex's being charged. If the vertex was excessively linked 
(i.e.\ more than minimally required), then we have the additional gain from this
extra linking. Marked vertices are exactly those charged vertices which have been minimally linked. Thus, in order to gain from a marked vertex we cannot use the simple power counting that underlies linking, but need the more refined mechanism of vertex resolution.

\begin{definition}[Charged and Marked vertices in $\Gamma$] \label{definition of marked vertex}
The set of \emph{charged vertices} of $\Gamma$ is by definition $V_c(\Gamma) \deq \pi^{-1}(V_c(\Delta))$ (see Definition \ref{def: charged}).

The vertex $i \in V_c(\Gamma)$ is called \emph{marked} if either
\begin{enumerate}
\item
$\pi(i) \notin F$ and $\deg_\Gamma(i) = \deg_\Delta(\pi(i))$, or
\item
$\pi(i) \in F$ and $\deg_\Gamma(i) = \deg_\Delta(\pi(i)) + 2$.
\end{enumerate}
We denote the set of marked vertices by $V_m(\Gamma) \subset V_c(\Gamma)$.
\end{definition}
Note that (i) corresponds to the case where $i$ was not linked to at all, and (ii) to the case where $i$ was linked to exactly once. The following lemma gives a lower bound on the number of edges of $\Gamma$. Informally, it states that if $i$ is not marked but $\pi(i)$ is charged then $i$ was linked to at least once more than the minimum required amount (zero if $\pi(i) \notin F$ and one if $\pi(i) \in F$).
\begin{lemma} \label{lemma: lower bound on edges of Gamma}
We have the bound
\begin{equation} \label{lower bound on edges of Gamma}
\abs{E(\Gamma)} \;\geq\; p (\deg(\Delta) + \abs{F}) + \abs{V_c(\Gamma)} - \abs{V_m(\Gamma)}\,.
\end{equation}
\end{lemma}
\begin{proof}
By definition of $V_c(\Gamma)$ and $V_m(\Gamma)$, we find that $i \in V_s(\Gamma)$ was linked to at least once if
\begin{equation*}
i \;\in\; \pi^{-1}(F^c) \cap V_c(\Gamma) \cap V_m(\Gamma)^c \qquad \text{or} \qquad
i \;\in\; \pi^{-1}(F) \cap \pb{V_c(\Gamma) \cap V_m(\Gamma)^c}^c\,,
\end{equation*}
and $i$ was linked to at least twice if
\begin{equation*}
i \;\in\; \pi^{-1}(F) \cap V_c(\Gamma) \cap V_m(\Gamma)^c\,.
\end{equation*}
Since each linking adds an edge to the $p \deg(\Delta)$ edges of $\gamma^p(\Delta)$, we find
\begin{align*}
\abs{E(\Gamma)} &\;\geq\; p \deg(\Delta) + \absB{\pi^{-1}(F^c) \cap V_c(\Gamma) \cap V_m(\Gamma)^c} + \absB{\pi^{-1}(F) \cap \pb{V_c(\Gamma) \cap V_m(\Gamma)^c}^c}
\\
&\mspace{40mu}+ 2 \absB{\pi^{-1}(F) \cap V_c(\Gamma) \cap V_m(\Gamma)^c}
\\
&\;=\; p \deg(\Delta) + p \abs{F} + p \abs{V_c(\Delta)} - \abs{V_m(\Gamma)}\,,
\end{align*}
where in the last step we used that $\pi$ is $p$-to-one and $V_m(\Gamma) \subset V_c(\Gamma)$.
\end{proof}

The goal of this section is to gain an extra factor $\Phi$ from each marked vertex of $\Gamma$ using vertex resolution. Provided we can do this, the proof of Theorem \ref{theorem: Z lemma} will be complete. This can be informally understood as follows. In order to get the estimate \eqref{main Z lemma estimate}, we need a bound of size $\Psi^{p (\deg(\Delta) + \abs{F})} \Phi^{p\abs{V_c(\Delta)}}$. Each vertex $i \in V_s(\Gamma)$ contributes factors $\Psi$ and $\Phi$ (in addition to the trivial $p \deg(\Delta)$) to the estimate as follows.
\begin{enumerate}
\item
$\pi(i) \notin F$ and $i \notin V_c(\Gamma)$. In this case $i$ yields no factor $\Psi$ or $\Phi$.
\item
$\pi(i) \notin F$ and $i \in V_c(\Gamma)$. If $\deg_\Gamma(i) = \deg_\Delta(\pi(i))$ then $i$ is marked and will yield a factor $\Phi$ by vertex resolution. If $\deg_\Gamma(i) > \deg_\Delta(\pi(i))$ then $i$ is not marked but carries an extra factor $\Psi$ since it has been linked to more times than needed. (Thus, $\Gamma$ has at least one extra edge, corresponding to an off-diagonal entry $G_{uv} \prec \Psi$, incident to $i$).
\item
$\pi(i) \in F$ and $i \notin V_c(\Gamma)$. In this case $i$ has been linked to at least once and is consequently incident to at least one extra edge. This yields a factor $\Psi$.
\item
$\pi(i) \in F$ and $i \in V_c(\Gamma)$. As in (iii), $i$ has been linked to at least once and hence yields a factor $\Psi$. In addition, $i$ yields a second factor $\Phi$ as follows. If $\deg_\Gamma(i) = \deg_\Delta(\pi(i)) + 2$ then $i$ is marked and will yield an extra factor $\Phi$ by vertex resolution. If $\deg_\Gamma(i) > \deg_\Delta(\pi(i)) + 2$ then $i$ has been linked to at least twice, hence yielding a second factor $\Psi$. In either case the vertex $i$ generates a factor $\Psi \Phi$.
\end{enumerate}
Thus, from each case (i) -- (iv) we gain $\ell_\Psi$ factors $\Psi$ and $\ell_\Phi$ factors $\Phi$ in addition to the trivially available $p \deg(\Delta)$ factors of $\Psi$, where the values of $\ell_\Psi$ and $\ell_\Phi$ is as follows: (i) $\ell_\Psi = \ell_\Phi = 0$, (ii) $\ell_\Psi = 0$, $\ell_\Phi = 1$, (iii) $\ell_\Psi = 1$, $\ell_\Phi = 0$, (iv) $\ell_\Psi = \ell_\Phi = 1$. From this the bound $\Psi^{p (\deg(\Delta) + \abs{F})}\Phi^{p\abs{V_c(\Delta)}}$ follows immediately.

Before moving on to the main argument of this section, we outline how a marked vertex yields an additional factor $\Phi$. We claim that Definitions \ref{def: charged} and \ref{definition of marked vertex} imply that
\begin{equation} \label{condition for marked vertices}
i \in V_s(\Gamma) \text{ is marked} \qquad \Longrightarrow \qquad \nu_i(\Gamma) \;\neq\; \nu_i^*(\Gamma)
\end{equation}
(see Definition \ref{def: incidence indices}). To see this, let $i$ be an arbitrary marked vertex. If $\pi(i) \notin F$ then by Definition \ref{definition of marked vertex} $i$ has not been linked to, and hence $\nu_i^\xi(\Gamma) = \nu_{\pi(i)}^\xi(\Delta)$ for $\xi = 1, *$. Therefore \eqref{condition for marked vertices} follows from Definition \ref{def: charged}. On the other hand, if $\pi(i) \in F$ then by Definition \ref{definition of marked vertex} $i$ has been linked to once, and either (a) $\nu_i(\Gamma) = \nu_{\pi(i)}(\Delta)$ and $\nu^*_i(\Gamma) = \nu^*_{\pi(i)}(\Delta) + 2$ or (b) $\nu_i(\Gamma) = \nu_{\pi(i)}(\Delta) + 2$ and $\nu^*_i(\Gamma) = \nu^*_{\pi(i)}(\Delta)$. Either way, Definition \ref{def: charged} yields \eqref{condition for marked vertices}.

Roughly, vertex resolution splits each summation vertex of degree $2d$ (we assume for simplicity that the vertex is of even degree) into $d$ fresh summation vertices of degree two. If the right-hand side of \eqref{condition for marked vertices} holds (as it does if we are resolving a marked vertex), at least one of the fresh summation vertices will be a chain vertex (i.e.\ both of its incident edges will have the same colour). The desired gain of a factor $\Phi$ will then come from an application of Proposition \ref{lemma: weak Z lemma}.

\subsection{General graphical representation} \label{sect: graphs}
Throughout Sections \ref{section: Z5} and \ref{section: simplifications}, we shall fix $\f a$ and apply three algebraic operations to the monomial encoded by $\Gamma$:
\begin{enumerate}
\item
Family A identities,
\item
Family B identities,
\item
partial expectation in $\f a$.
\end{enumerate}
In order to keep track of the structure of the ensuing terms, we make heavy use of graphs. For pedagogical reasons, we shall develop the algebraic and graphical languages in parallel. Each algebraic expression (a monomial in the matrix entries of $G$, $\cal G$, $H$, and $S$) is represented by a graph. Application of one of the three elementary algebraic identities listed above can, as before, be described by a elementary transformation on graphs.
Although the entire argument could be stated in terms of graphs alone, this would rather obscure the underlying mechanism, which always corresponds to applying one of the three algebraic operations listed above. Instead, we introduce each graph operation when it naturally arises in our argument. In order to obtain a set of graphs that is closed under all of the operations we shall need, we extend our set of graphs according to the following definition.

\begin{definition}[General graph] \label{def: general graphs}
By a \emph{graph} we mean a quintuple $(V_f, V_s, V_e, E, \xi)$ with the following properties. The set $E$ is a set of edges on the vertex set $V \deq V_f \sqcup V_s \sqcup V_e$. Multiple edges as well as loops are allowed. The colouring $\xi$ is a map
\begin{equation*}
\xi \col E \; \longrightarrow\; \hb{\text{solid},\, \text{dashed}, \,\text{dotted},\, \text{wiggly}}\,.
\end{equation*}
As in Definition \ref{definition: Delta}, we sometimes use the alternative notations $\text{solid} \equiv 1$ and $\text{dashed} \equiv *$.

An edge that is solid or dashed is called a \emph{resolvent edge}. Dotted and resolvent edges are directed, while wiggly edges are undirected. The vertices in $V_s$ are called the \emph{original summation vertices}, in $V_f$ the \emph{fresh summation vertices}, and in $V_e$ the \emph{external vertices}. Vertices in $V_f$ are drawn using white dots and vertices in $V_s \sqcup V_e$ using black dots.
\end{definition}
Figure \ref{figure: h and s} contains the dictionary of the colour-code: a solid edge encodes an entry of $G$, a dashed edge an entry of $G^*$, a dotted edge an entry of $H$, and a wiggly edge an entry of $S$. More precisely, each resolvent edge encodes a \emph{maximally expanded} (in ${\f a}$) resolvent entry (see Definition \ref{definition: maximally expanded}), and each dotted edge an \emph{$\f a$-admissible entry of $H$}, which is the subject of the following definition.
\begin{definition} \label{admissible h}
The entry $h_{uv}$ is an \emph{$\f a$-admissible entry of $H$} if $u \in \f a$ or $v \in \f a$.
\end{definition}
The arguments below consist of a series of operations on the set of graphs from Definition \ref{def: general graphs}. To be completely precise, below we shall in fact adorn the graphs from Definition \ref{def: general graphs} with \emph{decorations}: resolvent loops may be decorated with a black or white diamond (see Figure \ref{figure: simple graph for variance}), wiggly edges with an arbitrary number of crossing strokes, and original summation vertices with an arbitrary number of rings. (The latter two concepts are defined above Definition \ref{def: lumping} and in the beginning of the proof of Lemma \ref{lm:stage2} respectively.)

To each vertex $i \in V(\Gamma)$ of a graph $\Gamma$ we assign an index $u_i$. We shall consistently use the splitting
\begin{equation*}
\f u \;=\; (u_i)_{i \in V(\Gamma)} \;=\; (\f x, \f a, \f \mu)\,, \qquad \f x \;=\; (x_i)_{i \in V_f(\Gamma)} \,, \qquad \f a \;=\; (a_i)_{i \in V_s(\Gamma)}\,, \qquad \f \mu \;=\; (\mu_i)_{i \in V_e(\Gamma)}\,.
\end{equation*}
We say that the index $u_i$ is \emph{associated with} the vertex $i$, and vice versa.
We introduce the notation 
\begin{equation} \label{Max}
\cal A(\Gamma) \;\equiv\; \cal A_{\f a, \f x}(\Gamma)
\end{equation}
for the monomial (in the entries of $G^{(T)}$, $\cal G^{(T)}$, $H$, and $S$ where $T \subset \f a$) encoded by the graph $\Gamma$. Note that $\cal A(\Gamma)$ has an explicit formula, analogous to \eqref{definition of Z} except needing much heavier notation. As we shall not need it, we shall not give it. (Note also that our graphs do not keep track of any factors of $Q$, as we shall not need this information.)

\subsection{Generation of the fresh summation vertices} \label{sect: generation of fresh vertices}
Next, we define the vertex resolution operation precisely. Our starting point is a fixed $\Gamma \in \fra G_F^p(\Delta)$.
In order to streamline the argument, we at first make the following simplifying assumption on $\Delta$, which is removed in Section \ref{section: simplifications}.
\begin{itemize}
\item[{\bf (S2)}]
There are no diagonal entries $\cal G_{aa} = G_{aa} - m$ in $\cal Z(\Delta)$. (I.e.\ $\Delta$ has no loops.)
\end{itemize}
The operation of vertex resolution consists of two main steps: the \emph{generation} and \emph{lumping} of fresh summation vertices. The idea behind the first step -- the generation of fresh summation vertices -- is to resolve, using the Family B identities \eqref{res exp 2b} and \eqref{res exp 2b iterated}, each (already maximally expanded) off-diagonal resolvent entry $G_{uv}^{(\f a \setminus \{u,v\})}$, with $u \neq v$, in any summation index from the set $\{u,v\}$.
(The word ``resolve'' here refers to explicitly identifying the dependence on all matrix entries $h_{ij}$ with $i,j \in \f a$ so that partial expectation in these variables can be taken. After taking the partial expectation, we shall get expressions that can again be represented by admissible graphs.) More precisely, we write
\begin{equation} \label{identity for resolution}
G^{(\f a \setminus \{u,v\})}_{uv} \;=\;
\begin{cases}
- G_{uu}^{(\f a \setminus \{u\})} \sum_x^{(\f a)} h_{u x} G_{x v}^{(\f a)} & \text{($u$ summation and $v$ external)}\,,
\\
- G_{vv}^{(\f a \setminus \{v\})} \sum_x^{(\f a)} G_{u x}^{(\f a)} h_{xv} & \text{($u$ external and $v$ summation)}\,,
\\
G_{uu}^{(\f a \setminus \{u,v\})} G_{vv}^{(\f a \setminus \{v\})} \pb{-h_{uv} + \sum_{x,y}^{(\f a)} h_{ux} G_{xy}^{(\f a)} h_{xv}} & \text{($u$ and $v$ summation)}\,.
\end{cases}
\end{equation}
(As stated after Definition \ref{definition: Delta}, we exclude the trivial case where both $u$ and $v$ are external indices.)
The proof of \eqref{identity for resolution} is a straightforward consequence of the identities \eqref{res exp 2b} and \eqref{res exp 2b iterated}, and the fact that $u \neq v$ by definition of $X_F^w(\Delta)$. For example, if $a$ and $b$ are summation indices and $\mu$ and $\nu$ are external indices, we write
\begin{multline} \label{resolution example}
G_{\mu a}^{(b)} G_{ab} G_{b \mu}^{(a)*} G_{a \nu}^{(b)} G_{\nu a}^{(b)*}
\\
=\; G_{aa}^{(b)} G_{aa} G_{bb}^{(a)} G_{bb}^{(a) *} G_{aa}^{(b)} G_{aa}^{(b)*} \sum_{x,y,z,u,v,w}^{(ab)} G_{\mu x}^{(ab)} h_{xa} h_{a y} G^{(ab)}_{yz} h_{zb} h_{b u} G_{u \mu}^{(ab)*} h_{a v} G^{(ab)}_{v \nu} G_{\nu w}^{(ab)*} h_{wa}
\\
- G_{aa}^{(b)} G_{aa} G_{bb}^{(a)} G_{bb}^{(a) *} G_{aa}^{(b)} G_{aa}^{(b)*} \sum_{x,u,v,w}^{(ab)} G_{\mu x}^{(ab)} h_{xa} h_{ab} h_{b u} G_{u \mu}^{(ab)*} h_{a v} G^{(ab)}_{v \nu} G_{\nu w}^{(ab)*} h_{wa}\,.
\end{multline}
Notice that along this procedure we may generate non-maximally expanded diagonal terms, (e.g.\ $G_{aa}$ above), but off-diagonal
terms always have upper indices $\f a$. Moreover, all entries of $H$ on the right-hand side of \eqref{resolution example} are $\f a$-admissible.

At this point we make the following further simplification, which leaves the essence of the argument unchanged but removes some technicalities.
\begin{itemize}
\item[{\bf (S3)}]
We replace any diagonal term $G_{aa}^{(T)}$ with $m$ and any diagonal term $G_{aa}^{(T)*}$ with $\bar m$. (Recall that $G_{aa}^{(T)} \approx m$ in the sense that $G_{aa}^{(T)} - m \prec \Psi$ by definition of $\Lambda$.) This replacement is done in two places: in $\cal A(\Gamma)$ and in the identities \eqref{res exp 2b} and \eqref{res exp 2b iterated} which underlie the algebra of vertex resolution.
\end{itemize}
Again, Simplification {\bf (S3)} is removed in Section \ref{section: simplifications}. Thus, under Simplification {\bf (S3)}, we neglect all diagonal terms in \eqref{resolution example}. (More precisely, we replace each one of them with $m$ or $\bar m$; the resulting powers of $m$ and $\bar m$ are irrelevant for estimates by \eqref{m is bounded}.)

The use of graphs greatly simplifies the analysis of complicated expressions like \eqref{resolution example}. The identities \eqref{identity for resolution} all have obvious graphical representations. In Figure \ref{figure: vertex resolution} we give a graphical depiction of \eqref{resolution example}. (Recall the conventions introduced in Figure \ref{figure: h and s}.)
In the typical case, the summation vertex of degree
four associated with $a$ gives rise to four fresh summation vertices, associated with $x, y, v, w$; likewise the summation vertex of degree two associated with $b$ creates two fresh vertices associated with $u$ and $z$ (first term in the right side of \eqref{resolution example}).
Due to presence of the term $h_{uv}$  in \eqref{identity for resolution},
sometimes two summation indices are directly connected with a dotted line at the expense
of one fewer fresh summation vertex adjacent to each of these two indices. This results in the second term on the right-hand side of \eqref{resolution example}, which contains a factor $h_{ab}$ (note that $G_{ab}$ plays the role
of  $G_{uv}^{\f a \setminus \{ u,v\}}$ from \eqref{identity for resolution}).
\begin{figure}[ht!]
\begin{center}
\includegraphics{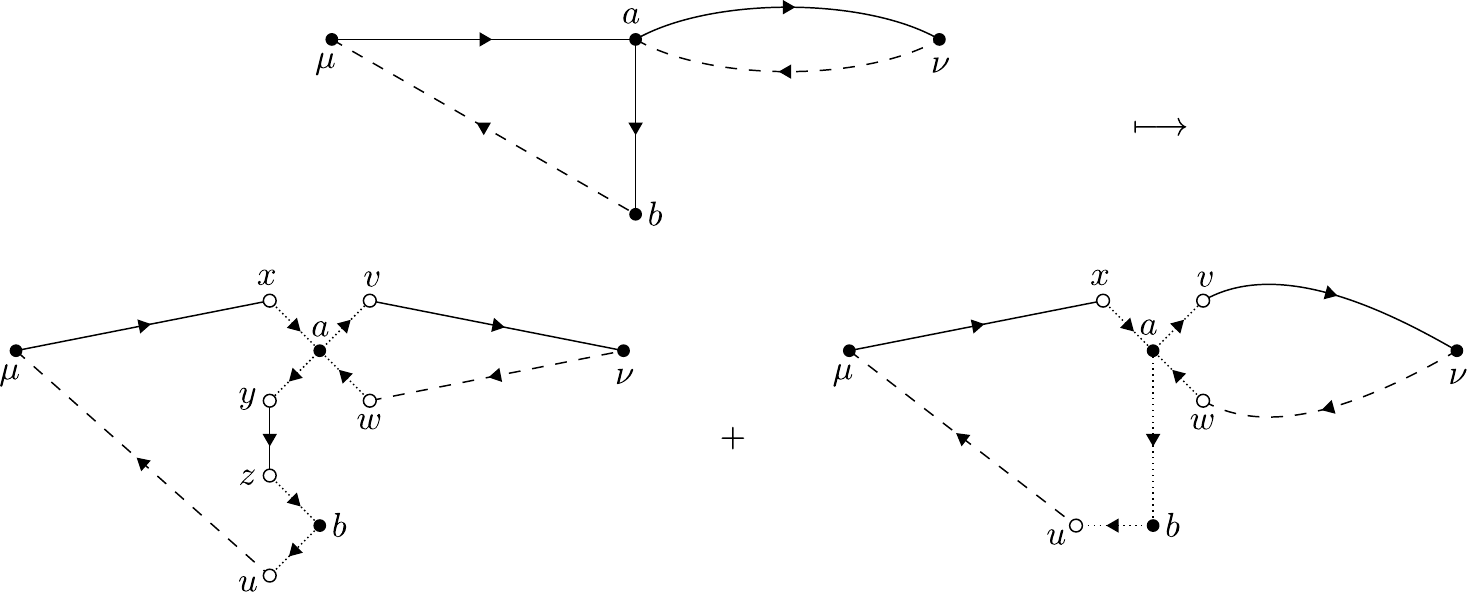}
\end{center}
\caption{The vertex resolution from \eqref{resolution example}. The graph $\Gamma$ is represented on the top and both graphs of $\wt {\fra R}(\Gamma)$ are represented on the bottom. In accordance with Simplification {\bf (S3)} we do not draw the diagonal entries of $G$. \label{figure: vertex resolution}}
\end{figure}

The generation of fresh summation vertices for a general $\Gamma \in \fra G_F^p(\Delta)$ is no different from the above example. Applying the graphical rules associated with \eqref{identity for resolution} to each edge of $\Gamma$, we get a finite family of graphs which we denote by $\wt {\fra R}(\Gamma)$. In accordance with Simplifications {\bf (S2)} and {\bf (S3)}, in this section we drop all diagonal terms, and hence all loops from the graphs in $\wt {\fra R}(\Gamma)$. (In Section \ref{section: simplifications} below we keep track of the loops, which will lead to the larger set $\fra R$.)

Any graph $\Theta \in \wt {\fra R}(\Gamma)$ has the following properties.

\begin{enumerate}
\item
$\Theta$ has only straight, dashed or dotted edges but no wiggly edge
(see Figure \ref{figure: h and s}).
\item
$\Theta$ has no loops or multiple edges.
\item
The sets $V_s(\Theta) = V_s(\Gamma)$ and $V_e(\Theta) = V_e(\Gamma)$ remain unchanged, as do the associated indices $\f a$ and $\f \mu$. In addition, we now have a new set of vertices, $V_f(\Theta) \neq \emptyset$, which indexes the fresh summation vertices $\f x$.
\item
Each resolvent edge of $\Theta$ encodes an entry of $G^{(\f a)}$ or $G^{(\f a)*}$ in $\cal A(\Theta)$ (in particular, no resolvent edge of $\Theta$ is incident to $V_s(\Theta)$). Each dotted edge of $\Theta$ encodes an $\f a$-admissible entry of $H$ in $\cal A(\Theta)$.
\item
For $i,j \in V_s(\Theta)$ let the symmetric function $\sigma(i,j)$ denote the number of dotted edges joining $i$ and $j$. The number of edges in $\Gamma$ and $\Theta$ is conserved in the sense that
\begin{equation} \label{edge conservation Theta}
\abs{E(\Gamma)} \;=\; \sum_{e \in E(\Theta)} \pb{\ind{\xi_e = 1} + \ind{\xi_e = *}} + \frac{1}{2}\sum_{i,j \in V_s(\Theta)} \sigma(i,j)\,.
\end{equation}
Informally: in the process $\Gamma \mapsto \Theta$ that generates fresh summation vertices, each resolvent entry either remains a resolvent entry or is replaced with an entry of $H$ with original summation vertices. This simply corresponds to the two terms in the last line of the right-hand side of \eqref{identity for resolution}.
\item
Each original summation vertex is incident only to dotted edges.
\item
Each fresh summation vertex has degree two and is incident to precisely one dotted edge.
\end{enumerate}

We remark that, by construction, the sets $\f x$ and $\f a$ are disjoint, as are the sets $\f a$ and $\f \mu$. However, $\f x$ and $\f \mu$ are in general not disjoint, and indices of $\f x$ may coincide.

\subsection{Lumping of the fresh summation vertices} \label{sec: vertex lumping}
We now take the expectation of $\cal A(\Theta)$. In fact, all that we shall need is the partial expectation in $\f a$. The key observation is that, by Property (iv) in Section \ref{sect: generation of fresh vertices}, each resolvent entry of $\cal A(\Theta)$ is independent of $\f a$ and each entry of $H$ is $\f a$-admissible. In particular, the partial expectation $\prod_{a \in \f a} P_a$ acting on $\cal A(\Theta)$ acts on the product of the entries of $H$ alone. Since this product is an explicit monomial, we can evaluate its expectation directly. If the random variables $h_{uv}$ were Gaussian, this would correspond to 
a simple Wick-pairing of the dotted edges. Pairing two dotted edges, each of them incident to a fresh summation vertex and a common original summation vertex, results in a pairing of two fresh summation vertices. Since $\f x$ and $\f a$ are distinct, a dotted edge incident to a fresh summation vertex cannot be paired with a dotted edge incident to two original summation vertices. In the non-Gaussian case, higher-order moments are also present, but they are suppressed by a combinatorial factor (i.e.\ a positive power of $M^{-1}$). Graphically, we represent the procedure of taking expectation by pairing (or in general lumping) fresh summation indices, and replace the corresponding doubled dotted line by a wiggly line. What follows is a more precise description.

We define the second step of vertex resolution -- the lumping of the fresh summation vertices. Before giving the general procedure, we complete the analysis of the example \eqref{resolution example}. From \eqref{resolution example}, assuming Simplification {\bf (S3)}, we get
\begin{align}
&\E G_{\mu a}^{(b)} G_{ab} G_{b \mu}^{(a)*} G_{a \nu}^{(b)} G_{\nu a}^{(b)*}
\notag \\
&\mspace{30mu}\overset{\mathrm{{\bf (S3)}}}{=}\; m^4 \bar m^2 \sum_{x,y,z,u,v,w}^{(ab)} \E G_{\mu x}^{(ab)} G^{(ab)}_{yz} G_{u \mu}^{(ab)*} G^{(ab)}_{v \nu} G_{\nu w}^{(ab)*} P_a \pb{h_{xa} h_{a y} h_{a v} h_{wa} } P_b (h_{zb} h_{b u}) 
\notag \\ \label{resolve step 1}
&\mspace{50mu}-  m^4 \bar m^2  \sum_{x,u,v,w}^{(ab)} \E G_{\mu x}^{(ab)} G_{u \mu}^{(ab)*} G^{(ab)}_{v \nu} G_{\nu w}^{(ab)*} h_{a v} h_{wa} h_{xa}  P_b \pb{h_{ab} h_{b u}}\,,
\end{align}
where we used the trivial identity $\E X = \E P_a P_b X$ and the fact that $G^{(ab)}$ is independent of $a$ and $b$ (see Definition \ref{definition: P Q}).
Since $a \neq u$, the partial expectation $P_b$ on the second line of \eqref{resolve step 1} vanishes. The partial expectations on the first line of \eqref{resolve step 1} may be computed explicitly, similarly to the computation \eqref{example pairings for S}:
\begin{multline*}
P_a \pb{h_{xa} h_{a y} h_{a v} h_{wa} } P_b (h_{zb} h_{b u})
\\
=\; s_{ax} s_{av} s_{bz} \delta_{xy} \delta_{vw} \delta_{zu} + s_{ax} s_{ay} s_{bz} \delta_{xv} \delta_{yw} \delta_{zu} + \ind{x = y = v = w} s_{ax}^2 s_{bz} \delta_{zu} \E \abs{\zeta_{ax}}^4\,.
\end{multline*}
Here we assumed the condition \eqref{CH}. In the case \eqref{RS}, we get the additional term
\begin{equation*}
s_{ax} s_{ay} \delta_{xw} s_{bz} \delta_{yv} \delta_{zu}\,.
\end{equation*}
Thus we may write (in the case \eqref{CH} for simplicity)
\begin{align}
\E G_{\mu a}^{(b)} G_{ab} G_{b \mu}^{(a)*} G_{a \nu}^{(b)} G_{\nu a}^{(b)*}
&\;\overset{\mathrm{{\bf (S3)}}}{=}\; m^4 \bar m^2 \E \sum_{x,v,z}^{(ab)} s_{ax} s_{av} s_{bz} \, G_{\mu x}^{(ab)} G^{(ab)}_{xz} G_{z \mu}^{(ab)*} G^{(ab)}_{v \nu} G_{\nu v}^{(ab)*}
\notag \\
&\mspace{30mu} +
m^4 \bar m^2 \E \sum_{x,y,z}^{(ab)} s_{ax} s_{ay} s_{bz} \, G_{\mu x}^{(ab)} G^{(ab)}_{yz} G_{z \mu}^{(ab)*} G^{(ab)}_{x \nu} G_{\nu y}^{(ab)*}
\notag \\ \label{resolve step 2}
&\mspace{30mu} +
m^4 \bar m^2 \E \sum_{x,z}^{(ab)} s_{ax}^2 s_{bz} \pb{\E \abs{\zeta_{ax}}^4} \, G_{\mu x}^{(ab)} G^{(ab)}_{xz} G_{z \mu}^{(ab)*} G^{(ab)}_{z \nu} G_{\nu x}^{(ab)*}\,.
\end{align}
Note that the summations on the right-hand side are performed with respect to weights (see Definition \ref{definition of weight}).
Now it is apparent how better estimates are available for each term on the right-hand side than the term on the left-hand side. Indeed, all terms contain five off-diagonal entries of $G$. In addition, however, the first two terms on the right-hand side contain a summation index associated with a chain vertex $x$ (see Definition \ref{def: chains}) which is summed over with respect to the chain weight $s_{ax}$ (see Definition \ref{def: chain weight}). The last term is suppressed by an additional factor $s_{ax} \leq M^{-1}$. Invoking Proposition \ref{lemma: weak Z lemma} (and neglecting the upper indices $(ab)$ which are dealt with easily in the full proof below), we find that the arguments of $\E$ on the right-hand side of \eqref{resolve step 2} are all $O_\prec(\Psi^5 \Phi)$. Here the extra factor $\Phi$ was extracted from the resolution of the vertex $i$ associated with the summation variable $a$, and arose from the fact that $\nu_i(\Gamma) \neq \nu_i^*(\Gamma)$.

Again, the lumping of fresh summation vertices is best represented graphically. In order to represent terms like the last term on the right-hand side of \eqref{resolve step 2} graphically, we represent the expression $\E \abs{h_{ij}}^{2 + k}$ using a wiggly line crossed by $k$ strokes. Thus, a wiggly edge may be either \emph{uncrossed} or \emph{crossed}. We shall always use the bound
\begin{equation*}
P_i \abs{h_{ij}}^{2 + k} \;=\; s_{ij}^{1 + k/2} \E \abs{\zeta_{ij}}^{2 + k} \;\leq\; s_{ij} M^{-k/2}
\end{equation*}
in combination with crossed wiggly lines. See Figure \ref{figure: pairing example} for a graphical depiction of \eqref{resolve step 2}.
\begin{figure}[ht!]
\begin{center}
\includegraphics{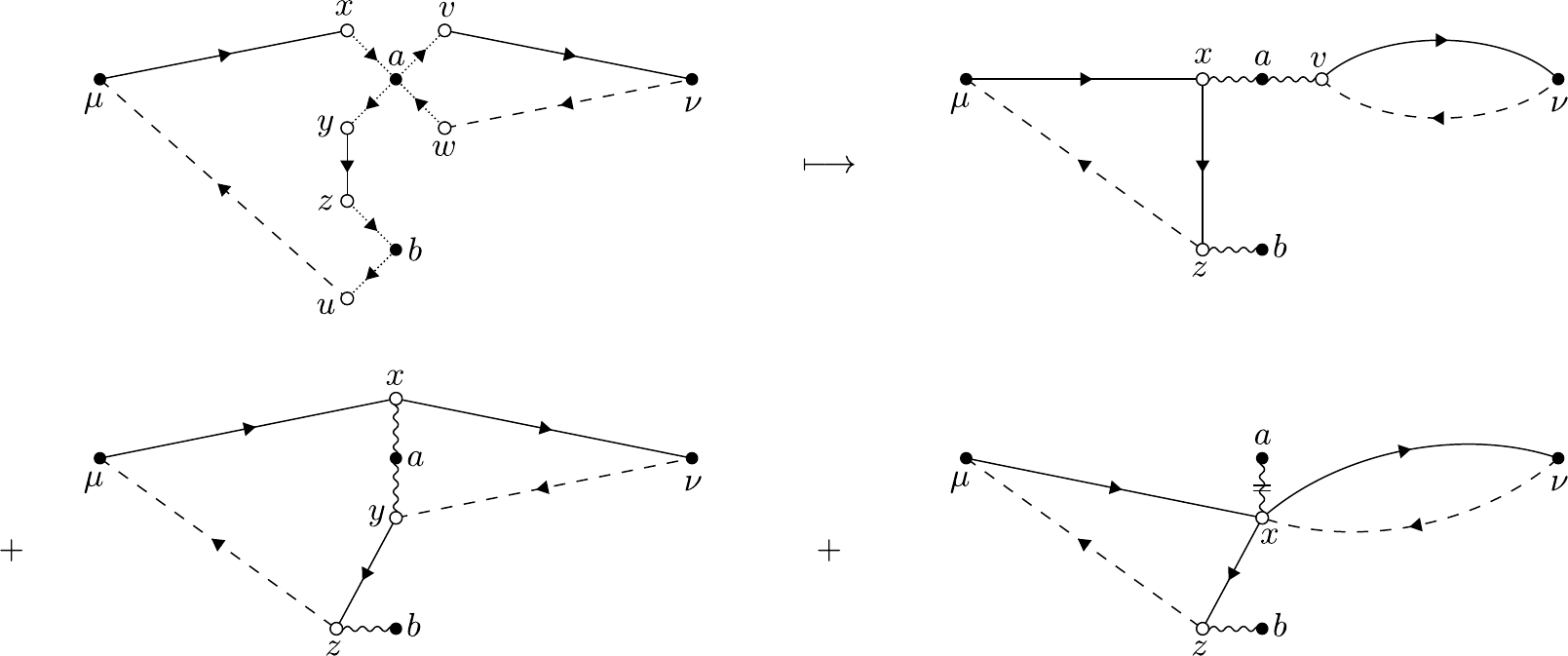}
\end{center}
\caption{The lumping of the fresh summation vertices in the example \eqref{resolution example}. The graph on the left-hand side is $\Theta$ (representing the first term on the right-hand side of \eqref{resolve step 1}), and the three graphs on the right-hand side are the elements of $\fra L(\Theta)$ (representing the three terms on the right-hand side of \eqref{resolve step 2}). \label{figure: pairing example}}
\end{figure}

The general case is similar to the example \eqref{resolve step 2}.
\begin{definition}[Lumping] \label{def: lumping}
A \emph{lumping} is a partition whose blocks, called \emph{lumps}, have size greater than or equal to two.
\end{definition}
Let $\Gamma \in \fra G_F^p(\Delta)$ and $\Theta \in \wt {\fra R}(\Gamma)$. From $\Theta$ we generate a finite family of graphs $\Upsilon$ by taking all lumpings of the white vertices (fresh summation vertices) of $\Theta$. (This lumping, or identification of vertices, arises from taking the partial expectation $\prod_{a \in \f a} P_a$ of all entries of $H$ in $\cal A(\Theta)$. Note that each fresh summation vertex of $\Theta$ must be lumped with at least another one because $h_{ij}$ and $h_{kl}$ are independent if $\{i,j\} \neq \{k,l\}$, and $P_i h_{ij} = 0$.) Thus, the result of the lumping is to merge some white vertices into lumps, where each lump consists of at least two vertices, and is again represented as a single white vertex. We denote by $\fra L(\Theta)$ the set of such graphs $\Upsilon$ obtained by lumping the fresh summation vertices of $\Theta$. Recall that all resolvent matrix entries encoded by the resolvent edges of $\Theta$ have upper indices $\f a$. Hence any factors $Q_a$ with $a \in \f a$ act trivially on them according to the identity $Q_a \pb{G^{(\f a)}_{ij} X} = G^{(\f a)}_{ij} Q_a X$ for all $a \in \f a$. Therefore the factors $Q$ only act on entries of $H$. As in the example \eqref{simple vertex resolution organized}, they simply forbid some pairings; this restriction is no importance for us, and we shall estimate the contribution of arbitrary pairings. Note that after the partial expectation in $\f a$ has been taken, no factors $Q$ remain. In particular, $\cal A(\Upsilon)$ has no factors $Q$.

Any graph $\Upsilon \in \fra L(\Theta)$ has the following properties.
\begin{enumerate}
\item
$\Upsilon$ has no loops or multiple edges.
\item
The vertex set of $\Upsilon$ is a disjoint union $V(\Upsilon) = V_f(\Upsilon) \sqcup V_s(\Upsilon) \sqcup V_e(\Upsilon)$, where $V_s(\Upsilon) = V_s(\Theta) = V_s(\Gamma)$ is the set of original summation vertices, and $V_e(\Upsilon) =  V_e(\Theta) = V_e(\Gamma)$ is the set of external summation vertices. (The set of fresh summation vertices $V_f(\Upsilon)$ is strictly smaller than $V_f(\Theta)$.)
\item
A directed edge $(i,j)$ of $\Theta$ has one of two colours: solid (encoding $G_{ij}^{(\f a)}$) or dashed (encoding $G_{ij}^{(\f a)*}$). An undirected edge is always wiggly. An uncrossed wiggly edge $\{i,j\}$ encodes $\E \abs{h_{ij}}^2 = s_{ij}$, and a wiggly edge $\{i,j\}$ crossed by $k$ strokes encodes $\E \abs{h_{ij}}^{2 + k} = s_{ij}^{1 + k/2} \E \abs{\zeta_{ij}}^{2 + k}$.
\item
Each $i \in V_f(\Upsilon)$ is adjacent (via a wiggly line) to a unique vertex $p(i) \in V_s(\Upsilon)$. Thus, the map $p \col V_f(\Upsilon) \to V_s(\Upsilon)$ is a projection which associates with each fresh summation vertex $i$ its ``parent'' original summation vertex $p(i)$.
(E.g.\ in the first graph on the right-hand side of Figure \ref{figure: pairing example} we have $p(x)=p(v)=a$ and $p(z)=b$.)
\item
For each $i,j \in V_s(\Upsilon)$, we have $\sigma(i,j) \in \{0,2,3,4,\dots\}$. (Recall that $\sigma(i,j)$ is the number of dotted lines in the graph $\Theta$
between vertices $i,j \in V_s(\Upsilon)=V_s(\Theta)$.) If $\sigma(i,j) \geq 2$ then $i$ and $j$ are connected by a wiggly line crossed by $\sigma(i,j) - 2$ strokes. Moreover, the number of edges is conserved in the sense that
\begin{equation} \label{edge conservation Upsilon}
\abs{E(\Gamma)} \;=\; \sum_{e \in E(\Upsilon)} \pb{\ind{\xi_e = 1} + \ind{\xi_e = *}} + \frac{1}{2}\sum_{i,j \in V_s(\Upsilon)} \sigma(i,j)\,,
\end{equation}
as follows from \eqref{edge conservation Theta}.
\end{enumerate}
See Figure \ref{figure: nonzero sigma} for an illustration of (v).

\begin{figure}[ht!]
\begin{center}
\includegraphics{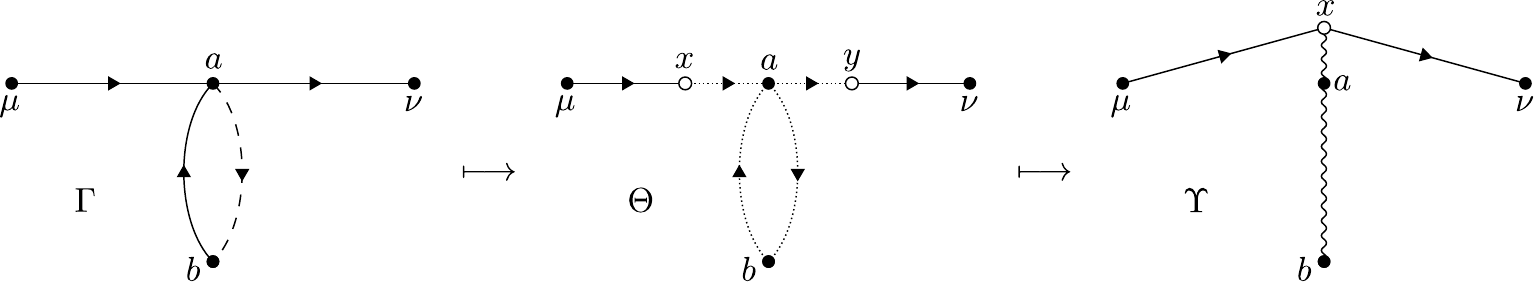}
\end{center}
\caption{The resolution process $\Gamma \mapsto \Theta \mapsto \Upsilon$, where in the first step we chose a $\Theta$ satisfying $\sigma(i,j) = 2$ where $i$ and $j$ are the original summation vertices associated with $a$ and $b$ respectively. This figure illustrates the property (v) above, as well as the conservation of the number of edges from \eqref{edge conservation Theta} and \eqref{edge conservation Upsilon}.
\label{figure: nonzero sigma}}
\end{figure}

\subsection{Summing over the fresh summation indices and completion of the proof} \label{sec: completion of proof under all S}
The complete process of vertex resolution may be summarized as
\begin{equation*}
\Gamma \in \fra G_F^p(\Delta) \quad \longmapsto \quad \Theta \in \wt {\fra R}(\Gamma) \quad \longmapsto \quad \Upsilon \in \fra L(\Theta)\,.
\end{equation*}

Here the first step represents explicitly resolving all maximally expanded entries of
$G$ (encoded by resolvent edges of $\Gamma$) in $\f a$-admissible entries of $H$.
The second step represents taking partial expectation in all these $h$-variables.

At this point, we introduce a further, and final, simplification that allows us to postpone some needless technicalities to Section \ref{section: simplifications}.
\begin{itemize}
\item[{\bf (S4)}]
The sets $\f x = (x_i)_{i \in V_f(\Upsilon)}$ and $\f \mu = (\mu_i)_{i \in V_e(\Upsilon)}$ are disjoint, and all indices of $\f x$ are distinct.
\end{itemize}
Thus, Simplification {\bf (S4)} is in the same spirit as Simplification {\bf (S1)}. We assume Simplification {\bf (S4)} throughout Section \ref{sec: completion of proof under all S}.

Let $\Gamma \in \fra G_F^p(\Delta)$, $\Theta \in \wt {\fra R}(\Gamma)$, and $\Upsilon \in \cal L(\Theta)$.
Choose and fix a marked vertex $i \in V_m(\Gamma)$ (see Definition \ref{definition of marked vertex}). In order to gain an extra factor $\Phi$ from $i$, we consider three cases, (a), (b), and (c).

First we explain them informally.  Case (a) is the typical situation. Since $i$ is marked,
it is a charged vertex in $\Delta$, i.e.\ after having been minimally linked to,
the number of solid and dashed edges adjacent to it are different.  In case (a) we
show that this property is inherited by at least one of the fresh summation vertices
whose parent is $i$. This fact is fairly clear since neither vertex generation
nor lumping alters the number of solid or dashed edges. This will imply
that one fresh summation vertex generated by $i$ is a chain vertex in $\Upsilon$;
this in turn will give the extra factor $\Phi$ associated with $i \in V_m(\Gamma)$.
The other two cases represent exceptional cases. Case (b) deals with a higher-order lumping (i.e.\ a lumping which has a lump of size greater than two). As indicated above, this results in a combinatorial gain expressed in terms of powers of $M^{-1}$. Such a factor is depicted graphically using a crossed wiggly line. Finally, case (c) deals with the consequence of the
factor $h_{uv}$ on the last line of \eqref{identity for resolution}, i.e.\ when a dotted edge joins two original summation vertices. We note that choosing the term $h_{uv}$ is the only way to change the number of solid or dashed edges (off-diagonal entries of $G$) in the process of vertex resolution.
Since dotted edges must be paired, we find that at least two parallel solid or
dashed edges must be replaced with a dotted edge; this corresponds to
replacing two off-diagonal factors $G_{uv}$ with $\abs{h_{uv}}^2$. After expectation,
this means trading in a factor $\Psi^2$ for $M^{-1}$. Since $\Psi \Phi \geq M^{-1/2}$, the factor $M^{-1}$ may be estimated by $\Psi^2 \Phi^2$. Out of this, $\Psi^2$ is used to compensate for the loss of $\Psi^2$ mentioned above (losing two off-diagonal entries of $G$), and the remaining $\Phi^2$ provides us with the gains of $\Phi$ associated with each of the two vertices incident to the dotted edge. See Figure \ref{figure: cases abc} for a graphical depiction of the cases (a), (b), and (c). Below we formalize these ideas.
\begin{figure}[ht!]
\begin{center}
\includegraphics{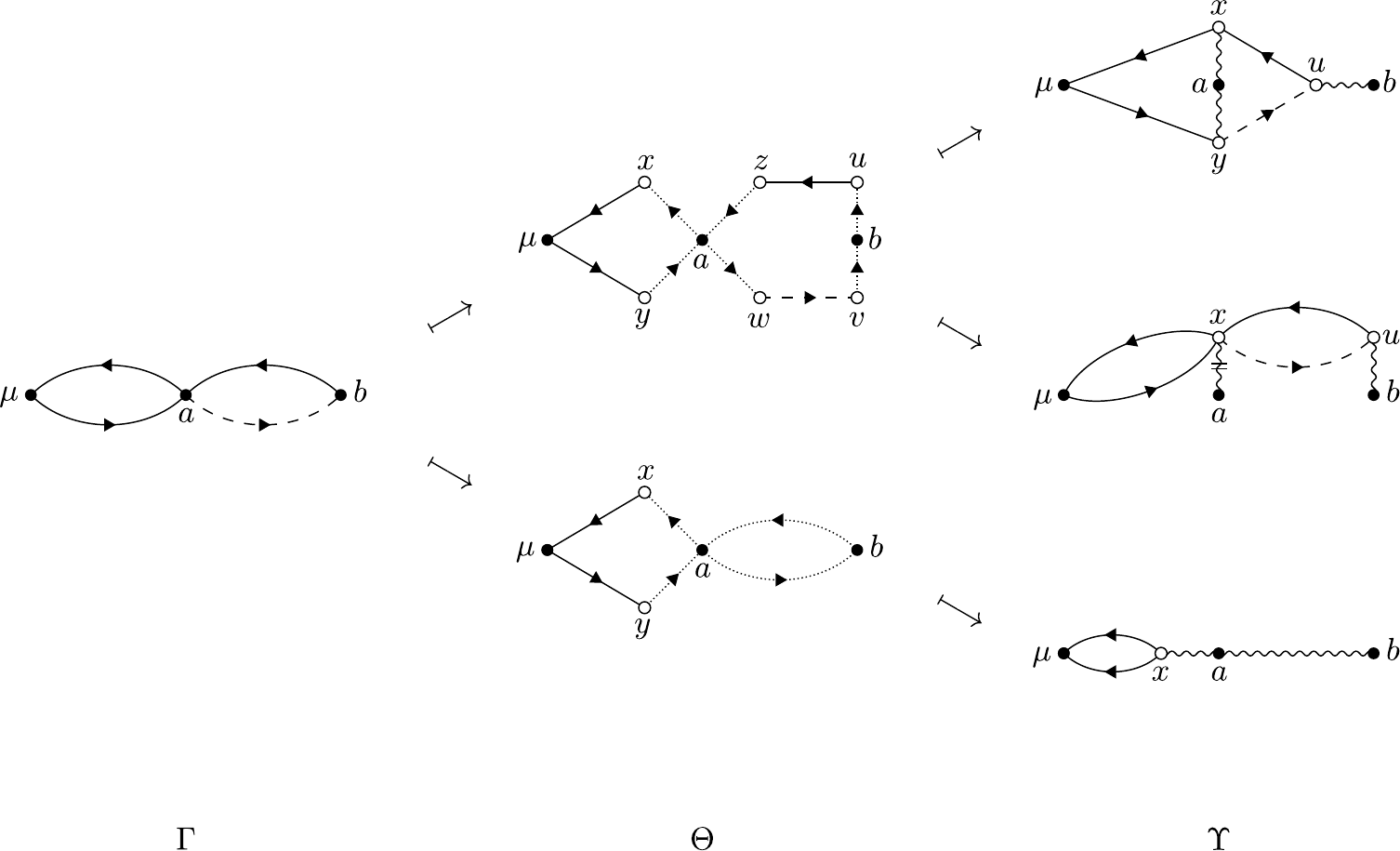}
\end{center}
\caption{A simple example of vertex resolution giving rise to the three cases (a), (b), and (c) when resolving the vertex associated with $a$. From top to bottom in the right-hand column: cases (a), (b), and (c). In the top graph, corresponding to case (a), $x$ is a chain vertex (within the chain $u\to x\to\mu$)
which is a remnant of the path $b\to a\to\mu$ consisting of solid edges in the graph $\Gamma$. In the middle graph, corresponding to case (b), the crossed wiggly line expresses that all four fresh summation indices arising from the resolution of $a$ coincide, and the four associated dotted edges of $\Theta$ were collapsed into one of $\Upsilon$.  Finally, in the bottom figure, corresponding to case (c), the solid and dashed edges between the vertices associated with $a$ and $b$ in $\Gamma$ give rise to a single wiggly line 
according to $s_{ab}=\E \abs{h_{ab}}^2$.
\label{figure: cases abc}}
\end{figure}

Recall the definitions of $\nu_i$ and $\nu_i^*$ for an arbitrary edge-coloured, directed multigraph (Definition \ref{def: incidence indices}). Informally, $\nu_i^\xi$ gives the number of legs of colour $\xi$ incident to $i$. The following definition gives a natural extension of chain weights to graphs with wiggly edges.
\begin{definition} \label{def: details of Upsilon}
A vertex $i \in V_f(\Upsilon)$ is a \emph{chain vertex} if it has degree three, such that: (i) $i$ is incident to exactly one wiggly edge, and (ii) $i$ is incident to exactly two resolvent edges, which are of the same colour.
\end{definition}
Thus a chain vertex $i \in V_f(\Upsilon)$ corresponds to a chain vertex in the sense of Definition \ref{def: incidence indices} (i.e.\ $\nu_i(\Upsilon) + \nu_i^*(\Upsilon) = 2$ and one of the terms vanishes), with the added condition that the summation over the index of $i$ is done with respect to a chain weight (encoded in $\cal A(\Upsilon)$ by the wiggly edge of $\Upsilon$ incident to $i$).

Now we present the details of the cases (a), (b), and (c).
Let $\Gamma \in \fra G_F^p(\Delta)$, $\Theta \in \wt {\fra R}(\Gamma)$, and $\Upsilon \in \cal L(\Theta)$. Let $i \in V_m(\Gamma)$ be marked.
\begin{itemize}
\item[(a)]
Suppose that $\sigma(i,j) = 0$ for all $j \in V_s(\Upsilon)$ (i.e.\ there are no wiggly edges in $\Upsilon$ that join two original summation vertices). Suppose moreover that the lumping of the fresh summation vertices of $p^{-1}(i)$ that generates $\Upsilon$ is a pairing (hence $\deg_\Gamma(i)$ must be even).

Then $\nu_j(\Upsilon) + \nu_j^*(\Upsilon)=2$ for any $j\in p^{-1}(i)$. More precisely, each $j \in p^{-1}(i)$ has degree three in $\Upsilon$; two of the incident edges are resolvent edges, and the third edge is a wiggly uncrossed edge connecting $j$ with $i$. Moreover,
\begin{equation} \label{conservation of degree}
\nu_i^\xi(\Gamma) \;=\; \sum_{j \in p^{-1}(i)} \nu_j^\xi(\Upsilon)
\end{equation}
for $\xi = 1,*$, since, under the assumption that there are no wiggly edges in $\Upsilon$ that join two original summation vertices, the total number of resolvent edges incident to $i$ does not change by vertex resolution and taking expectation. Hence we find from \eqref{condition for marked vertices} and from $\nu_j(\Upsilon) + \nu_j^*(\Upsilon) = 2$ that there must exist a $j \in p^{-1}(i)$ such that either $\nu_j(\Upsilon) = 0$ or $\nu_j^*(\Upsilon) = 0$; this implies that $j$ is a chain vertex of $\Upsilon$. We conclude:

At least one $j \in p^{-1}(i)$ is a chain vertex of $\Upsilon$.

\item[(b)]
Suppose that $\sigma(i,j) = 0$ for all $j \in V_s(\Upsilon)$, and the lumping of the fresh summation vertices of $p^{-1}(i)$ that generates $\Upsilon$ is a not pairing.

In this case one fresh summation vertex $j \in p^{-1}(i)$ was obtained by lumping together three or more fresh summation vertices of $\Theta$. We conclude:

At least one $j \in p^{-1}(i)$ is connected in $\Upsilon$ to $i$ by a crossed wiggly edge.

\item[(c)]
Suppose that $\sigma(i,j) > 0$ for some $j \in V_s(\Gamma)$. By property (v) of Section \ref{sec: vertex lumping}, $\sigma(i,j) \geq 2$. By construction of $\Upsilon$, if $\sigma(i,j) \geq 2$ then $i$ is joined to $j$ in $\Upsilon$ by a wiggly edge that is crossed by $\sigma(i,j) - 2$ strokes. We conclude:

The vertex $i$ is connected in $\Upsilon$ to some $j \in V_s(\Upsilon)$ by a wiggly edge.
\end{itemize}

Partition $V_m(\Gamma) = V_m^{(a)} \sqcup V_m^{(b)} \sqcup V_m^{(c)}$ into three subsets according to the three case (a), (b), and (c).
Now we may estimate the contribution of $\Upsilon$. For the whole estimate we freeze the original summation vertices $\f a$. We shall use Proposition \ref{lemma: weak Z lemma} on the $\abs{V_m^{(a)}}$ chain vertices of $\Upsilon$. In order to do so, we still have to get rid of the upper indices $(\f a)$ from each resolvent entry. This is done exactly as in the end of Section \ref{section: Z4}, using the identity \eqref{inverted res exp 1}. We omit further details. Thus Proposition \ref{lemma: weak Z lemma} is applicable to each chain vertex of $\Upsilon$. The remainder of the proof is a simple counting of different types of vertices.

Let
\begin{equation*}
\sigma \;\deq\; \frac{1}{2} \sum_{i,j \in V_s(\Theta)} \sigma(i,j)
\end{equation*}
denote the number of dotted edges of $\Theta$ that join two vertices in $V_s(\Theta)$. From \eqref{edge conservation Upsilon} we find that $\Upsilon$ has $\abs{E(\Gamma)} - \sigma$ resolvent edges. Moreover, $\Upsilon$ has $\abs{V_m^{(a)}}$ chain vertices. The gain from cases (b) and (c) is as follows. From the vertices of type (b) we gain $M^{-\abs{V_m^{(b)}} / 2}$. (Each such vertex is incident to a crossed wiggly line, and dropping the strokes crossing  such a lines yields a factor $M^{-1/2}$.) From the vertices of type (c) we gain $M^{-\sigma / 2}$. (Each dotted edge, encoding $h_{uv} = (s_{uv})^{1/2} \zeta_{uv}$, yields a contribution of size $(s_{uv})^{1/2} \leq C M^{-1/2}$ after taking the expectation in the $h$-variables.)

Now we sum over $\f x$ (while still keeping $\f a$ frozen). Invoking Proposition \ref{lemma: weak Z lemma} to estimate the chain vertices of $\Upsilon$, we therefore find that the contribution of $\Upsilon$ is bounded by
\begin{equation*}
\cal X \;\deq\; \Psi^{\abs{E(\Gamma)} - \sigma} \, \Phi^{\abs{V_m^{(a)}}} \, M^{-\abs{V_m^{(b)}} / 2}\, M^{-\sigma / 2}
\;\leq\; \Psi^{\abs{E(\Gamma)} - \sigma} \, \Phi^{\abs{V_m^{(a)}} + \abs{V_m^{(b)}}} \, M^{-\sigma / 2}\,,
\end{equation*}
where we used $M^{-1/2}\leq \Phi$.
Since $\sigma(i,j)\geq 2$ in case (c), it is easy to see that $\abs{V_m^{(c)}} \leq \sigma$. Thus we have $\abs{V_m^{(a)}} + \abs{V_m^{(b)}} + \sigma \geq \abs{V_m(\Gamma)}$. This yields the bound
\begin{equation*}
\cal X \;\leq\; \Psi^{\abs{E(\Gamma)}} \Phi^{\abs{V_m^{(a)}} + \abs{V_m^{(b)}} + \sigma} \, \pb{\Psi^{-1} \Phi^{-1} M^{-1/2}}^\sigma \;\leq\; \Psi^{\abs{E(\Gamma)}} \Phi^{\abs{V_m(\Gamma)}}\,,
\end{equation*}
where we also used that $\Psi \Phi \geq M^{-1/2}$. Recalling Lemma \ref{lemma: lower bound on edges of Gamma}, we find
\begin{equation} \label{main resolution estimate}
\cal X \;\leq\; \Psi^{p (\deg(\Delta) + \abs{F}) + \abs{V_c(\Gamma)} - \abs{V_m(\Gamma)}} \, \Phi^{\abs{V_m(\Gamma)}}
\;\leq\; \Psi^{p (\deg(\Delta) + \abs{F})} \, \Phi^{p \abs{V_c(\Delta)}}
\,,
\end{equation}
where we used $\abs{V_m(\Gamma)} \leq \abs{V_c(\Gamma)} = p \abs{V_c(\Delta)}$ and $\Psi\leq \Phi$. This estimate was obtained for the sum over $\f x$ with fixed $\f a$. Finally, we sum over $\f a$ trivially, using the fact the $\f a$-summation is performed with respect to a weight.
This concludes the proof of Theorem \ref{theorem: Z lemma} under Simplifications {\bf (S1)} -- {\bf (S4)}.

\section{Removing Simplifications {\bf (S1)} -- {\bf (S4)}} \label{section: simplifications}

In this section we go back to the proof of Theorem \ref{theorem: Z lemma} of Section \ref{section: Z5}, and give the additional arguments required to remove the Simplifications {\bf (S1)} -- {\bf (S4)} which were assumed there. For ease of reference, we recall them here.
\begin{itemize}
\item[{\bf (S1)}]
All summation indices in the expanded summation $\E \abs{X_F^w(\Delta)}^p$ (see \eqref{E X^p} below) are distinct. (I.e.\ we ignore repeated indices which give rise to a smaller combinatorics of the summation.)
\item[{\bf (S2)}]
There are no diagonal entries $\cal G_{aa} = G_{aa} - m$ in $\cal Z(\Delta)$. (I.e.\ $\Delta$ has no loops.)
\item[{\bf (S3)}]
We replace any diagonal term $G_{aa}^{(T)}$ with $m$ and any diagonal term $G_{aa}^{(T)*}$ with $\bar m$. (Recall that $G_{aa}^{(T)} \approx m$ in the sense that $G_{aa}^{(T)} - m \prec \Psi$ by definition of $\Lambda$.) This replacement is done in two places: in $\cal A(\Gamma)$ and in the identities \eqref{res exp 2b} and \eqref{res exp 2b iterated} which underlie the algebra of vertex resolution.
\item[{\bf (S4)}]
The families $\f x = (x_i)_{i \in V_f(\Upsilon)}$ and $\f \mu = (\mu_i)_{i \in V_e(\Upsilon)}$ are disjoint, and all indices of $\f x$ are distinct.
\end{itemize}

\subsection{Removing Simplification {\bf (S3)}} \label{section: Z6}
We start by removing Simplification {\bf (S3)}, while still assuming Simplifications {\bf (S1)}, {\bf (S2)}, and {\bf (S4)}. In order to remove Simplification {\bf (S3)}, we have deal with the error terms made in the replacement $G_{aa}^{(T)} \mapsto m$. The key formula for dealing with the diagonal terms is \eqref{res exp 2c} with the error terms $h_{aa}$, $Z_a^{(T)}$, and $U_a^{(Ta)}$ (see \eqref{def of Zi and Ji}). Recall that by \eqref{hsmallerW} we have $h_{aa} \prec M^{-1/2}$. The other error terms are estimated in the following lemma.

\begin{lemma} \label{lemma: diagonal estimates}
Suppose that $\Lambda \prec \Psi$ for some admissible control parameter $\Psi$. Fix $\ell \in \N$. Then we have
\begin{equation} \label{bound on A}
Z_a^{(T)} \;\prec\; \Psi \,, \qquad U_a^{(S)} \;\prec\; \min \{\varrho \Psi^2 ,  \Psi\} \;\leq\; \Psi \Phi\,,
\end{equation}
for $\abs{T}, \abs{S} \leq \ell$ and $a \in \{1, \dots, N\} \setminus T$.
Moreover, $Z_a^{(T)}$ is independent of $T$ and $U_a^{(S)}$ is independent of $S$.
\end{lemma}
\begin{proof}
See Appendix \ref{section: proof of res id}.
\end{proof}
Using Lemma \ref{lemma: diagonal estimates} and \eqref{res exp 2c} we write, for any fixed $K \in \N$,
\begin{equation} \label{main identity for diag entries}
\frac{1}{G_{aa}^{(\f a \setminus \{a\})}} \;=\; \frac{1}{m} + h_{aa} - Z_a^{(\f a \setminus \{a\})} - U_a^{(\f a)} \,, \qquad G_{aa}^{(\f a \setminus \{a\})} \;=\; \sum_{k = 0}^{K - 1} m^{k+1} \pb{-h_{aa} + Z_a^{(\f a \setminus \{a\})} + U_a^{(\f a)}}^k + O_\prec(\Psi^K)\,.
\end{equation}
(The error term $O_\prec(\Psi^K)$ is uniform in the same sense as the estimates of \eqref{bound on A}.)

In this section we revisit the argument from Section \ref{section: Z5}, and explain the differences resulting from the added diagonal terms.

\subsubsection{Generation of the fresh summation vertices (revisited), Step I: from $\Gamma$ to $\Pi$} \label{section: generation of fresh summation 1}
As in Section \ref{sect: generation of fresh vertices}, we start with $\Gamma \in \fra G_F^p(\Delta)$. 
The goal in Section \ref{sect: generation of fresh vertices} was to decompose $\Gamma$ into
a finite union of graphs (called $\wt {\fra R}(\Gamma)$) whereby a resolvent edge of $\Theta \in \wt {\fra R}(\Gamma)$ encoded in $\cal A(\Theta)$ an entry of $G^{(\f a)}$ or $G^{(\f a)*}$, and a dotted edge an $\f a$-admissible entry of $H$.  Thus $\cal A(\Theta)$ was well-suited for taking the partial expectation in $\f a$. In this section we keep track of the diagonal entries (represented graphically by loops) that arise both
in Family A and B identities and were previously freely replaced by powers of $m$ and $\bar m$, according to Simplification {\bf (S3)}.
The main difficulty here is that resolving diagonal entries
requires the more complicated formulas \eqref{main identity for diag entries}
instead of \eqref{identity for resolution} (which immediately yielded
resolvents with upper indices $\f a$). 

The ultimate goal of this section and of Section \ref{section: generation of fresh summation 2} is the same as that of Section \ref{sect: generation of fresh vertices}: to obtain graphs $\Theta$ whose resolvent edges encode resolvent entries with upper indices $\f a$
(up to a negligible error term that can be estimated brutally). We shall reach this in two steps.
In the first step, which is the content of this section (Section \ref{section: generation of fresh summation 1}), we express $\cal A(\Gamma)$ as a sum of monomials whose off-diagonal resolvent entries have upper indices $\f a$ and whose diagonal resolvent entries are maximally expanded.  We denote by
$\fra D(\Gamma)$ the family of graphs encoding these new monomials, and we shall use the letter $\Pi$ for a generic element of $\fra D(\Gamma)$. Graphically, therefore, this step corresponds to the mapping $\Gamma \mapsto \{\Pi_\al\} = \fra D(\Gamma)$. Sometimes we shall refer to it informally as  $\Gamma \mapsto \Pi$.

In the second step, which is the content of Section \ref{section: generation of fresh summation 2}, we use \eqref{main identity for diag entries}
to replace all maximally expanded diagonal resolvent entries (in $\cal A(\Pi)$ for any $\Pi \in \fra D(\Gamma)$) with resolvent entries having upper indices $\f a$ (again up to a negligible error term that can be estimated brutally).
We generically call the resulting graphs $\Theta$,
and let $\fra R(\Pi) =\{ \Theta_\alpha\}$ be the collection of such graphs obtained from a
fixed $\Pi \in \fra D(\Gamma)$.  Graphically, this step corresponds to the mapping $\Pi \mapsto \{ \Theta_\alpha \} = \fra R(\Pi)$.
The graphs $\Theta$ play the same role as the graphs $\Theta$ in Section \ref{section: Z5}. Indeed, each resolvent edge of $\Theta$ encodes in $\cal A(\Theta)$ a resolvent entry that has upper indices $\f a$, and each dotted edge an $\f a$-admissible entry of $H$. Hence $\cal A(\Theta)$ is amenable to taking the partial expectation in $\f a$. (This will be done in Section \ref{sec: lumping without S3}.)

\begin{lemma}\label{lm:stage1}
For any $K \in \N$ we have the decomposition
\begin{equation*}
\cal A(\Gamma) \;=\; \sum_\alpha A_\alpha + O_\prec(\Psi^K)\,,
\end{equation*}
where the summation is over a finite $N$-independent set, and $A_\al$ is a monomial in the entries of $G^{(T)}$, the entries of $G^{(T)*}$, and the $\f a$-admissible entries of $H$. Moreover, each entry $G_{uv}^{(T)}$ of $A_\alpha$ satisfies the condition
\begin{itemize}
\item[($*$)]
$G^{(T)}_{uv}$ either has upper indices $T = \f a$ or is a maximally expanded diagonal entry ($u = v$ and $T = \f a \setminus \{u\}$).
\end{itemize}
The same condition also applies to each entry $G_{uv}^{(T)*}$.
\end{lemma}

We shall apply this lemma below with the choice $K \deq p (\deg(\Delta) + 2 \abs{V_s(\Delta)})$, which will ensure that the error term $O_\prec(\Psi^K)$ is negligible.

\begin{proof}[Proof of Lemma \ref{lm:stage1}]
We apply \eqref{identity for resolution} to each off-diagonal (maximally expanded) resolvent entry of $\cal A(\Gamma)$.
After an application of \eqref{identity for resolution}, the resulting expression does not in general satisfy ($*$), due to the factor $G^{(\f a \setminus \{u,v\})}_{uu}$ on the last line of \eqref{identity for resolution}, which is not maximally expanded. (Note that all other factors on the right-hand side of \eqref{identity for resolution} satisfy ($*$).) As always, we use \eqref{resolvent expansion type 1} to make $G^{(\f a \setminus \{u,v\})}_{uu}$ maximally expanded:
\begin{equation} \label{iteration for diagonal terms}
G^{(\f a \setminus \{u,v\})}_{uu} \;=\; G^{(\f a \setminus \{u\})}_{uu} + \frac{G^{(\f a \setminus \{u,v\})}_{uv} G^{(\f a \setminus \{u,v\})}_{vu}}{G^{(\f a \setminus \{u,v\})}_{vv}}\,,
\qquad
\frac{1}{G^{(\f a \setminus \{u,v\})}_{uu}} \;=\; \frac{1}{G^{(\f a \setminus \{u\})}_{uu}} - \frac{G^{(\f a \setminus \{u,v\})}_{uv} G^{(\f a \setminus \{u,v\})}_{vu}}{G^{(\f a \setminus \{u,v\})}_{uu} G^{(\f a \setminus \{u\})}_{uu} G^{(\f a \setminus \{u,v\})}_{vv}}\,.
\end{equation}
Here we use the first identity of \eqref{iteration for diagonal terms}. The first term is maximally expanded and good as it is. The second consists of two maximally expanded off-diagonal terms in the numerator and one diagonal term in the denominator which is not maximally expanded. We now apply \eqref{identity for resolution} to each of the terms in the numerator. The result is an expression with entries of $G$ that either have upper indices $\f a$ or are diagonal. The diagonal entries are not maximally expanded, and hence we must apply \eqref{resolvent expansion type 1} to each of them. Moreover, the diagonal entry in the denominator is not maximally expanded, and must be further expanded using the second identity of \eqref{iteration for diagonal terms}. We continue in this manner, successively using \eqref{identity for resolution} on maximally expanded off-diagonal entries and \eqref{resolvent expansion type 1} on diagonal entries that are not maximally expanded. This procedure is reminiscent of the one introduced after Definition \ref{definition: maximally expanded}. As in Section \ref{section: Z3}, although this procedure in general does not terminate, it does increase the number of off-diagonal terms, which allows us to stop brutally once a sufficient number of off-diagonal terms have been generated. Note that, unlike the one-step iteration of Section \ref{section: Z3} (which only used \eqref{resolvent expansion type 1}), we now have a two-step iteration, which repeatedly uses \eqref{resolvent expansion type 1} and \eqref{identity for resolution} in tandem.

More formally, the algorithm may be described as follows. In order to define the brutal stopping rule precisely, we set
\begin{equation*}
\ell(A) \;\deq\; (\text{number of entries of $G^{(\f a)}$ in $A$}) + \sum_{a,b \in V_s(\Gamma)} (\text{number of entries $h_{ab}$ in $A$})\,,
\end{equation*}
where $A$ is a monomial in the entries of $G^{(T)}$ and $H$. Now set $A \deq \cal A(\Gamma)$; $A$ will denote the running monomial in the algorithm, and $\cal A(\Gamma)$ is its initial value.
\begin{itemize}
\item[Step 1.]
Pick an off-diagonal term $G^{(\f a \setminus \{u,v\})}_{uv}$ in $A$ which does not have upper indices $\f a$. If no such term exists, go to Step 2. Otherwise apply \eqref{identity for resolution} to $G^{(\f a \setminus \{u,v\})}_{uv}$. This yields the splitting $A = A' + A''$ (where $A'$
is the main term that contains a factor $G^{(\f a)}$ and $A'' = 0$ unless both $u$ and $v$ are summation vertices, in which case $A''$
contains the special factor $h_{uv}$ from
 the third line of \eqref{identity for resolution}). Repeat step 1 for $A'$ and $A''$ (provided $A'' \neq 0$). (Notice that at each repetition of  Step 1
the number of off-diagonal terms with no upper index $\f a$
decreases by one, so after finitely many steps the algorithm exits to Step 2.)
\item[Step 2.]
If $ \ell(A) \geq K$,
stop. Otherwise go to Step 3.
\item[Step 3.]
Pick a diagonal term $G^{(\f a \setminus \{u,v\})}_{uu}$ in $A$ that is not maximally expanded. If no such term exists, stop. Otherwise apply \eqref{iteration for diagonal terms} to $G^{(\f a \setminus \{u,v\})}_{uu}$. This induces a splitting $A = A' + A''$ according to the two summands in either identity of \eqref{iteration for diagonal terms}. Repeat Step 1 for both $A'$ and $A''$.
\end{itemize}
Since Step 1 increases $\ell$ by exactly one, it follows by Step 2 that the algorithm must terminate after a finite, $N$-independent, number of steps. The result is a finite sum of terms $\{A_\alpha\}$ whose number does not depend on $N$. Pick one such $A_\alpha \equiv A$. We consider two cases depending on whether the algorithm, in generating $A$, stopped at Step 2 or Step 3. In other words, we differentiate based on whether it is stopped because there are sufficiently many small factors (stopping at Step 2) or because each resolvent entry satisfies ($*$) (stopping at Step 3).

Consider first the case where the algorithm stopped at Step 2. 
This corresponds to a brutal stopping, where we may estimate $A$ by a simple power counting using the lower bound on $\ell(A)$. We claim that we have the trivial bound
\begin{equation} \label{brutal estimate of A}
A \;=\; O_\prec(\Psi^{K}).
\end{equation}
In order to see this, we note that 
\eqref{identity for resolution} (reading these formulas from right to left) and \eqref{h prec Psi} imply\footnote{Note that these estimates also follow directly from basic large deviation results such as Lemmas B.1 and B.2 in \cite{EYY1}.}
\begin{equation} \label{estimates for brutal cutoff}
\sum_x^{(\f a)} h_{u x} G_{x v}^{(\f a)} \;=\; O_\prec(\Psi), \qquad
\sum_x^{(\f a)} G_{u x}^{(\f a)} h_{xv} \;=\; O_\prec(\Psi)\,, \qquad
\sum_{x,y}^{(\f a)} h_{ux} G_{xy}^{(\f a)} h_{xv} \;=\; O_\prec(\Psi)\,.
\end{equation}
Moreover, by definition of Step 1 and the explicit expressions in \eqref{identity for resolution} each entry of $G^{(\f a)}$ in $A$ comes in one of the three forms in \eqref{estimates for brutal cutoff}. Hence the lower bound $ \ell(A) \geq K$,
 the definition of $\ell$, and \eqref{h prec Psi} yield \eqref{brutal estimate of A}.

Apart from this error term, all other terms resulted in stopping the
algorithm at Step 3.
In this case it is immediate that each entry $G^{(T)}_{uv}$ of $A$ satisfies ($*$).
This proves Lemma~\ref{lm:stage1}.
\end{proof}

The algorithm from the proof of Lemma \ref{lm:stage1} (Steps 1 -- 3) has a trivial reformulation on the level of graphs.
This also yields a convenient graphical representation of the monomials $\{A_\alpha\}$. For future use, we give more details on the graphical version of Step 3. Let $\fra D(\Gamma)$ denote the set of graphs that encode the monomials  $\{A_\alpha\}$ obtained through Steps 1--3 starting from $\Gamma$
and stopping at Step 3. The elements of $\fra D(\Gamma)$ will generically be denoted by $\Pi$.

We use the graphical notations from Figures \ref{figure: diagonal elements} and \ref{figure: h and s}. In Figure \ref{figure: vertex resolution for edge} we summarize the rules \eqref{identity for resolution} graphically.
\begin{figure}[ht!]
\begin{center}
\includegraphics{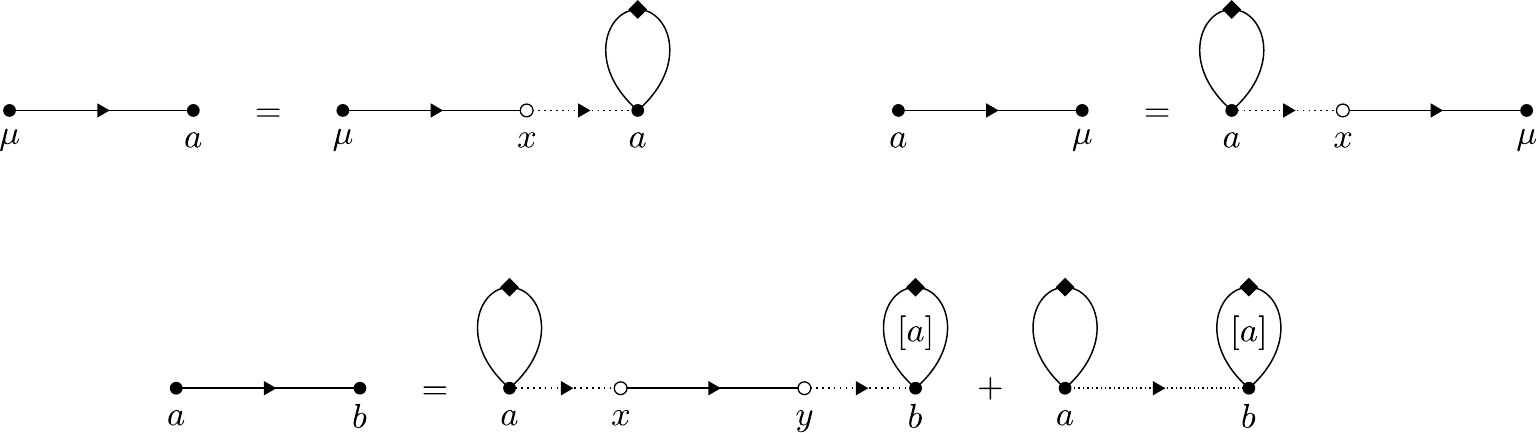}
\end{center}
\caption{The graphical representation of the rules \eqref{identity for resolution} (using $a$ and $b$ for summation indices and $\mu$ for external index
instead of the generic indices $u$ and $v$). Two of the loops on the last line correspond to resolvent entries $G_{bb}^{(\f a \setminus \{a,b\})}$ which are not maximally expanded: they still depend on $a$, which is indicated by the label $[a]$ inside the loop.
\label{figure: vertex resolution for edge}}
\end{figure}

The term $A''$ from Step 3 arises from taking the second term on the right-hand sides of \eqref{iteration for diagonal terms}, which in the graphical language translates to creating two (non-loop) resolvent edges connecting the vertices $i$ and $j$ associated with the indices $u$ and $v$ respectively, i.e.\ linking a loop at $i$ with $j$. We call this process \emph{linking $i$ with $j$}. See Figure \ref{figure: vertex linking}. 
\begin{figure}[ht!]
\begin{center}
\includegraphics{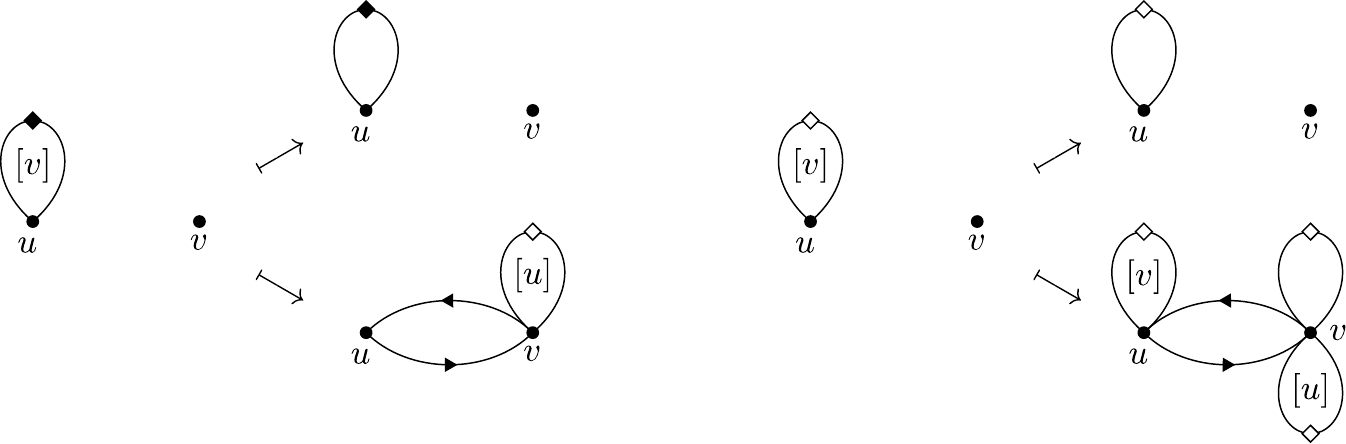}
\end{center}
\caption{Linking the vertex $i$ (depicted in the picture with its associated index $u$) with the vertex $j$ (depicted with index $v$). These two diagrams correspond to the two identities in \eqref{iteration for diagonal terms}. As in Figure \ref{figure: vertex resolution for edge}, if a diagonal resolvent entry encoded by a loop is not maximally expanded, we indicate the index on which it depends in angular brackets inside the loop.
\label{figure: vertex linking}}
\end{figure}

This graphical algorithm provides an alternative, graphical, construction of $\fra D(\Gamma)$ starting from $\Gamma$. Each $\Pi \in \fra D(\Gamma)$ encodes a monomial $\cal A(\Pi)$ whose resolvent entries satisfy ($*$). As in Section \ref{section: Z5}, the vertex set of $\Pi$ may be written as $V(\Pi) = V_f(\Pi) \sqcup V_s(\Pi) \sqcup V_e(\Pi)$, corresponding to the fresh summation vertices, the original summation vertices, and the external vertices, respectively. Note that $\Pi$ now contains loops, which bear either a black or white diamond (encoding diagonal entries of $G$ in the numerator or denominator respectively). Moreover, each diagonal entry encoded by a loop of $\Pi$ is maximally expanded, and each off-diagonal entry encoded by a non-loop edge of $\Pi$ has upper indices $\f a$. See Figure \ref{figure: Pi} for a simple example of the process $\Gamma \mapsto \Pi \in \fra D(\Gamma)$.
\begin{figure}[ht!]
\begin{center}
\includegraphics{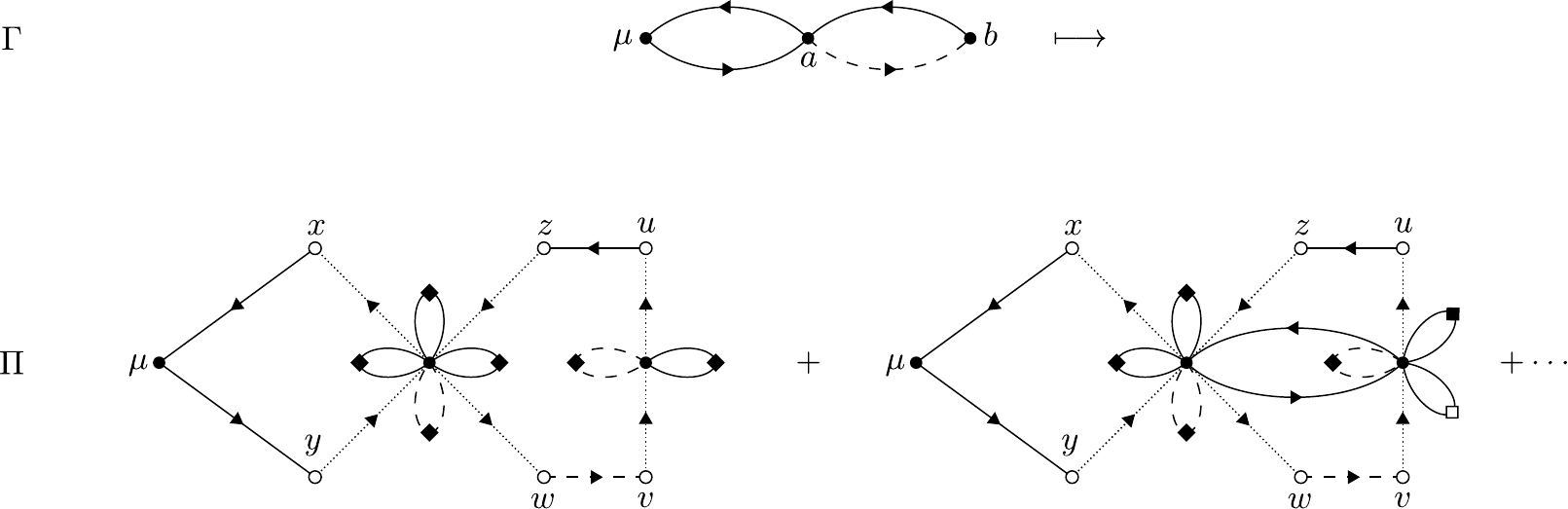}
\end{center}
\caption{The process $\Gamma \mapsto \Pi \in \fra D(\Gamma)$, where we draw the two simplest elements of $\fra D(\Gamma)$ on the bottom line. (For reasons of space, we omit the labels $a$ and $b$ on the bottom line.) The first graph is just the top-middle graph of
 Figure~\ref{figure: cases abc} but with added loops. In the second graph one
loop at $a$ was linked with $b$.
\label{figure: Pi}}
\end{figure}

\subsubsection{Generation of the fresh summation vertices (revisited), Step II: from $\Pi$ to $\Theta$} \label{section: generation of fresh summation 2}
In this section we complete the second part of the generation of the fresh summation vertices, by constructing a family of graphs $\Theta \in \fra R(\Pi)$ from $\Pi$. The underlying algebraic identity is \eqref{main identity for diag entries}.

\begin{lemma}\label{lm:stage2}
For any $K \in \N$ and any $\Pi \in \fra D(\Gamma)$
there is a decomposition
\begin{equation} \label{statement of lmstage2}
\cal A(\Pi) \;=\; \sum_\alpha B_\alpha + O_\prec(\Psi^K)\,,
\end{equation}
such that each $B_\al$ is a monomial in the entries of $G^{(\f a)}$ and the $\f a$-admissible entries of $H$. The sum over $\alpha$ ranges over a finite set that is independent of $N$.
\end{lemma}

We shall apply Lemma \ref{lm:stage2} with the choice $K \deq p (\deg(\Delta) + 2 \abs{V_s(\Delta)})$, which will ensure that the error term $O_\prec(\Psi^K)$ is negligible.

\begin{proof}[Proof of Lemma \ref{lm:stage2}]
We simply apply \eqref{main identity for diag entries} to each diagonal resolvent entry of $\cal A(\Pi)$. Recall that each diagonal resolvent entry of $\cal A(\Pi)$ is maximally expanded, which implies that all resolvent entries explicitly appearing in the definition \eqref{def of Zi and Ji} for $Z_a^{(\f a \setminus \{a\})}$ and $U_a^{(\f a)}$ have upper indices $\f a$.
Then, as above, it immediately follows that if we pick the rest term $O_\prec(\Psi^K)$ in \eqref{main identity for diag entries} from any diagonal entry, the resulting monomial is $O_\prec(K)$ and may be absorbed into the error term on the right-hand side of \eqref{statement of lmstage2}.

For the following we therefore assume that there are no rest terms $O_\prec(\Psi^K)$ in the expansion \eqref{main identity for diag entries} of the diagonal resolvent entries of $\cal A(\Pi)$. The result is a finite family of monomials whose number does not depend on $N$ (but does of course depend on $K$), and which may again be represented graphically. In such graphs we represent a term $U_a^{(\f a)}$ with a solid ring around the vertex associated with $a$
(these terms $U_a^{(\f a)}$ will not be expanded further, so their precise structure
does not matter; the number of rings simply encode their size). See Figure \ref{figure: three error pieces} for a depiction of the three nontrivial terms arising from the expansion of $G_{aa}^{(\f a \setminus \{a\})}$.
\begin{figure}[ht!]
\begin{center}
\includegraphics{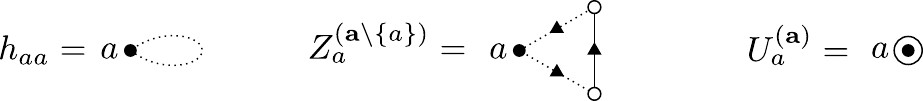}
\end{center}
\caption{The graphical representation of the three error terms resulting from the expansion of $G_{aa}^{(\f a \setminus \{a\})}$ using \eqref{main identity for diag entries}. Apart from the entries of $H$ encoded by the dotted edges, all terms are independent of $\f a$.}
\label{figure: three error pieces}
\end{figure}
Thus, when expanding a loop at $a$ with a white diamond (encoding $1 / G_{aa}^{(\f a \setminus \{a\})}$), we replace the loop with either nothing  (corresponding to the term $1/m$ used in the argument of Section~\ref{section: Z5})
or one of the three pieces in Figure \ref{figure: three error pieces}.
Similarly, when expanding a loop at $a$ with a black diamond (encoding $G_{aa}^{(\f a \setminus \{a\})}$), we replace the loop with either nothing (corresponding to a factor $m$ coming from the zeroth order term in the summation in \eqref{main identity for diag entries}) or an agglomeration of pieces from Figure \ref{figure: three error pieces} at the vertex associated with $a$. (We use concentric rings around $a$ to depict several factors $U_a^{(\f a)}$). See Figure \ref{figure: expanding diagonal}.
\begin{figure}[ht!]
\begin{center}
\includegraphics{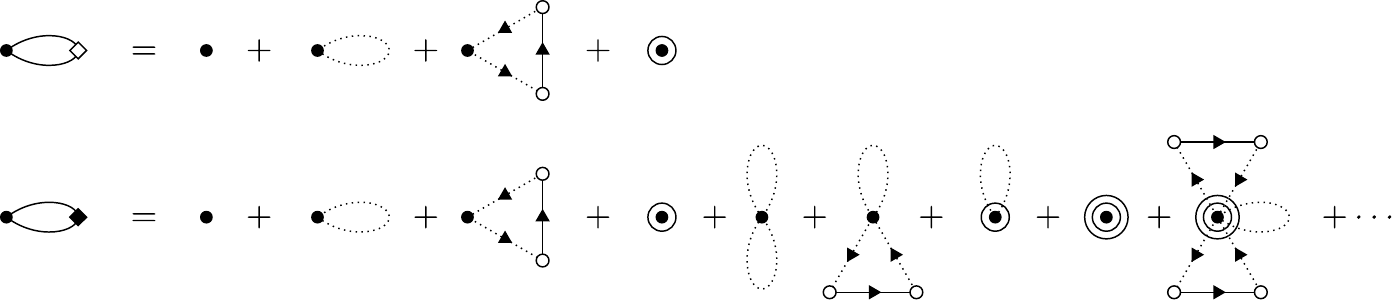}
\end{center}
\caption{Expanding a loop. The two lines correspond to the two identities of \eqref{main identity for diag entries} respectively.}
\label{figure: expanding diagonal}
\end{figure}

The application of \eqref{main identity for diag entries} to the diagonal entries encoded by $\Pi$ yields a new family of graphs, which we call $\fra R(\Pi)$ and whose elements we denote by $\Theta$. Each resolvent edge of $\Theta \in \fra R(\Pi)$ now encodes an entry of $G^{(\f a)}$ or $G^{(\f a) *}$. We also have the usual self-explanatory splitting $V(\Theta) = V_f(\Theta) \sqcup V_s(\Theta) \sqcup V_e(\Theta)$. See Figure \ref{figure: example Theta with diag} for an example of the process $\Pi \mapsto \fra R(\Pi)$.
\begin{figure}[ht!]
\begin{center}
\includegraphics{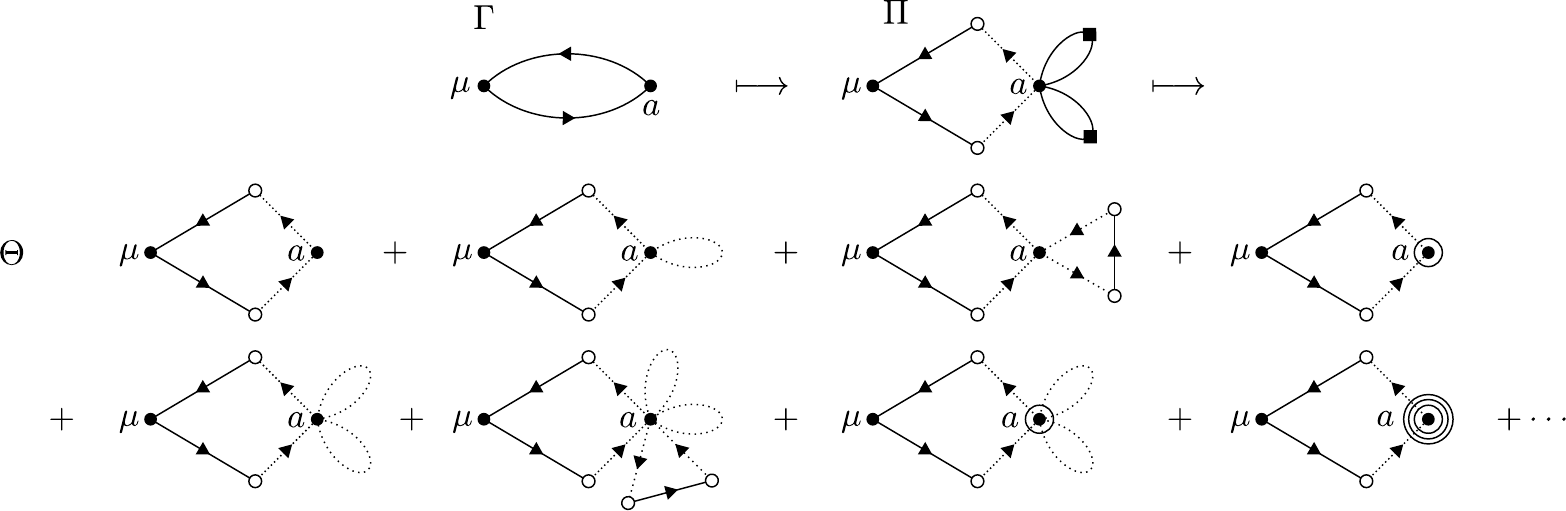}
\end{center}
\caption{The process $\Gamma \mapsto \Pi \mapsto \Theta$. On the second line we draw four graphs $\Theta$ corresponding to choosing one of the four terms on the right-hand side of the first equation of \eqref{main identity for diag entries}. On the third line we draw some more complicated graphs $\Theta$.  \label{figure: example Theta with diag}}
\end{figure}
\end{proof}
 
Summarizing the results of this Sections \ref{section: generation of fresh summation 1} and \ref{section: generation of fresh summation 2}, for a given $\Gamma \in \fra G_F^p(\Delta)$, we have constructed an $N$-independent set of graphs,
\begin{equation*}
\fra R(\fra D(\Gamma)) \;=\; \bigcup \hb{\fra R(\Pi) \col \Pi \in \fra D(\Gamma)}\,.
\end{equation*}
If $\Theta \in \fra R (\fra D(\Gamma))$ then each resolvent entry of $\cal A(\Theta)$ has upper indices $\f a$, and the fresh summation indices $\f x = (x_i)_{i \in V_f(\Theta)}$ and the original summation indices $\f a = (a_i)_{i \in V_s(\Theta)}$ are disjoint.
Moreover, we have the splitting
\begin{equation*}
\cal A_{\f a}(\Gamma) \;=\; \sum_{\Theta \in \fra R (\fra D(\Gamma))} \sum_{\f x}^{(\f a)} \cal A_{\f a, \f x}(\Theta) + O_\prec\pb{\Psi^{p (\deg(\Delta) + 2 \abs{V_s(\Delta)})}}
\end{equation*}
where we explicitly indicated the set of summation indices in the subscript of $\cal A$, see \eqref{Max}.
Note that the elements of the family  $\fra R(\fra D(\Gamma))$ have the same properties as the elements of the smaller set $\wt {\fra R}(\Gamma)$ from Section \ref{section: Z5}.

\subsubsection{Lumping of the fresh summation vertices (revisited) and conclusion of the estimate} \label{sec: lumping without S3}
Fix a $\Theta \in \fra R(\fra D(\Gamma))$. Now we may proceed as in Section \ref{sec: vertex lumping} and take the lumping of the entries of $H$ in $\cal A(\Theta)$ by computing their partial expectation $\prod_{a \in \f a} P_a$. Since all resolvent entries of $\cal A(\Theta)$ are independent of $\f a$, this partial expectation acts only on the entries of $H$, and leads to lumpings exactly as in Section \ref{sec: vertex lumping}. This gives rise to a family of graphs $\Upsilon \in \fra L(\Theta)$. As before, we seek to gain a factor $\Phi$ from each marked vertex $i \in V_m(\Gamma)$.

Thus, let us fix a sequence $\Gamma \mapsto \Pi \mapsto \Theta \mapsto \Upsilon$. It is convenient to extend the definition of the degree of a vertex as follows. By definition, the \emph{degree} of $i \in V(\Theta)$, written $\deg_\Theta(i)$, is equal to the number of legs incident to $i$ plus two times the number of rings around $i$. This convention is chosen so that each error term in Figure \ref{figure: three error pieces} increases the degree of $i$ by two.

Now take a marked vertex $i \in V_m(\Gamma)$. Note that, by construction of $\Pi$ and $\Theta$, we have $\deg_\Theta(i) \geq \deg_\Gamma(i)$. We consider two cases.
\begin{enumerate}
\item
Suppose that $\deg_\Theta(i) = \deg_\Gamma(i)$. This means that in the process $\Gamma \mapsto \Pi$ the original summation vertex $i$ was not linked with another original summation vertex (see Section \ref{section: generation of fresh summation 1}), and that in the process $\Pi \mapsto \Theta$ (see Section \ref{section: generation of fresh summation 2}) we always chose the main term ($1/m$ or $m$) on the right-hand sides of \eqref{main identity for diag entries} when applying \eqref{main identity for diag entries} to any diagonal entries with lower indices $a_i a_i$.
In particular, \eqref{conservation of degree} holds. We may therefore proceed exactly as in Section \ref{section: Z5}: any pairing in $\Theta \mapsto \Upsilon$ of the white vertices adjacent to $i$ gives rise to at least one chain vertex of $\Upsilon$ (see Definition \ref{def: details of Upsilon}). A higher-order lumping (i.e.\ one that is not a pairing) gives rise to a positive power of $M^{-1/2} \leq \Psi$. Either way, we shall gain a factor $\Phi$ from $i$ after summing over $\f x$ and invoking Proposition \ref{lemma: weak Z lemma}.
\item
Suppose that $\deg_\Theta(i) > \deg_\Gamma(i)$. In this case we have that either
\begin{enumerate}
\item[(ii.1)]
in the process $\Gamma \mapsto \Pi$ the vertex $i$ was linked to another original summation vertex, or
\item[(ii.2)]
in the process $\Pi \mapsto \Theta$ we chose at least one error term (represented graphically by one of the graphs in Figure \ref{figure: three error pieces}) on the right-hand sides of \eqref{main identity for diag entries} when applying \eqref{main identity for diag entries} to the diagonal entries with lower indices $i$.
\end{enumerate}
We claim that either case, (ii.1) or (ii.2), results in an extra error factor in $\Upsilon$ of order $\Psi$.

In order to see this, consider first the case (ii.1). Here the linking means that there is a $j \in V_s(\Upsilon)$ such that in $\Upsilon$ we have two extra resolvent edges (as compared to case (i)), each connecting a vertex in $p^{-1}(i)$ to a vertex in $p^{-1}(j)$. This yields a factor $\Psi^2$. Thus we gain a 
factor $\Psi$ that we ascribe to $i$ (the other factor $\Psi$ is in general not available, as it may be needed for exactly the same reason at the vertex $j$). Next, consider the case (ii.2). If there is a term $h_{aa}$ or $U_a^{(\f a)}$ in $\cal A(\Theta)$, then we immediately get a factor $\Psi \Phi$. 
(For $U_a^{(\f a)}$ this is trivial by \eqref{bound on A} and for $h_{aa}$ taking $P_a$ implies that there must be at least another factor $h_{aa}$, in which case we get a factor $M^{-1} \leq \Psi \Phi$.) Finally, if we have a factor $Z_a^{(\f a \setminus \{a\})}$ observe that in the expression
\begin{equation} \label{explicit Za}
Z_a^{(\f a \setminus \{a\})} \;=\; Q_a \pBB{\sum_{x,y}^{(\f a)} h_{a x} G_{xy}^{(\f a)} h_{y a}}
\end{equation}
we cannot pair $h_{ax}$ with $h_{y a}$ when computing the partial expectation $P_a$ (since $P_a Q_a = 0$). (Of course in a higher-order lumping, they could be in the same lump provided this lump contains at least three elements.) This implies that, in the leading-order pairing, the fresh summation vertices of $\Theta$ associated with $x$ and $y$ will be paired into different vertices of $\Upsilon$. In particular, we gain an additional off-diagonal resolvent entry $G_{xy}^{(\f a)} \prec \Psi$. See Figure \ref{figure: resolution of Z} for a graphical depiction of this lumping.
\end{enumerate}
\begin{figure}[ht!]
\begin{center}
\includegraphics{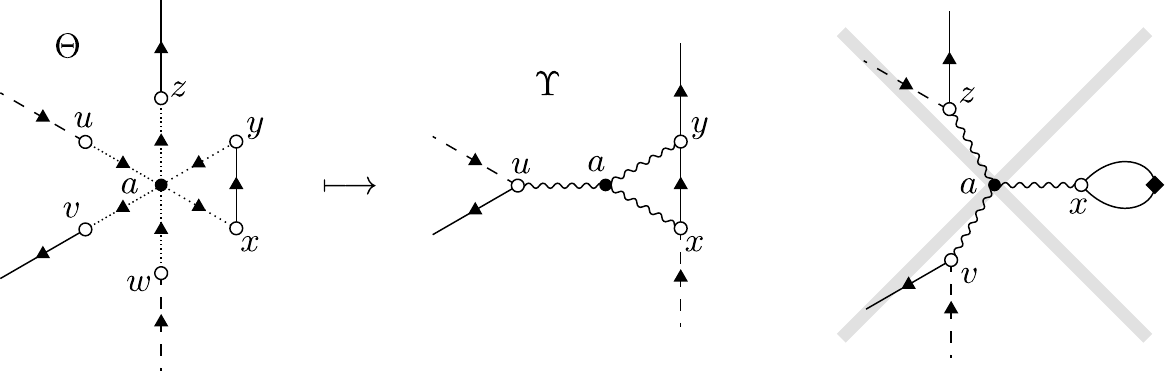}
\end{center}
\caption{The lumping of fresh summation vertices in the presence of a factor $Z_a^{(\f a \setminus \{a\})}$ (represented by a triangle in $\Theta$.). Due to the $Q_a$ in the definition of $Z_a^{(\f a \setminus \{a\})}$, the last graph (crossed out in grey) does not contribute. \label{figure: resolution of Z}}
\end{figure}
In summary, each marked vertex $i$ therefore yields a gain of $\Phi$ upon summation over $\f x$. This concludes the proof of Theorem \ref{theorem: Z lemma} without Simplification {\bf (S3)}.

\subsection{Removing Simplification {\bf (S2)}} \label{section: Z7}
In this section we revisit the arguments of Sections \ref{section: Z5} and \ref{section: Z6}, and explain the modifications required if we relax Simplification {\bf (S2)}, i.e.\ allow diagonal entries $\cal G_{aa} = G_{aa} - m$ in the definition of $\cal Z$. (On the level of $\Delta$, this amounts to allowing loops.) The construction of $\Gamma \in \fra G_F^p(\Delta)$ remains unchanged. Each diagonal entry of $\cal A(\Gamma)$ is maximally expanded, i.e.\ of the form $\cal G_{aa}^{(\f a \setminus \{a\})}$. Hence the construction of $\Pi \in \fra D(\Gamma)$ from $\Gamma$ carries over unchanged from Section \ref{section: generation of fresh summation 1}. Now $\Pi$ has loops of three kinds: with a black diamond (encoding $G_{aa}^{(\f a \setminus \{a\})}$), with a white diamond (encoding $1/G_{aa}^{(\f a \setminus \{a\})}$), and plain (encoding $\cal G_{aa}^{(\f a \setminus \{a\})}$). Note that a decorated loop encodes a factor of size $O_\prec(1)$ while a plain loop encodes a factor of size $O_\prec(\Psi)$. The additional difficulty in this section as compared to Section \ref{section: Z5} is that the naive size of a plain loop is smaller than the size of the decorated loops dealt with in Section \ref{section: Z5}. Thus we have to establish bounds which, in addition to the gain extracted in Section \ref{section: Z5}, also contain the smallness associated with the naive size of a plain loop.

The process $\Pi \mapsto \Theta \in \fra R(\Pi)$, in which all (maximally expanded) diagonal entries are expanded using \eqref{main identity for diag entries} is again the same as that of Section \ref{section: generation of fresh summation 2}. For $1/G_{aa}^{(\f a \setminus \{a\})}$ and $G_{aa}^{(\f a \setminus \{a\})}$ we use \eqref{main identity for diag entries}, and, in addition, for $\cal G_{aa}^{(\f a \setminus \{a\})}$ we use
\begin{equation} \label{identity for diag cal G}
\cal G_{aa}^{(\f a \setminus \{a\})} \;=\; \sum_{k = 1}^{K - 1} m^{k+1} \pb{-h_{aa} + Z_a^{(\f a \setminus \{a\})} + U_a^{(\f a)}}^k + O_\prec(\Psi^K)
\end{equation}
(note that the sum starts with $k = 1$). Finally, lumping the white summation vertices yields the graph $\Upsilon$. As in Section \ref{sec: lumping without S3}, the important observation is that the two white vertices associated with a $Z_a^{(\f a \setminus \{a\})}$ (see Figure \ref{figure: three error pieces}) cannot be paired. In summary, the resolution process $\Gamma \mapsto \Pi \mapsto \Theta \mapsto \Upsilon$ is almost identical to that in Sections \ref{section: Z5} and \ref{section: Z6}. There is only one new ingredient: the expansion \eqref{identity for diag cal G} which starts with $k = 1$.

To illustrate this procedure, let us consider the simple example with $\f a = \{a\}$
\begin{align}
\E G_{\mu a}^* G_{a \mu} \cal G_{aa} &\;=\; \E \sum_{x,y}^{(a)} G_{aa} G_{aa}^* h_{a x} G_{x \mu}^{(a)} G_{\mu y}^{(a)*} h_{y a} \cal G_{aa}
\notag \\
&\;=\; m^2 \bar m \, \E \sum_{x,y}^{(a)} h_{a x} G_{x \mu}^{(a)} G_{\mu y}^{(a)*} h_{y a} \pbb{Q_a \sum_{z,w}^{(a)} h_{az} G_{zw}^{(a)} h_{wa} - h_{aa} + U_a^{(a)}} + \cdots
\notag \\ \label{simple resolution for cal G}
&\;=\; m^2 \bar m \, \E \sum_{x,y}^{(a)} s_{ax} s_{ay} G_{\mu y}^{(a)*} G_{yx}^{(a)} G_{x \mu}^{(a)} + 0 + m^2 \bar m \, \E \sum_{x} s_{ax} G_{\mu x}^{(a)*} G_{x \mu}^{(a)} U_{a}^{(a)} + \cdots\,,
\end{align}
where $+ \cdots$ denotes higher-order terms in the expansion \eqref{main identity for diag entries} and \eqref{identity for diag cal G}. The expectation of the middle term in the parentheses vanishes because $\E h_{aa} = 0$. See Figure \ref{figure: cal G} for a graphical version of \eqref{simple resolution for cal G}.
\begin{figure}[ht!]
\begin{center}
\includegraphics{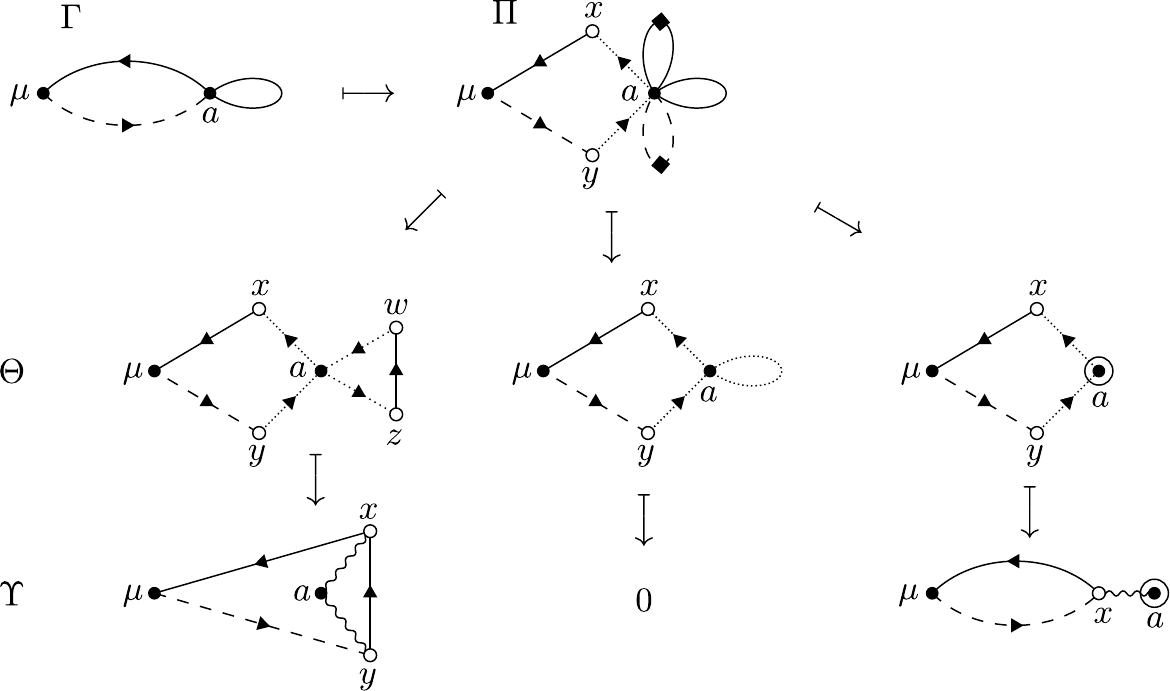}
\end{center}
\caption{The complete resolution process $\Gamma \mapsto \Pi \mapsto \Theta \mapsto \Upsilon$ for the example \eqref{simple resolution for cal G}. At each step we only draw the leading-order graphs. The first, second, and third lines of the figure correspond to the first, second, and third lines of \eqref{simple resolution for cal G} respectively. \label{figure: cal G}}
\end{figure}

Note that each unmarked loop (encoding a diagonal entry of $\cal G$) contributes a factor $O_\prec(\Psi)$ to $\cal A(\Gamma)$. When performing the vertex resolution $\Gamma \mapsto \Upsilon$, we therefore have to ensure that this gain of $\Psi$ is not lost (i.e.\ that $\cal A(\Upsilon)$ has an associated factor of size $O_\prec(\Psi)$). In addition, we have to gain a factor $\Phi$ from each marked vertex of $\Upsilon$.

In the example \eqref{simple resolution for cal G}, the vertex $a$ is marked and each term on the bottom line of \eqref{simple resolution for cal G} is of order $\Psi^3 \Phi$. This bound should be read as $\Psi^2 \Psi \Phi$, where $\Psi^2$ is the trivial bound on the off-diagonal entries, $\Psi$ is the bound on the diagonal entry of $\cal G$, and $\Phi$ is the additional gain arising from the fact that $a$ is marked. Indeed, the first term on the bottom line of \eqref{simple resolution for cal G} is of order $\Psi^3 \Phi$ by Proposition \ref{lemma: weak Z lemma}, and the last term of order $\Psi^4$ by Lemma \ref{lemma: diagonal estimates}.

This is in fact a general phenomenon. Let $i \in V_m(\Gamma)$ be marked, with associated summation index $a$. We shall give the details only for a leading-order graph $\Upsilon$, i.e.\  a graph $\Upsilon$ that satisfies:
\begin{enumerate}
\item
In the process $\Gamma \mapsto \Pi$ the original summation vertex $i$ was not linked with another original summation vertex.
\item
In the process $\Pi \mapsto \Theta$ we always chose the main term ($1/m$ or $m$) on the right-hand sides \eqref{main identity for diag entries}, and the term $m Z_a^{(\f a \setminus \{a\})}$ on the right-hand side of \eqref{identity for diag cal G}.
\item
In the process $\Theta \mapsto \Upsilon$, we chose a pairing of the white vertices incident to $i$. (I.e.\ no higher-order lumping is allowed.)
\end{enumerate}
If $\Upsilon$ does not satisfy (i) -- (iii), an argument almost identical to that of Sections \ref{sec: completion of proof under all S} and \ref{sec: lumping without S3} yields an extra factor $\Phi$, in addition to $\Psi^\ell$ where $\ell$ is the number of plain loops incident to $i$ in $\Gamma$. (This is a simple power counting that uses the fact that each $U_a^{(\ba)}$ yields a factor $\Psi \Phi$, and each $h_{aa}$ and $Z_a^{(\f a \setminus \{a\})}$ yield a factor $\Psi$ each. That $Z_a^{(\f a \setminus \{a\})}$ yields a factor $\Psi$ follows from the observation that after resolution it yields an off-diagonal resolvent entry in $\cal A(\Upsilon)$, as explained after \eqref{explicit Za}. Note that, unlike in Section \ref{sec: lumping without S3} where it was enough to gain a factor $\Psi$ from $U_a^{(\f a)}$, here it is crucial that $U_a^{(\f a)} \prec \Psi \Phi$.)

Let us therefore assume that $\Upsilon$ satisfies (i) -- (iii). By (i) and (ii) we have \eqref{conservation of degree}. Recall the definition of the projection $p$ from (iv) in Section \ref{sec: vertex lumping}. Each $j \in p^{-1}(i)$ is incident to precisely two resolvent edges and one wiggly edge that is also incident to $i$. Moreover, no vertex of $p^{-1}(i)$ is incident to a loop; this follows from the above observation that the two white vertices associated with $m Z_a^{(\f a \setminus \{a\})}$ cannot be paired. From \eqref{condition for marked vertices} and \eqref{conservation of degree}, we therefore find that at least one vertex in $p^{-1}(i)$ is a chain vertex. Consequently summation over $\f x$ results in an extra factor $\Phi$ by Proposition \ref{lemma: weak Z lemma}, and hence completes the argument.

\subsection{Removing Simplification {\bf (S4)}} \label{section: Z8}
In this section we remove Simplification {\bf (S4)}, by allowing the fresh summation indices $\f x$ to coincide with each other and with external indices $\f \mu$. This entails proving Proposition \ref{lemma: weak Z lemma} without the simplifying assumption {\bf (S4)} that was assumed in its proof. Roughly, there are two kinds of problems arising from such coincidences: an off-diagonal resolvent entry $G_{xy}$ may become diagonal (hence leading to a loss of a factor $\Psi$), and a chain vertex may cease to be one (hence leading to a loss of a factor $\Phi$). However, these losses are compensated by powers of $M^{-1}$ resulting from a reduction in the number of independent summation variables. The main point is to prove each coincidence of summation variables
results in a loss of at most two factors of $\Psi$ and at most two factors of $\Phi$. Since $M^{-1}\leq \Psi^2\Phi^2$, the gain of $M^{-1}$ will be enough to compensate this loss.

Throughout Sections \ref{section: Z5}, \ref{section: Z6}, and \ref{section: Z7}, we invoked Proposition \ref{lemma: weak Z lemma} in order to gain from chain vertices. To that end, we had to assume Simplification {\bf (S4)} (since the indices $(\f x, \f \mu)$ are assumed to be disjoint in Proposition \ref{lemma: weak Z lemma}).
The main result of this section is the following extension of Proposition \ref{lemma: weak Z lemma}. It states that the stochastic bound of Proposition \ref{lemma: weak Z lemma} is valid even if the summation over $\f x$ has no restriction. (As in Proposition \ref{lemma: weak Z lemma}, we use $\f a$ to denote the summation indices; in our applications of Proposition \ref{corollary: weak Z lemma} $\f a$ always consists of fresh summation vertices which we denoted by $\f x$ in Sections \ref{section: Z5}, \ref{section: Z6}, and \ref{section: Z7}.)

\begin{proposition} \label{corollary: weak Z lemma}
Suppose that $\Lambda \prec \Psi$ for some admissible control parameter $\Psi$. Let $\Delta$ be a chain encoding $\cal Z_{\f a}$. Then
\begin{equation} \label{weak estimate for coinciding indices}
\sum_{\f a} w(\f a) \cal Z_{\f a} \;\prec\; \Psi^{\deg(\Delta)} \Phi^{c(\Delta)}
\end{equation}
for any $\f \mu$ and chain weight $w$.
\end{proposition}
\begin{proof}
The basic idea is to split the summation into partitions
\begin{equation*}
\sum_{\f a} w(\f a) \cal Z_{\f a} \;=\; \sum_{P} \sum_{\f a} \indb{\cal P(\f \mu, \f a) = P} w(\f a) \cal Z_{\f a}\,,
\end{equation*}
where $P$ ranges over all partitions of $V(\Delta)$, and $\cal P$ was introduced in Definition \ref{def: partition}. Note that, since $\f \mu$ are constrained to be distinct, if $P$ yields a nonzero contribution each of its blocks may contain at most one vertex in $V_e(\Delta)$. For the following we fix a partition $P$ and prove that
\begin{equation*}
\cal X_P \;\deq\; \sum_{\f a} \indb{\cal P(\f \mu, \f a) = P} w(\f a) \cal Z_{\f a}
\end{equation*}
is stochastically bounded by the right-hand side of \eqref{weak estimate for coinciding indices}. On the level of the graph $\Delta$, a nontrivial partition $P$ of $V(\Delta)$ results in a merging of vertices $V(\Delta)$. By merging vertices of $\Delta$ we therefore get a new graph which we denote by $P(\Delta)$. The vertex set of $P(\Delta)$ has the usual decomposition $V(P(\Delta)) = V_e(P(\Delta)) \sqcup V_s(P(\Delta))$, where $V_e(P(\Delta)) = V_e(\Delta)$ and $V_s(P(\Delta))$ is given by the set of blocks of $P$ that do not contain a vertex from $V_e(\Delta)$. A vertex $i \in V(P(\Delta))$ is \emph{unmerged} if the corresponding block has size one, and \emph{merged} otherwise. See Figure \ref{figure: merging} for an example of the merging $\Delta \mapsto P(\Delta)$.
\begin{figure}[ht!]
\begin{center}
\includegraphics{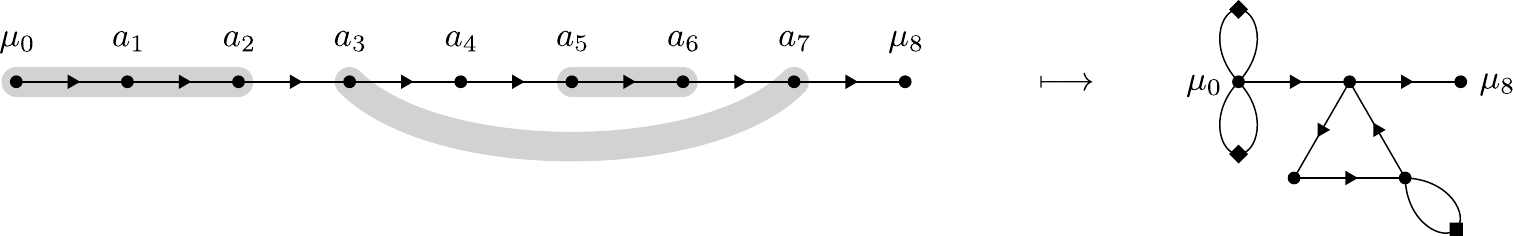}
\end{center}
\caption{The merging $\Delta \mapsto P(\Delta)$ of summation vertices of a chain. The vertices of $\Delta$ are $V_s(\Delta) = \{1, \dots, 7\}$ and $V_e(\Delta) = \{0,8\}$. We chose the partition $P = \h{\{0,1,2\}, \{3,7\}, \{4\}, \{5,6\}, \{8\}}$. \label{figure: merging}}
\end{figure}

For concreteness assume first that $\Delta$ is an open chain with $V(\Delta) = \{0, \dots, n\}$, where $V_e(\Delta) = \{0, n\}$. Thus, $\deg(\Delta) = n$. For any graph $\Delta'$ define the set of vertices
\begin{equation*}
V_g(\Delta') \;\deq\; \hb{i \in V_s(\Delta') \col \text{$i$ has degree two without counting loops}}\,.
\end{equation*}
Thus, the set $V_g(P(\Delta))$ includes not only the chain vertices of $P(\Delta)$ but also chain vertices to which one more loops are attached.
Let $r(\Delta')$ denote the number of edges of $\Delta'$ that are not loops, and set $c(\Delta') \deq \abs{V_g(\Delta')}$.
In particular, if $\Delta'$ is a chain then this definition agrees with that from Definition \ref{def: chains}.
(For example, in Figure~\ref{figure: merging} we have $c(P(\Delta))=2$ and $r(P(\Delta))=5$.)

Define $k \deq n - 1 - \abs{V_s(P(\Delta))}$. Informally, $k$ is the number of summation vertices of $V_s(\Delta)$ that have been merged into some other vertex. As we shall see, $k$ is the exponent of $M^{-1}$ which describes the reduction in the combinatorics of the summation.
(In Figure~\ref{figure: merging} we have $n=8$ and $ \abs{V_s(P(\Delta))}=3$, which gives $k=4$.)
We claim that
\begin{equation} \label{main estimate for coinciding indices}
n - k \;\leq\; r(P(\Delta)) \;\leq\; n\,, \qquad  r(P(\Delta))+ c(P(\Delta)) \;\geq\; 2n - 1 - 2k\,.
\end{equation}
The easiest way to prove \eqref{main estimate for coinciding indices} is by the following inductive argument.
We construct $P(\Delta)$ from $\Delta$ by successively merging one vertex at a
time, and follow the change of the functions $r(\cdot)$ and $r(\cdot)+c(\cdot)$ at each step. The formal procedure is the following.
We construct a sequence of graphs $\Delta_0 = \Delta, \Delta_1, \dots, \Delta_k = P(\Delta)$ as follows. We start from $\Delta_0 \deq \Delta$. Recall that the vertices of $\Delta$ are naturally ordered by $\leq$. Let $i_1 \in V_s(\Delta)$ be the smallest vertex of $\Delta$ that is in a nontrivial block (i.e.\ of size greater than one) of $P$. Set $\Delta_1$ to be the graph obtained from $\Delta_0$ by merging $i_1$ with the (unique) vertex $j \in V(\Delta_0)$ satisfying $j < i_1$. The vertices of $\Delta_1$ remain ordered after we assign the newly created merged vertex the index $j$. Similarly, $\Delta_{l+1}$ is obtained from $\Delta_l$ by choosing the smallest unmerged vertex $i_l \in V_s(\Delta_l)$ that is in a nontrivial block of $P$, and merging it with the unique $j \in V(\Delta_l)$ satisfying $j < i_l$. After $k$ steps of this procedure, we obtain $\Delta_k = P(\Delta)$. Moreover, it is easy to see for $0 \leq l \leq k - 1$ that
\begin{equation} \label{iteration for coinciding indices}
r(\Delta_l) - 1 \;\leq\; r(\Delta_{l + 1}) \;\leq\; r(\Delta_l)\,, \qquad r(\Delta_{l + 1}) + c(\Delta_{l + 1}) \;\geq\; r(\Delta_l) + c(\Delta_l)- 2\,.
\end{equation}
Indeed, either $i_l$ is merged with a vertex adjacent to itself, in which case we have $r(\Delta_{l+1}) = r(\Delta_l) - 1$ and $c(\Delta_{l+1}) \geq c(\Delta_l) - 1$, or $i_l$ is merged with a vertex not adjacent to itself, in which case we have $r(\Delta_{l+1}) = r(\Delta_l)$ and $c(\Delta_{l+1}) \geq c(\Delta_l) - 2$. Since $r(\Delta) + c(\Delta) = 2n - 1$, \eqref{main estimate for coinciding indices} follows from \eqref{iteration for coinciding indices}.

We may now sum over $(a_i)_{i \in V_s(P(\Delta))}$. To that end, if $i \in V_g(P(\Delta))$ and there is a loop (or several loops) at $i$, then we expand each corresponding diagonal term $G_{a_i a_i}$ as $G_{a_i a_i} = m + (G_{a_i a_i} - m)$. If we pick a factor $m$ from each loop, $i$ becomes a chain vertex. If we pick at least one factor $G_{a_i a_i} - m$, $i$ is not a chain vertex but carries a factor of order $\Psi$. Either way, summing over $a_i$ yields a factor $\Phi$ by Proposition \ref{lemma: weak Z lemma}. (Note that Proposition \ref{lemma: weak Z lemma} is applicable to the graph $P(\Delta)$ because all summation indices are constrained to be distinct.) Thus we get the bound
\begin{equation*}
\cal X_P \;\prec\; M^{-k} \Psi^{r(P(\Delta))} \Phi^{c(P(\Delta))}
\;\leq\; M^{-k} \Psi^{n - k} \Phi^{n - 1 - 2k}\,,
\end{equation*}
where in the last step we used \eqref{main estimate for coinciding indices}. Since
$M^{-1} \Psi^{-1} \Phi^{-2} \leq \Psi \leq 1$
we find $\cal X_P \prec \Psi^n \Phi^{n - 1}$,
which is \eqref{weak estimate for coinciding indices}.

The case of a closed chain $\Delta$ of degree $n$ is handled similarly. For definiteness assume that $\Delta$ has no external vertex. Now we have $k \deq n - \abs{V_s(P(\Delta))} \leq n - 1$ and we let $l$ range from $0$ to $k$. Then \eqref{iteration for coinciding indices} holds for $l = 0, \dots, n - 3$. If $l = n - 2$ then \eqref{iteration for coinciding indices} is in general false (as can be seen e.g.\ on the open chain of degree two with $V_s(\Delta) = \{1,2\}$ and $P = \{\{1,2\}\}$). In that case we replace it with the trivial bounds $r(\Delta_{n - 1}) \geq 0$ and $c(\Delta_{n - 1}) \geq 0$. Thus if $k \leq n - 2$ then we find \eqref{weak estimate for coinciding indices} exactly as above, and if $k = n - 1$ we get using $n \geq 2$
\begin{equation*}
\cal X_P \;\prec\; M^{-k} \Psi^{r(P(\Delta))} \Phi^{c(P(\Delta))} \;\leq\; M^{-n + 1} \;\leq\; \Psi^n \Phi^n\,,
\end{equation*}
which is \eqref{weak estimate for coinciding indices}.
\end{proof}

To conclude this section, we address an issue concerning coinciding indices that was repeatedly swept under the rug in Sections \ref{sec: chains}, \ref{section: Z5}, \ref{section: Z6}, and \ref{section: Z7}. Essentially, we do an inclusion-exclusion argument on the values of the summation indices of a union of chains so as to decouple the summations associated with different subchains. Recall the definition of $X_F^w(\Delta)$ from \eqref{general sum of Gs}.

\begin{lemma} \label{lemma: factor chains}
If $\Delta = \Delta_1 \cup \cdots \cup \Delta_k$ is a union\footnote{By union we mean that the chains $\Delta_1, \dots, \Delta_k$ may share external vertices but not summation vertices.} of chains then
\begin{equation} \label{factor chains}
X^w_\emptyset(\Delta) \;\prec\; \prod_{l = 1}^k \Psi^{\deg(\Delta_l)} \Phi^{c(\Delta_l)}\,.
\end{equation}
\end{lemma}

In words: if $\Delta$ is a union of chains, then in the summation over $\f a$ in $X^w_\emptyset(\Delta)$ we can decouple the summations associated with different subchains of $\Delta$. (In $X^w_\emptyset(\Delta)$ these summations are coupled by the constraint that indices associated with different subchains are distinct.)

\begin{proof}[Proof of Lemma \ref{lemma: factor chains}]
This is a simple decoupling of the summation indices. Let $\f a^l$ and $\f \mu^l$ denote the summation and external indices of $\Delta_l$. Abbreviate $\cal Z_{\f a^l}^{\f \mu^l}(\Delta_l) \equiv \cal Z^l_{\f a^l}$. Thus we have
\begin{equation*}
X_\emptyset(\Delta) \;=\; \sum_{\f a^1 \cdots \f a^k}^{(\f \mu^1 \cdots \f \mu^k)*} w(\f a^1, \dots, \f a^k) \cal Z^1_{\f a^1} \cdots \cal Z^k_{\f a^k}
\;=\; \sum_{\f a^1 \cdots \f a^k} I(\f a^1, \dots, \f a^k) \, w(\f a^1, \dots, \f a^k) \cal Z^1_{\f a^1} \cdots \cal Z^k_{\f a^k}\,,
\end{equation*}
where the indicator function $I$ explicitly enforces that all $a_i^l$'s are distinct from all $\mu_i^l$'s
and they are district among themselves. Explicitly,
\begin{multline*}
I(\f a^1, \dots, \f a^k) \;\deq\; \qBB{\prod_{l = 1}^k \prod_{i,j \in V_s(\Delta_l)}^* \pb{1 - \ind{a_i^l = a_j^l}}}
\qBB{\prod_{l,m \leq k} \prod_{i \in V_s(\Delta_l)}  \prod_{j \in V_e(\Delta_m)} \pb{1 - \ind{a_i^l = \mu_j^m}}}
\\
\times
\qBB{\prod_{l,m \leq k}^* \prod_{i \in V_s(\Delta_l)} \prod_{j \in V_s(\Delta_m)} \pb{1 - \ind{a_i^l = a_j^m}}}\,.
\end{multline*}
Multiplying out each parenthesis in the definition of $I$, we get a splitting of the form $I = \sum_\alpha I_\alpha$ (the sum ranges over a finite set which depends only on $\Delta$). For each $\alpha$, we may now estimate
\begin{equation} \label{splitting of coinciding for chains}
\sum_{\f a^1 \cdots \f a^k} I_\alpha(\f a^1, \dots, \f a^k) \, w(\f a^1, \dots, \f a^k) \cal Z^1_{\f a^1} \cdots \cal Z^k_{\f a^k} \;\prec\; \prod_{l = 1}^k \Psi^{\deg(\Delta_l)} \Phi^{c(\Delta_l)}\,.
\end{equation}
To see this, we note that picking the term $\ind{\cdots}$ from the parenthesis $(1 - \ind{\cdots})$ results in the merging of two vertices. Thus, the left-hand side of \eqref{splitting of coinciding for chains} is encoded by a graph $\Delta^{(\alpha)}$ obtained from $\Delta$ by merging vertices according to $I_\alpha$. Whenever two vertices are merged, we may lose two chain vertices, but gain a power $M^{-1}$ from the chain weight (since if indices $a$ and $a'$ coincide, then one
of the factors $s_{ab}$ and $s_{a'b'}$ in the chain weight (see \eqref{chain weight}) can be dropped from
the weight and estimated by $M^{-1}$). The associated loss of $\Phi^2$ is therefore compensated by $M^{-1} \leq \Phi^2$.
In order to gain from the chain vertices in the merged graph, we invoke Lemma \ref{corollary: weak Z lemma} to get
\begin{equation} \label{product estimte for subchain}
\sum_{\f a'} w'(\f a') \cal Z_{\f a'}(\Delta') \;\prec\; \Phi^{\deg(\Delta')} \Phi^{c(\Delta')}
\end{equation}
for each subchain $\Delta'$ of $\Delta^{(\alpha)}$.
Here \eqref{product estimte for subchain} is applicable because the left-hand side of \eqref{splitting of coinciding for chains} factors into a product of expressions encoded by the subchains of $\Delta^{(\alpha)}$ (i.e.\ there are no summation constraints that involve two different subchains of $\Delta^{(\alpha)}$). This completes the proof of \eqref{splitting of coinciding for chains}, and hence of \eqref{factor chains}.
\end{proof}

\subsection{Removing Simplification {\bf (S1)} and completion of the proof of Theorem \ref{theorem: Z lemma}} \label{section: Z9}
In this section we remove Simplification {\bf (S1)} and put the arguments from Sections \ref{section: Z3}, \ref{section: Z5}, \ref{section: Z6}, \ref{section: Z7}, and \ref{section: Z8} together to complete the proof of Theorem \ref{theorem: Z lemma} in full generality.

The following tensorization property of weights plays a crucial role in this section.
\begin{lemma} \label{lemma: product weight}
If $w'(\f a')$ and $w''(\f a'')$ are weights then so is $w(\f a', \f a'') \deq w'(\f a') w''(\f a'')$.
\end{lemma}
\begin{proof}
The claim easily follows from Definition \ref{definion: w}.
\end{proof}

Recall that Simplification {\bf (S1)} states that no index coincidences occur among the indices $\f a$ when we compute the $p$-th power of $X_F(\Delta)$, i.e.\ in going from $\Delta$ to $\gamma^p(\Delta)$. In order to relax Simplification {\bf (S1)}, we go back to Section \ref{section: Z3}. In this section we add a tilde to the original summation indices in \eqref{E X^p}: $\tilde {\f a} = (\tilde a_i)_{i \in V_s(\gamma^p(\Delta))}$ (we shall use $\f a$ to denote the merged summation indices; see below). Let $\tilde w(\tilde{\f a})$ denote the product weight (see Lemma \ref{lemma: product weight}) in the $p$-fold copy of $X_F(\Delta)$. In general, if we do not assume Simplification {\bf (S1)} then in \eqref{E X^p} the original summation vertices $\tilde{\f a}$ associated with different copies of $\Delta$ may coincide. As in the proof of Proposition \ref{corollary: weak Z lemma}, we split the summation using partitions by introducing the factor
\begin{equation*}
1 \;=\; \sum_P \ind{\cal P(\tilde{\f a}) = P}
\end{equation*}
into the right-hand side of \eqref{E X^p}. Here the summation ranges over partitions of $V_s(\gamma^p(\Delta))$. Thus we get a finite collection of terms indexed by partitions $P$, which we estimate individually (The combinatorics stemming from the number of partitions is independent of $N$ 
and will be included in the irrelevant constant prefactors in the final estimate).

Thus, for the sequel we choose and fix a partition $P$ of $V_s(\gamma^p(\Delta))$. If two vertices of $V_s(\gamma^p(\Delta))$ are in the same block of $P$, we merge them and get a single vertex. Thus we get a new graph which we denote by $\gamma^p_P(\Delta)$. As before we have the splitting $V(\gamma^p_P(\Delta)) = V_e(\gamma^p_P(\Delta)) \sqcup V_s(\gamma^p_P(\Delta))$, where $V_e(\gamma^p_P(\Delta)) = V_e(\gamma^p(\Delta)) = V_e(\Delta)$ and $V_s(\gamma^p_P(\Delta))$ is given by the blocks of $P$. We use $\f a = (a_i)_{i \in V_s(\gamma^p_P(\Delta))}$ to denote the summation indices of the graph $\gamma^p_P(\Delta)$. Each summation vertex of $\gamma^p_P(\Delta)$ is either \emph{unmerged} or \emph{merged}, depending on whether the associated block of $P$ is of size one or greater than one. We have the trivial lift $\tilde{\f a} = L_P(\f a)$ defined by $\tilde a_l = a_i$ if $l$ belongs to the block $i$ of $P$. In merging two vertices $i$ and $j$ in $V_s(\gamma^p(\Delta))$, we lose in general all mechanisms that extract smallness (ingredients (b) and (c) in the list of the guiding principle of Section \ref{section: Z1})
from them, including the linking associated with the possible factors $Q_{a_i}$ or $Q_{a_j}$. On the other hand, we gain a factor $M^{-1}$ from the reduction of the combinatorics of the summation. Generally, the reduced summation yields a factor $M^{\abs{V_s(\gamma^p_P(\Delta))} - \abs{V_s(\gamma^p(\Delta))}}$. More precisely,
\begin{equation} \label{merged weight}
\sum_{\f a} w_P(\f a) \;\leq\; 1 \,, \qquad w_P(\f a) \;\deq\; \tilde w(L_P(\f a)) M^{\abs{V_s(\gamma^p(\Delta))} - \abs{V_s(\gamma^p_P(\Delta))}}\,.
\end{equation}
This follows from \eqref{definition of weight} and the fact that $\tilde w$ is a weight by Lemma \ref{lemma: product weight}. We stress that this is the only point where the assumption \eqref{definition of weight} is needed in our proof.

Having fixed the merging of the vertices, we may now construct all graphs $\Gamma \in \fra G^p_{F,P}(\Delta)$; note that this set now depends on $P$. Here $\fra G^p_{F,P}(\Delta)$ is constructed using the same algorithm as $\fra G^p_F(\Delta)$ in Section \ref{section: Z3}. In this case, however, each graph $\Gamma \in \fra G^p_{F,P}(\Delta)$ has the property that \emph{unmerged} summation vertices of $\gamma^p_P(\Delta)$ which come with a $Q$ have have been linked with an edge of $\Gamma$. There is no similar constraint for merged vertices. (The proof is the same as that for $\fra G^p_F(\Delta)$ in Section \ref{section: Z3}.)

Now we may repeat the arguments of Sections \ref{section: Z5}, \ref{section: Z6}, \ref{section: Z7}, and \ref{section: Z8} almost verbatim. The only difference is that we only gain from the \emph{unmerged} vertices of $\Gamma$. For example, if $i \in V_s(\Gamma)$ is unmerged and satisfies $i \in \pi^{-1}(F)$, then it must have been linked with an edge. Similarly, if $i \in V_s(\Gamma)$ is unmerged and marked, it will give rise to a chain vertex after vertex resolution, and hence a factor $\Phi$.

In order to account for the gain from the merged original summation vertices of $\Gamma$, we interpret the estimate \eqref{merged weight} as stating that
each summation vertex $i$ of $\gamma^p(\Delta)$  carries a factor $M^{-1/2}$. 
This means if $i$ is merged then we gain a factor $M^{-1/2}$ 
over the unmerged scenario. (It is easy to see
that this counting corresponds to the worst-case scenario where vertices 
of $\gamma^p(\Delta)$ were paired to get $\gamma^p_P(\Delta)$.  For example,
if we have a weight $w(a,b,c,d)$ with $\sum_{abcd} w(a,b,c,d)=1$
and we merge $a$ with $b$ and $c$ with $d$,
then the new weight $w_{P}(a,c)= w(a,a,c,c)$ will sum up to
\begin{equation*}
\sum_{a,c} w_{P}(a,c) \;=\; \sum_{a,c} w(a,a,c,c) \;\leq\; M^{-2}
\end{equation*}
by \eqref{definition of weight}.  The gain of order $M^{-2}$ can then be distributed among the
four vertices involved in the merging, each receiving a factor $M^{-1/2}$.)
This gain of $M^{-1/2}$ compensates any possible gain associated with $i$, which is at best $\Psi \Phi$ (in the case where $\pi(i)$ is marked and belongs to $F$). See the guiding principle in Section \ref{section: Z1}.

The proof is then completed by the simple observation that $M^{-1/2} \leq \Psi \Phi$.

\section{Proof of Theorem \ref{theorem: Z lemma variant}} \label{sec: proof of variant}

In this section we prove Theorem \ref{theorem: Z lemma variant}. The proof relies on some ideas from the proof of Theorem \ref{theorem: Z lemma}, but is considerably easier. The strategy is to resolve (using the Family B identities) the summation vertices (associated with indices $\f a$) using the partial expectation $\prod_{a \in \f a} P_a$, and to 
estimate the resulting averaging using Theorem \ref{theorem: Z lemma}. Thus, unlike in the proof of Theorem \ref{theorem: Z lemma}, there is no need to estimate high moments.

Before giving the general proof, let us consider the simple example
\begin{align*}
P_a G_{\mu a} G_{a \mu} &\;=\; P_a \frac{m^2}{G_{aa}^2} G_{\mu a} G_{a \mu} + P_a \pbb{1 - \frac{m^2}{G_{aa}^2}} G_{\mu a} G_{a \mu}
\\
&\;=\; m^2 P_a \sum_{x,y}^{(a)} G_{\mu x}^{(a)} h_{x a} h_{a y} G_{y \mu}^{(a)} + O_\prec(\Psi^3)
\\
&\;=\; m^2 \sum_{x}^{(a)} s_{ax} G_{\mu x}^{(a)} G_{x \mu}^{(a)} + O_\prec(\Psi^3)
\\
&\;=\; m^2 \sum_{x}^{(a)} s_{ax} G_{\mu x} G_{x \mu} + O_\prec(\Psi^3)
\\
&\;=\; O_\prec(\Psi^2 \Phi)\,,
\end{align*}
where in the second step we used \eqref{res exp 2b} and the bound $\Lambda \prec \Psi$, in the third step \eqref{variance of h}, in the fourth step \eqref{resolvent expansion type 1}, and in the last step Theorem \ref{theorem: Z lemma} (or Proposition \ref{lemma: weak Z lemma}).

The argument for a general graph $\Delta$ is similar. We have to gain a factor $\Phi$ from each vertex $i \in V_c(\Delta)$ (in addition to the trivial $\deg(\Delta)$ factors $\Psi$). We use the terminology of Sections \ref{section: Z3} -- \ref{section: simplifications} without further comment. The proof consists of the following steps, which we merely sketch as they are almost identical to those of Sections \ref{section: Z5} and \ref{section: simplifications}.
\begin{enumerate}
\item
Make all entries of $\cal Z_{\f a}(\Delta)$ maximally expanded in $\f a$ using the algorithm from the proof of Lemma \ref{lemma: weak moment estimate}. The resulting linking yields a set of graphs $\fra G(\Delta)$ satisfying (recall the notation \eqref{Max})
\begin{equation*}
\cal Z_{\f a}(\Delta) \;=\; \sum_{\Gamma \in \fra G(\Delta)} \cal A_{\f a}(\Gamma) + O_\prec(\Psi^{\deg(\Delta) + \abs{V_s(\Delta)}})\,,
\end{equation*}
where all resolvent entries of $\cal A_{\f a}(\Gamma)$ are maximally expanded in $\f a$. Each graph $\Gamma \in \fra G(\Delta)$ resulted from $\Delta$ by a finite number (possibly zero) of linking operations. In particular, $V_s(\Gamma) = V_s(\Delta)$.
\item
Let $V_m(\Gamma) \subset V_c(\Delta)$ denote those vertices of $V_c(\Delta)$ that were not linked to in the process $\Delta \mapsto \Gamma$. (In other words, $i \in V_m(\Gamma)$ if and only if $\deg_\Delta(i) = \deg_\Gamma(i)$.) We have to gain a factor $\Phi$ from each vertex $i \in V_m(\Gamma)$; note that each $i \in V_c(\Delta) \setminus V_m(\Gamma)$ yields a factor $\Psi$ due to the additional edge incident to $i$ produced by the linking to $i$.

Now we follow the vertex resolution of Sections \ref{section: Z5}, \ref{section: Z6}, and \ref{section: Z7} to the letter. The only difference is that the $\cal A_{\f a}(\Gamma)$ is not contained within a full expectation $\E$ but a partial expectation $\prod_{a \in \f a}P_a$ instead. 
We resolve all vertices in $V_s(\Gamma)$, which yields the splitting
\begin{equation*}
\cal Z_{\f a}(\Delta) \;=\; \sum_{\Gamma \in \fra G(\Delta)} \sum_{\Upsilon \in \fra L(\fra R(\Gamma))} \sum_{\f x}^{(\f a)} \cal A_{\f a, \f x}(\Upsilon) + O_\prec(\Psi^{\deg(\Delta) + \abs{V_s(\Delta)}})\,,
\end{equation*}
where $\f x \in \{1, \dots N\}^{V_f(\Upsilon)}$ denotes the fresh summation indices of $\Upsilon$.
\item
Exactly as in Sections \ref{sec: completion of proof under all S} and \ref{section: Z6}, each vertex $i \in V_m(\Gamma)$ either carries an extra factor $\Phi$ (if an error term of subleading order was chosen in the resolution of $i$) or gives rise to a fresh summation vertex $j \in p^{-1}(i)$ that is a chain vertex of $\Upsilon$. Hence we may invoke Theorem \ref{theorem: Z lemma}, for each fixed $\Upsilon \in \fra L(\fra R(\Gamma))$, to get
\begin{equation*}
\sum_{\f x}^{(\f a)} \cal A_{\f a, \f x}(\Upsilon) \;\prec\; \Psi^{\deg(\Delta)} \Phi^{\abs{V_c(\Delta)}}\,.
\end{equation*}
This concludes the proof of Theorem \ref{theorem: Z lemma variant}.
\end{enumerate}

\appendix

\section{Basic resolvent bounds} \label{section: proof of res id}

In this appendix we collect some useful tools about resolvents, and in particular prove Lemmas \ref{lemma: Lambda T}, \ref{lemma: rough bounds on G}, and \ref{lemma: diagonal estimates}.

\begin{proof}[Proof of Lemma \ref{lemma: Lambda T}]
Let $\epsilon > 0$ and $D > 0$ be arbitrary. From \eqref{lower bound on W} and \eqref{admissible Psi} we find that there exists $c_0, c_1 \in (0,\epsilon/2)$ and an event $\Xi$ such that
\begin{equation*}
\Lambda(z) \ind{\Xi} \;\leq\; N^{c_0} \Psi(z) \;\leq\; N^{-c_1}
\end{equation*}
for all $z \in \f S$ and large enough $N$, and $\P(\Xi^c) \;\leq\; N^{- D}$.
Thus we conclude using \eqref{m is bounded} that
\begin{equation*}
\sup_{z \in \f S} \max_i \pB{\absB{1/G_{ii}(z)} \ind{\Xi}} \;\leq\; C
\end{equation*}
for large enough $N$. Using the first identity of \eqref{resolvent expansion type 1} and \eqref{m is bounded} again, we find
\begin{equation} \label{induction for minors}
\max_{\abs{T} = \ell} \max_{i,j \notin T} \pB{\absb{G_{ij}^{(T)}(z) - \delta_{ij} m(z)} \ind{\Xi}} \;\leq\; C N^{c_0} \Psi(z) \,, \qquad \sup_{z \in \f S} \max_{\abs{T} = \ell} \max_{i \notin T} \pB{\absB{1/G^{(T)}_{ii}(z)} \ind{\Xi}} \;\leq\; C\,.
\end{equation}
for $\ell = 1$. Using the first identity of \eqref{resolvent expansion type 1} and \eqref{m is bounded}, we may now proceed inductively on $\ell = 1,2,\dots$, at each step proving \eqref{induction for minors} for $\ell$ assuming it holds for $\ell - 1$. The result is
\begin{equation} \label{minor estimate}
\sup_{\abs{T} \leq \ell} \max_{i,j \notin T} \pB{\absb{G_{ij}^{(T)}(z) - \delta_{ij} m(z)} \ind{\Xi}} \;\leq\; C_\ell N^{c_0} \Psi(z) \;\leq\; N^\epsilon \Psi(z)
\end{equation}
for all $z \in \f S$. This concludes the proof.
\end{proof}

\begin{proof}[Proof of Lemma \ref{lemma: rough bounds on G}]
The estimate \eqref{rough bound 1} follows immediately from $\absb{G_{ij}^{(T)}(E + \ii \eta)} \leq 
\eta^{-1}$ and the definition of $\f S$.

In order to prove \eqref{rough bound 3}, we choose $D \deq 10 p$ and let $\Xi$ denote the event from the proof of Lemma \ref{lemma: Lambda T} above. First we deal with the high-probability event $\Xi$. From \eqref{minor estimate} and \eqref{m is bounded} we immediately get
\begin{equation} \label{rough estimate on good event}
\sup_{z \in \f S} \sup_{\abs{T} \leq \ell} \max_{i \notin T} \pB{\absB{1/G_{ii}^{(T)}(z)} \ind{\Xi}} \;\leq\; C\,.
\end{equation}
In order to handle the exceptional event $\Xi^c$, we use Schur's formula \eqref{schur}.
Then by Cauchy-Schwarz, \eqref{rough bound 1}, and \eqref{finite moments}, we find
\begin{equation} \label{rough estimate on bad event}
\E \pbb{\absB{1/G_{ii}^{(T)}(z)}^p \ind{\Xi^c}} \;\leq\; \qbb{\E \pbb{\absB{1/G_{ii}^{(T)}(z)}^{2p} \ind{\Xi^c}}}^{1/2} \P(\Xi^c)^{1/2}
\;\leq\; (C + N^3)^{p} N^{-5p}\,.
\end{equation}
Combining \eqref{rough estimate on good event} and \eqref{rough estimate on bad event} yields \eqref{rough bound 3}.
\end{proof}

\begin{proof}[Proof of Lemma \ref{lemma: diagonal estimates}]
To simplify notation, we set $T = \emptyset$ (the proof for nonempty $T$ is the same).
A simple large deviation estimate (see e.g.\ Lemmas B.1 and B.2 in \cite{EYY1}) applied to
\begin{equation*}
Z_i \;=\; \sum_{k}^{(i)} \pb{\abs{h_{ik}}^2 - s_{ik}} G^{(i)}_{kk} + \sum_{k \neq l}^{(i)} h_{ik} G^{(i)}_{kl} h_{li}
\end{equation*}
implies $Z_i \prec \Psi$.

As above, for the estimate of $U_i^{(S)}$ we set $S = \emptyset$ to simplify notation.
Using \eqref{S is stochastic} we write
\begin{equation*}
U_i \;=\; \sum_k s_{ik} (G_{kk} - m) \;=\; \sum_k s_{ik} P_k (G_{kk} - m) + \sum_k s_{ik} Q_k (G_{kk} - m) \;=\; \sum_k s_{ik} P_k (G_{kk} - m) + O_\prec(\Psi^2)\,,
\end{equation*}
where the last step follows from Proposition \ref{prop: warm-up}. Now we expand the inverse of
\eqref{res exp 2c} using \eqref{identity for msc} to get
\begin{equation*}
G_{kk} - m \;=\; m^2 \pb{-h_{kk} + Z_k + U_k^{(k)}} + O_\prec(\Psi^2)\,,
\end{equation*}
where we estimated the higher-order terms using \eqref{m is bounded} and the trivial bounds $h_{ii} \prec \Psi$, $U_i^{(i)} \prec \Psi$, and $Z_i \prec \Psi$ (as proved in the previous paragraph). Using $P_k h_{kk} = 0$ and $P_k Z_k = 0$ we therefore get
\begin{align*}
U_i &\;=\; m^2 \sum_k s_{ik} P_k U_k^{(k)} + O_\prec(\Psi^2)
\\
&\;=\; m^2 \sum_{k} s_{ik} \pbb{\sum_{l}^{(k)} s_{kl} P_k G_{ll}^{(k)} - m} + O_\prec(\Psi^2)
\\
&\;=\; m^2 \sum_{k,l} s_{ik} s_{kl} (G_{ll} - m) + O_\prec(\Psi^2)
\\
&\;=\; m^2 \sum_{k} s_{ik} U_k + O_\prec(\Psi^2)\,,
\end{align*}
where in the third step we used \eqref{s leq W}, \eqref{S is stochastic}, \eqref{admissible Psi}, and \eqref{resolvent expansion type 1}. Inverting the operator $1 - m^2 S$ therefore yields $U_i \prec \varrho \Psi^2$. On the other hand, the estimate $U_i \prec \Psi$ is trivial. This concludes the proof.
\end{proof}

\section{The coefficient $\varrho$ for band matrices} \label{appendix: rho for band}

In this section we prove an explicit bound for the coefficient $\varrho$ defined in \eqref{def of rho}, 
in the case that $S$ is the variance matrix of
 a band matrix $H$, as defined in Example \ref{example: band matrix}.
In fact, we need only that the spectrum $\sigma(S)$ of $S$ is separated away from $-1$; that is this always true for band matrices is the content of the following lemma.

\begin{lemma} \label{lemma: lower gap}
Suppose that $H$ is a $d$-dimensional band matrix from Example \ref{example: band matrix}.
Then there is a constant $\delta_- > 0$, depending only on  the profile function
$f$, such that $\sigma(S) \subset [-1 + \delta_-, 1]$.
\end{lemma}
\begin{proof}
See \cite[Lemma A.1]{EYY1}.
\end{proof}

\begin{proposition} \label{prop: rho for band} Let $S$ be a doubly stochastic matrix satisfying $\sigma(S) \subset [-1 + \delta_-, 1]$ for some $\delta_->0$.
Then there is a universal constant $C$ such that
\begin{equation}\label{rhob}
\varrho \;\leq\; \frac{C \log N}{\min \h{\delta_-, (\im m)^2}}\,.
\end{equation}
In particular, using Lemma~\ref{lemma: lower gap} we find that \eqref{rhob} holds for a $d$-dimensional band matrix from Example \ref{example: band matrix}, with a constant $C$ depending only on the profile function $f$.
\end{proposition}

The rest of this appendix is devoted to the proof of Proposition \ref{prop: rho for band}. A similar argument was given in the proof of \cite[Lemma 3.5]{EYY1}. The main difference is that here we do not assume the existence of a spectral gap near $+1$ in the spectrum of $S$.

\begin{proof}[Proof of Proposition \ref{prop: rho for band}]
Abbreviate $\zeta \deq m^2$ and write
\begin{equation*}
\frac{1}{1 - \zeta S} \;=\; \frac{1/2}{1 - (1 + \zeta S)/2}\,.
\end{equation*}
We have the bound
\begin{equation*}
\normbb{\frac{1 + \zeta S}{2}}_{\ell^\infty \to \ell^\infty} \;\leq\; \max_i \sum_j \absbb{\pbb{\frac{1 + \zeta S}{2}}_{ij}} \;\leq\; 1\,,
\end{equation*}
where we used that $\abs{\zeta} \leq 1$ as follows from \eqref{m is bounded}. 
By the condition on the spectrum, we have
\begin{equation*}
\normbb{\frac{1 + \zeta S}{2}}_{\ell^2 \to \ell^2} \;\leq\; \max_{x \in [-1 + \delta_-,1]} \frac{\abs{1 + \zeta x}}{2} \;\leq\; \max \hbb{1 - \frac{\delta_-}{2}, \frac{\abs{1 + \zeta}}{2}}\,.
\end{equation*}
An elementary calculation yields $1 - \abs{1 + \zeta}/2 \geq c (\im m)^2$ for some constant $c > 0$, from which we conclude
\begin{equation} \label{l2 operator bound}
\normbb{\frac{1 + \zeta S}{2}}_{\ell^2 \to \ell^2} \;\leq\; 1 - c \min \h{\delta_-, (\im m)^2}
\end{equation}
for some small universal constant $c > 0$.
 For $n_0 \in \N$ we therefore have
\begin{align*}
\normbb{\frac{1}{1 - \zeta S}}_{\ell^\infty \to \ell^\infty} &\;\leq\; \sum_{n = 0}^{n_0 - 1} \normbb{\frac{1 + \zeta S}{2}}_{\ell^\infty \to \ell^\infty}^n + \sqrt{N} \sum_{n  = n_0}^\infty \normbb{\frac{1 + \zeta S}{2}}_{\ell^2 \to \ell^2}^n
\\
&\;\leq\; n_0 + \sqrt{N} \pb{1 - c \min \h{\delta_-, (\im m)^2}}^{n_0} \, \frac{C}{\min \h{\delta_-, (\im m)^2}}
\\
&\;\leq\; \frac{C \log N}{\min \h{\delta_-, (\im m)^2}}\,,
\end{align*}
where the last step follows by taking $n_0 = C_0 \log N / \min \h{\delta_-, (\im m)^2}$ for large enough $C_0 > 0$.
\end{proof}

\providecommand{\bysame}{\leavevmode\hbox to3em{\hrulefill}\thinspace}
\providecommand{\MR}{\relax\ifhmode\unskip\space\fi MR }
\providecommand{\MRhref}[2]{%
  \href{http://www.ams.org/mathscinet-getitem?mr=#1}{#2}
}
\providecommand{\href}[2]{#2}


\begin{thebibliography}{10}

\bibitem{EK2}
L.~Erd{\H{o}}s and A.~Knowles, \emph{Quantum diffusion and delocalization for
  band matrices with general distribution}, Ann.\ H.\ Poincar{\'e} \textbf{12}
  (2011), 1227--1319.

\bibitem{EK1}
\bysame, \emph{Quantum diffusion and eigenfunction delocalization in a random
  band matrix model}, Comm.\ Math.\ Phys.\ \textbf{303} (2011), 509--554.

\bibitem{EKYY3}
L.~Erd{\H{o}}s, A.~Knowles, H.T. Yau, and J.~Yin, \emph{Delocalization and
  diffusion profile for random band matrices}, Preprint arXiv:1205.5669.

\bibitem{EKYY4}
\bysame, \emph{The local semicircle law for a general class of random
  matrices}, Preprint arXiv:1212.0164.

\bibitem{EKYY1}
\bysame, \emph{Spectral statistics of {E}rd{\H{o}}s-{R}\'enyi graphs {I}: Local
  semicircle law}, to appear in Ann. Prob. Preprint arXiv:1103.1919.

\bibitem{EKYY2}
\bysame, \emph{Spectral statistics of {E}rd{\H{o}}s-{R}\'enyi graphs {II}:
  Eigenvalue spacing and the extreme eigenvalues}, to appear in Comm. Math.
  Phys. Preprint arXiv:1103.3869.

\bibitem{ESY6}
L.~Erd{\H{o}}s, S.~P{\'e}ch{\'e}, J.A. Ramirez, B.~Schlein, and H.T. Yau,
  \emph{Bulk universality for {W}igner matrices}, Comm.\ Pure Appl.\ Math.\
  \textbf{63} (2010), 895--925.

\bibitem{ESY7}
L.~Erd{\H{o}}s, J.~Ramirez, B.~Schlein, T.~Tao, V.~Vu, and H.T. Yau, \emph{Bulk
  universality for {W}igner hermitian matrices with subexponential decay},
  Math. Res. Lett. \textbf{17} (2010), 667--674.

\bibitem{ESY5}
L.~Erd{\H{o}}s, J.~Ramirez, B.~Schlein, and H.T. Yau, \emph{Universality of
  sine-kernel for {W}igner matrices with a small {G}aussian perturbation},
  Electr.\ J.\ Prob. \textbf{15} (2010), 526--604.

\bibitem{ESY2}
L.~Erd{\H{o}}s, B.~Schlein, and H.T. Yau, \emph{Local semicircle law and
  complete delocalization for {W}igner random matrices}, Comm.\ Math.\ Phys.\
  \textbf{287} (2009), 641--655.

\bibitem{ESY1}
\bysame, \emph{Semicircle law on short scales and delocalization of
  eigenvectors for {W}igner random matrices}, Ann.\ Prob. \textbf{37} (2009),
  815--852.

\bibitem{ESY3}
\bysame, \emph{{W}egner estimate and level repulsion for {W}igner random
  matrices}, Int. Math. Res. Not. \textbf{2010} (2009), 436--479.

\bibitem{ESY4}
\bysame, \emph{Universality of random matrices and local relaxation flow},
  Invent. Math. \textbf{185} (2011), no.~1, 75--119.

\bibitem{ESYY}
L.~Erd{\H{o}}s, B.~Schlein, H.T. Yau, and J.~Yin, \emph{The local relaxation
  flow approach to universality of the local statistics of random matrices},
  Ann.\ Inst.\ Henri Poincar{\'e} (B) \textbf{48} (2012), 1--46.

\bibitem{EYY1}
L.~Erd{\H{o}}s, H.T. Yau, and J.~Yin, \emph{Bulk universality for generalized
  {W}igner matrices}, Preprint arXiv:1001.3453.

\bibitem{EYY3}
\bysame, \emph{Rigidity of eigenvalues of generalized {W}igner matrices}, to
  appear in Adv. Math. Preprint arXiv:1007.4652.

\bibitem{EYY2}
\bysame, \emph{Universality for generalized {W}igner matrices with {B}ernoulli
  distribution}, J.\ Combinatorics \textbf{1} (2011), no.~2, 15--85.

\bibitem{PY}
N.S. Pillai and J.~Yin, \emph{Universality of covariance matrices}, Preprint
  arXiv:1110.2501.

\end{thebibliography}
\end{document}